\documentclass[10pt,reqno,a4]{amsart}

\usepackage[latin1]{inputenc}
\usepackage{amssymb,amsmath,amscd,amsfonts,color, amsthm}
\title[CLT for certain statistics of Gram random matrices]{A CLT for Information-theoretic statistics
of Gram random matrices with a given variance profile}

\author[Hachem et al.]{W. Hachem, P. Loubaton and J. Najim}
\date{\today}

\newtheorem{theo}{Theorem}[section]

\newtheorem{lemma}[theo]{Lemma}
\newtheorem{coro}[theo]{Corollary}

\newtheorem{prop}[theo]{Proposition}
\newtheorem{assump}{Assumption A-\hspace{-0.15cm}}
\newcommand{\R}{\mathbb{R}}
\newcommand{\C}{\mathbb{C}}
\newcommand{\E}{\mathbb{E}}
\newcommand{\bdm}{\begin{displaymath}}
\newcommand{\edm}{\end{displaymath}}
\newcommand{\bea}{\begin{eqnarray*}}
\newcommand{\eea}{\end{eqnarray*}}

\newcommand{\Cplus}{\mathbb{C}^+}
\newcommand{\Rplus}{\mathbb{R}^+}

\newcommand{\im}{\mathrm{Im}}
\newcommand{\bs}{\boldsymbol}
\newcommand{\ti}{\tilde}

\newcommand{\diag}{\mathrm{diag}}
\newcommand{\ii}{\mathbf{i}}
\numberwithin{equation}{section}
\theoremstyle{remark}
\newtheorem{rem}{Remark}[section]
\newcommand{\bv}{\boldsymbol{\varepsilon}}
\newcommand{\tr}{\mathrm{Tr}}
\newcommand{\cvgP}[1]{\xrightarrow[#1]{\mathcal P}}
\newcommand{\cvgD}{\xrightarrow[]{\mathcal D}}
\newcommand{\eqdef}{\stackrel{\triangle}{=}} 
\newcommand{\leftrownorm}{\left|\!\left|\!\left|} 
\newcommand{\rightrownorm}{\right|\!\right|\!\right|} 

\begin{document}
\bibliographystyle{plain}
\begin{abstract}
  Consider a $N\times n$ random matrix $Y_n=(Y_{ij}^{n})$ where the
  entries are given by 
$$
Y_{ij}^{n}=\frac{\sigma_{ij}(n)}{\sqrt{n}}
  X_{ij}^{n}\ ,
$$ 
the $X_{ij}^{n}$ being centered, independent and
  identically distributed random variables with unit
  variance and $(\sigma_{ij}(n); 1\le i\le N, 1\le j\le n)$ being an
  array of numbers we shall refer to as a variance profile. We study
  in this article the fluctuations of the random variable
$$
\log\det\left(Y_n Y_n^* + \rho I_N \right)
$$
where $Y^*$ is the Hermitian adjoint of $Y$ and $\rho > 0$ is an
additional parameter.  We prove that when centered and properly
rescaled, this random variable satisfies a Central Limit Theorem (CLT)
and has a Gaussian limit whose parameters are identified. A complete
description of the scaling parameter is given; in particular it is
shown that an additional term appears in this parameter in the case
where the 4$^\textrm{th}$ moment of the $X_{ij}$'s differs from the
4$^{\textrm{th}}$ moment of a Gaussian random variable. 
Such a CLT is of interest in the field of wireless communications. 
\end{abstract}

\maketitle

\section{Introduction} 
\subsection*{The model and the statistics} 
Consider a $N\times n$ random matrix $Y_n=(Y_{ij}^{n})$ whose entries are 
given by
\begin{equation}
\label{variance-profile-variable}
Y_{ij}^{n}=\frac{\sigma_{ij}(n)}{\sqrt{n}} X_{ij}^{n}\ ,
\end{equation}
where $(\sigma_{ij}(n),\ 1 \leq i \leq N,\ 1 \leq j \leq n)$  is
a uniformly bounded sequence of real numbers, and the random variables
$X_{ij}^{n}$ are complex, centered, independent and identically distributed
(i.i.d.) with unit variance and finite $8^{\mathrm{th}}$ moment. 
Consider the following linear statistics of the eigenvalues:
$$
{\mathcal I}_n(\rho) = \frac 1N \log \det \left( Y_n Y_n^* + \rho I_N \right)
=\frac 1N \sum_{i=1}^N \log (\lambda_i +\rho)
$$
where $I_N$ is the $N\times N$ identity matrix, $\rho > 0$ is a given
parameter and the $\lambda_i$'s are the eigenvalues of matrix $Y_n
Y_n^*$.  This functional known as the mutual information for multiple
antenna radio channels is very popular in wireless communication.
Understanding its fluctuations and in particular being able to
approximate its standard deviation is of major interest for various
purposes such as for instance the computation of the so-called outage
probability.

\subsection*{Presentation of the results}

The purpose of this article is to establish a Central Limit Theorem (CLT) for
${\mathcal I}_n(\rho)$ whenever $n\rightarrow \infty$ and 
$\frac Nn \rightarrow c\in (0,\infty)$. 

\subsubsection*{The centering procedure} 
It has been proved in Hachem {\em et al.} \cite{HLN07} that there
exists a sequence of deterministic probability measures $(\pi_n)$ such
that the mathematical expectation $\E {\mathcal I}_n(\rho)$ satisfies:
\[ 
\E {\mathcal I}_n(\rho) 
-\int \log(\lambda + \rho)\pi_n(\,d\lambda) 
\xrightarrow[n\rightarrow \infty]{} 0\ .
\] 
Moreover, $\int \log(\lambda + \rho)\pi_n(\,d\lambda)$ has a closed
form formula (see Section \ref{subsec-first-order}) and is easier to
compute\footnote{especially in the important case where the variance
  profile is separable, \emph{i.e.}, where $\sigma^2_{ij}(n)$ is
  written as $\sigma^2_{ij}(n)= d_i(n) \tilde{d}_j(n)$.}  than $\E
{\mathcal I}_n$ (whose evaluation would rely on massive Monte-Carlo
simulations).  For these reasons, we study in this article the
fluctuations of
$$
\frac 1N \log \det(Y_n Y_n^* + \rho I_N ) -  \int \log(\rho +t)\pi_n(\,dt)\ ,
$$
and prove that this quantity properly rescaled converges in distribution 
toward a Gaussian random variable. 
Although phrased differently, such a centering procedure relying on a 
deterministic equivalent is used in \cite{AndZei06} and \cite{BaiSil04}.

In order to prove the CLT, we study separately the quantity
$N({\mathcal I}_n(\rho) - \mathbb{E} {\mathcal I}_n(\rho))$ from which
the fluctuations arise and the quantity $N(\mathbb{E} {\mathcal
  I}_n(\rho) - \int \log(\lambda + \rho)\pi_n(\,d\lambda) )$ which
yields a bias.

\subsubsection*{The fluctuations} We shall prove in this paper that
the variance $\Theta^2_n$ of $N({\mathcal I}_n(\rho) - \mathbb{E}
{\mathcal I}_n(\rho))$ takes a remarkably simple closed-form
expression. In fact, there exists a $n\times n$ deterministic matrix
$A_n$ (described in Theorem \ref{th-variance}) whose entries depend on
the variance profile $(\sigma_{ij})$ such that the variance takes the
form:
$$
\Theta_n^2 =\log\det (I_n -A_n) +\kappa \tr A_n,
$$
where $\kappa= \mathbb{E} |X_{11}|^4 -2$ in the fourth cumulant of the
complex variable $X_{11}$ and the CLT expresses as:
$$
\frac N{\Theta_n}\left( {\mathcal I}_n -\mathbb{E}{\mathcal I}_n
\right) \xrightarrow[n\rightarrow \infty]{\mathcal L} {\mathcal
  N}(0,1).
$$
In the case where $\kappa=0$ (which happens if $X_{ij}$ is a complex
gaussian random variable for instance), the variance has the log-form
$\Theta_n^2 =\log\det (I_n -A_n)$. This has already been noticed for
different models in the engineering literature by Moustakas {\em et
  al.} \cite{MSS03}, Taricco \cite{Tar06}. See also Hachem {\em et
  al.} in \cite{HKLNP06pre}.

\subsubsection*{The bias} It is proved in this paper that there exists
a deterministic quantity ${\mathcal B}_n(\rho)$ (described
in Theorem \ref{th-bias}) such that:
$$
N\left(\mathbb{E} {\mathcal
  I}_n(\rho) - \int \log(\lambda + \rho)\pi_n(\,d\lambda) \right) - {\mathcal B}_n(\rho) \xrightarrow[n\rightarrow \infty]{} 0\ . 
$$
If $\kappa=0$, then ${\mathcal B}_n(\rho)=0$
and there is no bias in the CLT.

\subsection*{About the literature}
Central limit theorems have been widely studied for various models of
random matrices and for various classes of linear statistics of the
eigenvalues in the physics, engineering and  mathematical literature.

In the mathematical literature, CLTs for Wigner matrices 
can be traced back to Girko \cite{Gir75} (see also \cite{Gir03}). 
Results for this class of 
matrices have also been obtained by Khorunzhy {\em et al.} \cite{KKP96}, Johansson
\cite{Joh98}, Sinai and Sochnikov \cite{SinSos98}, Soshnikov
\cite{Sos00}, Cabanal-Duvillard \cite{Cab01}. For band matrices, let
us mention the paper by Khorunzhy {\em et al.} \cite{KKP96},
Boutet de Monvel and Khorunzhy \cite{BouKho98}, Guionnet \cite{Gui02},
Anderson and Zeitouni \cite{AndZei06}. The case of Gram matrices has
been studied in Jonsson \cite{Jon82} and Bai and Silverstein
\cite{BaiSil04}.  For a more detailed overview, the reader is referred
to the introduction in \cite{AndZei06}.  In the physics literature,
so-called replica methods as well as saddle-point methods have long
been a popular tool to compute the moments of the limiting
distributions related to the fluctuations of the statistics of the
eigenvalues.

Previous results and methods have recently been exploited in the
engineering literature, with the growing interest in random matrix
models for wireless communications (see the seminal paper by Telatar
\cite{Tel99} and the subsequent papers of Tse and co-workers
\cite{TseZei00}, \cite{TseHan99}; see also the monograph by Tulino and
Verdu \cite{TulVer04} and the references therein). One main interest
lies in the study of the convergence and the fluctuations of the
mutual information $\frac 1N \log \det \left( Y_n Y_n^* + \rho
  I_N\right)$ for various models of matrices $Y_n$. General
convergence results have been established by the authors in
\cite{HLN07,HLN05,HLN06} while fluctuation results based on Bai
and Silverstein \cite{BaiSil04} have been developed in Debbah and
M\"uller \cite{DebMul03} and Tulino and Verdu \cite{TulVer05}. Other
fluctuation results either based on the replica method or on
saddle-point analysis have been developed by Moustakas, Sengupta and
coauthors \cite{MSS03,SenMit00pre}, Taricco \cite{Tar06}. In a
different fashion and extensively based on the Gaussianity of the
entries, a CLT has been proved in Hachem {\em et al.}
\cite{HKLNP06pre}.

\subsubsection*{Comparison with existing work}
There are many overlaps between this work and other works in the
literature, in particular with the paper by Bai and Silverstein
\cite{BaiSil04} and the paper by Anderson and Zeitouni \cite{AndZei06}
(although this last paper is primarily devoted to band matrix models,
{\em i.e.} symmetric matrices with a symmetric variance profile). The
computation of the variance and the obtention of a closed-form formula
significantly extend the results obtained in \cite{HKLNP06pre}.

In this paper, we deal with complex variables which are more relevant
for wireless communication applications. The case of real random
variables would have led to very similar computation, the cumulant
$\kappa= \mathbb{E}|X|^4-2$ being replaced by $\tilde \kappa =
\mathbb{E} X^4 -3$. Due to the complex nature of the variables, the
CLT in \cite{AndZei06} does not apply directly. Moreover, we
substantially relax the moment assumptions concerning the entries with
respect to \cite{AndZei06} where the existence of moments of all order
is required.  In fact, we shall only assume the finiteness of the
$8^{\mathrm{th}}$ moment.  Bai and Silverstein \cite{BaiSil04}
consider the model $ T_n^{\frac 12} X_n X_n^* T_n^{\frac 12}$ where
the entries of $X_n$ are i.i.d. and have gaussian fourth moment.
%
This assumption can be skipped in our framework, where a good
understanding of the behaviour of the diagonal individual entries of
the resolvent $(-zI_n +Y_n Y_n^*)^{-1}$ enables us to deal with
non-gaussian entries.

On the other hand, it must be noticed that we establish the CLT for
the single functional $\log\det(Y_n Y_n^* + \rho I_N)$ and do not
provide results for a general class of functionals as in
\cite{AndZei06} and \cite{BaiSil04}. We do believe however that all
the computations performed in this article are a good starting point
to address this issue.

\subsection*{Outline of the article}

\subsubsection*{Non-asymptotic vs asymptotic results} 
As one may check in Theorems \ref{th-variance}, \ref{th-clt} and
\ref{th-bias}, we have deliberately chosen to provide non-asymptotic
({\em i.e.} depending on $n$) deterministic formulas for the variance
and the bias that appear in the fluctuations of ${\mathcal
  I}_n(\rho)$. This approach has at least two virtues: Non-asymptotic
formulas exist for very general variance profiles $(\sigma_{ij}(n))$
and provide a natural discretization which can easily be implemented.
In the case where the variance profile is the sampling of some
continuous function , {\em i.e.}  $\sigma_{ij}(n)=\sigma(i/N, j/n)$
(we shall refer to this as the existence of a limiting variance
profile), the deterministic formulas converge as $n$ goes to infinity
(see Section \ref{section-limiting}) and one has to consider Fredholm
determinants in order to express the results.

\subsubsection*{The general approach} 
The approach developed in this article is conceptually simple. The
quantity ${\mathcal I}_n(\rho) - \E {\mathcal I}_n(\rho)$ is
decomposed into a sum of martingale differences; we then
systematically approximate random quantities such as quadratic forms
$\mathbf{x}^T A \mathbf{x}$ where $\mathbf{x}$ is some random vector
and $A$ is some deterministic matrix, by their deterministic
counterparts $\frac 1n \mathrm{Trace}\,A$ (in the case where
the entries of $\mathbf{x}$ are i.i.d. with variance $\frac 1n$) as
the size of the vectors and the matrices goes to infinity.  A careful
study of the deterministic quantities that arise, mainly based on
(deterministic) matrix analysis is carried out and yields the
closed-form variance formula.  The martingale method which is used to
establish the fluctuations of ${\mathcal I}_n(\rho)$ can be traced
back to Girko's {\sc REFORM} ({\sc RE}solvent, {\sc FOR}mula and {\sc
  M}artingale) method (see \cite{Gir75,Gir03}) and is close to the one
developed in \cite{BaiSil04}.

\subsubsection*{Contents}
In Section \ref{sec-notations-assumptions}, we introduce the main
notations, we provide the main assumptions and we recall all the first
order results (deterministic approximation of $\mathbb{E} {\mathcal
  I}_n(\rho)$) needed in the expression of the CLT.  In Section
\ref{sec-results}, we state the main results of the paper: Definition
of the variance $\Theta_n^2$ (Theorem \ref{th-variance}), asymptotic
behaviour (fluctuations) of $N \left( {\mathcal I}_n(\rho) - \E
  {\mathcal I}_n(\rho) \right)$ (Theorem \ref{th-clt}), asymptotic
behaviour (bias) of $N \left( \E {\mathcal I}_n(\rho) - \int
  \log(\rho+t)\pi_n(dt) \right)$ (Theorem \ref{th-bias}).  Section
\ref{proof-variance} is devoted to the proof of Theorem
\ref{th-variance}, Section \ref{proof-clt}, to the proof of Theorem
\ref{th-clt} and Section \ref{sec-proof-bias}, to the proof of Theorem
\ref{th-bias}.

\subsection*{Acknowlegment}
This work was partially supported by the Fonds National de la Science (France)
via the ACI program ``Nouvelles Interfaces des Math\'ematiques'', project MALCOM n$^\circ$ 205.

\section{Notations, assumptions and first order results}
\label{sec-notations-assumptions}
\subsection{Notations and assumptions}
Let $N=N(n)$ be a sequence of integers such that 
$$
\lim_{n\rightarrow \infty} \frac{N(n)}{n} =c\in (0,\infty)\ .
$$
In the sequel, we shall consider a $N\times n$ random matrix $Y_n$
with individual entries:
$$
Y_{ij}^{n}=\frac{\sigma_{ij}(n)}{\sqrt{n}} X_{ij}^{n}\ ,
$$
where $X_{ij}^n$ are complex centered i.i.d random variables with unit variance and 
$(\sigma_{ij}(n);\ 1\le i\le N, 1\le j\le n)$ is a triangular array of real numbers. 
Denote by $\mathrm{var}(Z)$ the variance of the random variable $Z$. 
Since $\mathrm{var}(Y_{ij}^n)=\sigma_{ij}^2(n)/n$,
the family $(\sigma_{ij}(n))$ will be referred to as a variance profile. \\ 

\subsubsection*{The main assumptions}
\begin{assump}
\label{hypo-moments-X}
The random variables $(X_{ij}^n\ ;\ 1\le i\le N,\,1\le j\le n\,,\, n\ge1)$ 
are complex, independent and identically
distributed. They satisfy 
$$
\E X_{ij}^n = \E(X_{ij}^n)^2=0, \quad
\E|X_{ij}^n|^2=1 \quad \mathrm{and} \quad 
\E|X_{ij}^n|^{8}<\infty \ . 
$$
\end{assump}

\begin{assump}
\label{hypo-variance-field}
There exists a finite positive real number $\sigma_{\max}$ such that the 
family of real numbers 
$(\sigma_{ij}(n),\ 1\le i\le N,\ 1\le j\le n,\  n\ge 1)$ satisfies:
$$
\sup_{n\ge 1}\  
\max_{\genfrac{}{}{0pt}{}{1\le i\le N}{1\le j\le n} } 
|\sigma_{ij}(n)| \leq \sigma_{\max} \ .
$$
\end{assump}

\begin{assump}
\label{hypo-inf-trace-Dj}
There exists a real number $\sigma_{\min}^2 > 0$ such that
$$
\liminf_{n \geq 1} \min_{1\le j\le n}
\frac{1}{n} \sum_{i=1}^N \sigma^2_{ij}(n)
\geq \sigma_{\min}^2 \ . 
$$
\end{assump}
Sometimes we shall assume that the variance profile is obtained by sampling
a function on the unit square of $\R^2$. 
This helps to get limiting expressions and limiting behaviours 
(cf. Theorem \ref{prop-limit-var-profile-existence}): 

\begin{assump}
\label{hypo-limit-variance-field}
There exists a continuous function $\sigma^2 : [0,1] \times [0,1] \rightarrow
(0,\infty)$ such that $\sigma_{ij}^2(n) = \sigma^2(i/N, j/n)$.  
\end{assump}

\subsubsection*{Remarks related to the assumptions}
\begin{enumerate}
\item Using truncation arguments \`a la Bai and Silverstein
  \cite{BaiSil98,Sil95,SilBai95}, one may lower the moment assumption
  related to the $X_{ij}$'s in {\bf A-\ref{hypo-moments-X}}.
\item Obviously, assumption {\bf A-\ref{hypo-inf-trace-Dj}} holds if $\sigma_{ij}^2$ is uniformly
lower bounded by some nonnegative quantity. 
\item Obviously, assumption {\bf A-\ref{hypo-limit-variance-field}} implies both {\bf
    A-\ref{hypo-variance-field}} and {\bf A-\ref{hypo-inf-trace-Dj}}. When {\bf A-\ref{hypo-limit-variance-field}} holds, we shall say that
there exists a limiting variance profile.
\item If necessary, assumption {\bf A-\ref{hypo-inf-trace-Dj}} can be 
slightly improved by stating:
\[
\max\left( 
 \liminf_{n\ge 1} \min_{1\le j\le n} \frac{1}{n} \sum_{i=1}^N \sigma^2_{ij}(n) \ , \ 
 \liminf_{n\ge 1} \min_{1\le i\le N} \frac{1}{n} \sum_{j=1}^n \sigma^2_{ij}(n) 
\right) > 0 \ . 
\]
In the case where the first liminf is zero, one may notice that 
$\log\det(Y_n Y_n^* + \rho I_N ) = \log\det(Y^*_n Y_n + \rho I_n) + (n-N)\log\rho$ and 
consider $Y_n^* Y_n$ instead.
\end{enumerate}

\subsubsection*{Notations}
The indicator function of the set ${\mathcal A}$ will be denoted by
${\bf 1}_{\mathcal A}(x)$, its cardinality by $\#{\mathcal A}$. 
As usual, $\Rplus = \{ x \in \R \ : \ x \geq 0 \}$ and 
$\Cplus = \{ z \in \C \ : \ \im(z) > 0 \}$.

We denote by $\cvgP{ }$ the convergence in probability of random variables and 
by $\cvgD$ the convergence in distribution of probability measures.  

Denote by $\mathrm{diag}(a_i;\,1\le i\le k)$ the $k\times k$ diagonal
matrix whose diagonal entries are the $a_i$'s.  Element $(i,j)$ of
matrix $M$ will be either denoted $m_{ij}$ or $[M]_{ij}$ depending on
the notational context.  Denote by $M^T$ the matrix transpose of $M$,
by $M^*$ its Hermitian adjoint, by $\tr (M)$ its trace and $\det(M)$ its
determinant (if $M$ is square), and by $F^{M\,M^*}$, the empirical
distribution function of the eigenvalues of $M\,M^*$, \emph{i.e.}
\[
F^{M\,M^*}(x) = \frac{1}{N} \#\{ i : \lambda_i \leq x \} \ ,
\]
where $M\,M^*$ has dimensions
$N \times N$ and the $\lambda_i$'s are the eigenvalues of $M\,M^*$.

When dealing with vectors, $\|\cdot\|$ will refer to the Euclidean 
norm, and $\| \cdot \|_\infty$, to the max (or $\ell_\infty$) norm. 
In the case of matrices, $\|\cdot\|$ will refer to the 
spectral norm and $\leftrownorm \cdot \rightrownorm_\infty$ to
the maximum row sum norm (referred to as the max-row norm), \emph{i.e.}, 
$\leftrownorm M \rightrownorm_\infty = \max_{1\le i \le N} \sum_{j=1}^N
| [ M ]_{ij} |$ when $M$ is a $N \times N$ matrix. 
We shall denote by $r(M)$ the spectral radius of matrix $M$.

When no confusion can occur, we shall often drop subscripts and superscripts
$n$ for readability. 

\subsection{Stieltjes Transforms and Resolvents}
\label{subsec-ST-resolvents} 
In this paper, Stieltjes transforms of probability measures
play a fundamental role. 
Let $\nu$ be a bounded non-negative measure over $\mathbb{R}$. Its Stieltjes 
transform $f$ is defined as: 
$$
f(z)=\int_{\mathbb{R}} \frac{\nu(d\lambda)}{\lambda-z},\quad 
z \in \C \setminus \mathrm{supp}(\nu)\ , 
$$
where $\mathrm{supp}(\nu)$ is the support of the measure $\nu$. 
We shall denote by ${\mathcal S}(\Rplus)$ the set of Stieltjes transforms
of probability measures with support in $\Rplus$. 

We list in the following proposition the main properties of the Stieltjes 
transforms that will be needed in the paper: 
\begin{prop} The following properties hold true.
\label{prop-properties-ST} 
\begin{enumerate}
\item Let $f$ be the Stieltjes transform of a probability measure $\nu$ on $\R$, then:
\begin{itemize}
\item[-] The function $f$ is analytic over $\C \setminus \mathrm{supp}(\nu)$.
\item[-] If $f(z) \in {\mathcal S}(\Rplus)$,
then $| f(z)| \le ({\bf d}(z, \Rplus))^{-1}$ where ${\bf d}(z, \Rplus)$ denotes the
distance from $z$ to $\Rplus$.\\ 
\end{itemize}
\item 
\label{prop-properties-ST-convergence} 
Let $\mathbb{P}_n$ and $\mathbb{P}$ be probability measures over 
$\mathbb{R}$ and denote by $f_n$ and $f$
their Stieltjes transforms. Then 
$$
\left( \forall z\in\Cplus,\ f_n(z) \xrightarrow[n\rightarrow\infty]{} f(z)\right) \quad \Rightarrow \quad 
\mathbb{P}_n\xrightarrow[n\rightarrow\infty]{\mathcal D} \mathbb{P}.  
$$
\end{enumerate}
\end{prop}
There are very close ties between the Stieltjes transform of the
empirical distribution of the eigenvalues of a matrix and the
resolvent of this matrix.  Let $M$ be a $N \times n$ matrix. The
resolvent of $M M^*$ is defined as:
$$
Q(z)=(M M^* -z\,I_N)^{-1}=\left(q_{ij}(z)\right)_{1\le i,j,\le N},
\quad z\in \C - \Rplus\ . 
$$
The following properties are straightforward.
\begin{prop}
\label{prop-resolvent-properties} 
Let $Q(z)$ be the resolvent of $M M^*$, then:
\begin{enumerate}
\item The function $h_n(z)=\frac1N \tr\ Q(z)$ is the Stieltjes 
transform of the empirical distribution of the eigenvalues of $M M^*$. Since the 
eigenvalues of this matrix are non-negative, 
$h_n(z) \in {\mathcal S}(\Rplus)$.
\item For every $z\in\C-\Rplus$, $ \|Q(z)\|\le \left( {\bf d}(z, \Rplus) 
\right)^{-1}$. In particular, if $\rho>0$, $ \|Q(-\rho)\|\le \rho^{-1}$.
 \\ 
\end{enumerate}
\end{prop}

\subsection{First Order Results: A primer} 
\label{subsec-first-order}
Recall that ${\mathcal I}_n(\rho)=\frac 1N \log \det (Y_n Y_n^* +\rho I)$ and let $\rho>0$.
We remind below some results related to the asymptotic behaviour
of $\E {\mathcal I}_n(\rho)$. 
As 
$$
{\mathcal I}_n(\rho) = 
\frac 1N \sum_{i=1}^N \log\left( \lambda_i + \rho \right)
= 
\int_0^\infty \log(\lambda + \rho) \, dF^{Y_n Y_n^*}(\lambda) \ , 
$$
where the $\lambda_i$'s are the eigenvalues of $Y Y^*$, 
the approximation of $\E {\mathcal I}_n(\rho)$ is closely related to the  
``first order'' 
approximation of $F^{Y_n Y_n^*}$ as $n \to \infty$ and $N / n \to c > 0$. 

The following theorem summarizes the first order results needed
in the sequel. It is a direct consequence of \cite[Sections 2
and 4]{HLN07} (see also \cite{Gir01a}):
\begin{theo}[\cite{HLN07}, \cite{Gir01a}]\label{first-order}
  Consider the family of random matrices  $(Y_n Y_n^*)$ and assume
  that {\bf A-\ref{hypo-moments-X}} and {\bf A-\ref{hypo-variance-field}} hold. Then, the following hold true:
\begin{enumerate}
\item The system of $N$ functional 
equations: 
\begin{equation}
\label{eq-systeme-approx-deter} 
t_{i}(z) =  
\frac{1}
{-z  + \frac 1n \sum_{j=1}^n 
\frac{\sigma^2_{ij}(n)}
{1 + \frac{1}{n} \sum_{\ell=1}^N \sigma^2_{\ell j}(n) t_{\ell}(z)} }
\end{equation} 
admits a unique solution $\left(t_{1}(z),\cdots,t_{N}(z) \right)$ in
${\mathcal S}(\Rplus)^{N}$. In particular, $m_n(z)=\frac 1N
\sum_{i=1}^N t_i(z)$ belongs to ${\mathcal S}(\Rplus)$ and there
exists a probability measure $\pi_n$ on $\Rplus$ such that:
$$ 
m_n(z)=\int_{0}^\infty \frac{\pi_n(d\lambda)}{\lambda-z} \ .
$$

\item For every continuous and bounded function $g$ on $\mathbb{R}^+$,
$$
\int_{\mathbb{R}^+} g(\lambda) \, dF^{Y_n Y_n^*}(\lambda) -
\int_{\mathbb{R}^+} g(\lambda) \, \pi_n(d\lambda)
\xrightarrow[n\rightarrow \infty]{}0 \quad \textrm{a.e.}
$$ 
\item The function $V_n(\rho) = 
\int_{\mathbb{R}^+} \log(\lambda + \rho)\pi_n(d\lambda)$ is finite for every 
$\rho > 0$ and
$$
\E {\mathcal I}_n(\rho) - V_n(\rho) \xrightarrow[n\rightarrow
\infty]{}0 \qquad \textrm{where}\quad {\mathcal I}_n(\rho) = \frac 1N \log
\det \left( Y_n Y_n^* + \rho I_N \right)\ .
$$
Moreover, $V_n(\rho)$ 
admits the following closed form formula:
\begin{multline*}
V_n(\rho) = 
- \frac 1N \sum_{i=1}^N \log t_i(-\rho)  
+ \frac 1N \sum_{j=1}^n 
\log\left( 1 + \frac 1n \sum_{\ell=1}^N \sigma^2_{\ell j}(n) t_{\ell}(-\rho) \right) 
\\ 
- \frac{1}{Nn} \sum_{i=1:N,j=1:n} 
\frac{\sigma^2_{ij}(n) t_i(-\rho)}
{1 + \frac 1n \sum_{\ell=1}^N \sigma^2_{\ell j}(n) t_{\ell}(-\rho)} \ .
\end{multline*}
where the $t_i$'s are defined above.
\end{enumerate}
\end{theo}

Theorem \ref{first-order} partly follows from the following lemma which will
be often invoked later on and whose statement emphasizes the
symmetry between the study of $Y_n Y_n^*$ and $Y_n^* Y_n$.  Denote by
$Q_n(z)$ and $\ti Q_n(z)$ the resolvents of $Y_n Y_n^*$ and $Y_n^*
Y_n$, i.e.
\begin{eqnarray*} 
Q_n(z) &=& (Y_n Y_n^* -zI_N)^{-1}= \left( q_{ij}(z)\right)_{1\le i,j\le N},
\quad z\in \C - \Rplus \\ 
\ti{Q}_n(z) &=&  (Y_n^* Y_n -zI_n)^{-1}= 
\left(\ti{q}_{ij}(z)\right)_{1\le i,j\le n},\quad z\in \C - \Rplus \ . 
\end{eqnarray*} 

\begin{lemma}
  \label{th-deter-approx-details} Consider the family of random
  matrices $(Y_n Y_n^*)$ and assume that {\bf A-\ref{hypo-moments-X}} and
  {\bf A-\ref{hypo-variance-field}} hold. Consider the following system of
  $N+n$ equations:
\begin{equation}
\left\{
\begin{array}{ll}
t_{i,n}(z)&= \frac{-1}{z\left( 1 +\frac 1n \tr \ti D_{i,n} \ti T_n(z)
\right)}
\quad \textrm{for}\quad 1\le i\le N \nonumber \\
\ti t_{j,n}(z)&= \frac{-1}{z\left( 1 +\frac 1n \tr D_{j,n} T_n(z)
\right)}\quad \textrm{for}\quad 1\le j\le n
\end{array}\right. 
\label{eq-approx-determinist-complete} 
\end{equation}
where 
\begin{eqnarray*}
T_n(z)&=\diag(t_{i,n}(z),\, 1\le i\le N), 
\quad \ti T_n(z)&=\diag(\ti t_{j,n}(z),\, 1\le j\le n)\ ,\\
D_{j,n}&=\diag(\sigma^2_{ij}(n),\, 1\le i\le N),\quad \ti 
D_{i,n} &=\diag(\sigma^2_{ij}(n),\, 1\le j\le n)\ .
\end{eqnarray*}
Then the following holds true: 
\begin{itemize}
\item[(a)] \cite[Theorem 2.4]{HLN07} This system admits a unique solution
$$
(t_{1,n},\cdots,t_{N,n},\ti t_{1,n},\cdots \ti t_{n,n}) 
\in {\mathcal S}(\mathbb{R}^+)^{N+n}
$$ 
\item[(b)] \cite[Lemmas 6.1 and 6.6]{HLN07}
For every sequence $U_n$ of $N \times N$ diagonal matrices 
and every sequence $\widetilde{U}_n$ of $n \times n$ diagonal matrices
such as $\sup_n \max\left( \| U_n \|,  \| \ti{U}_n \| \right) < \infty \, $,
the following limits hold true almost surely:
\[
\begin{array}{lcll}
\displaystyle{\lim_{n\to\infty, N/n \to c} 
\frac 1N \tr\left( U_n \left( Q_n(z) - T_n(z) \right) \right)} &=& 0 & 
\forall z \in \C - \Rplus, \\
\displaystyle{\lim_{n\to\infty, N/n \to c} 
\frac 1n \tr\left( \ti U_n \left( \ti Q_n(z) - \ti T_n(z) \right) \right)}  
&=& 0 & \forall z \in \C - \Rplus \ . 
\end{array} 
\]
\end{itemize}
\end{lemma}

In the case where there exists a limiting variance profile, the results
can be expressed in the following manner:
\begin{theo}
[\cite{BKV96}, \cite{Gir90}, \cite{HLN06} ] 
\label{prop-limit-var-profile-existence}
Consider the family of random matrices $(Y_n Y_n^*)$ and assume that
{\bf A-\ref{hypo-moments-X}} and {\bf A-\ref{hypo-limit-variance-field}}
hold. Then: 
\begin{enumerate}
\item The functional equation
\begin{equation}
\label{eq-func-limit-var-profile}
\tau(u,z)=\left( -z +\int_0^1 \frac{\sigma^2(u,v)}{1+c\int_0^1 \sigma^2(x,v)\tau(x,z)\,dx}dv \right)^{-1}
\end{equation}
admits a unique solution among the class of functions 
$\Phi:[0,1]\times \mathbb{C}\setminus \mathbb{R} \rightarrow \mathbb{C}$
such that  $u\mapsto \Phi(u,z) $ is continuous over $[0,1]$ and  
$z\mapsto \Phi(u,z)$ belongs to ${\mathcal S}(\Rplus)$.\\
\item The function $f(z)=\int_0^1 \tau(u,z)\,du$ where $\tau(u,z)$ is
  defined above is the Stieltjes transform of a probability measure
  $\mathbb{P}$. Moreover, we have
$$
F^{Y_n Y_n^*} \xrightarrow[n\rightarrow \infty]{\mathcal D} \mathbb{P}\ a.s. 
$$
\end{enumerate}
\end{theo}

\begin{rem} \label{tau-tilde} If one is interested in the Stieltjes
  function related to the limit of $F^{Y^*_n Y_n}$, then one must introduce 
the following function $\tilde \tau$, which is the counterpart of $\tau$:
$$
\tilde \tau(v,z)=\left( -z + c \int_0^1 \frac{\sigma^2(t,v)}{1+\int_0^1 \sigma^2(t,s)\tilde \tau(s,z)\,ds}dt \right)^{-1}\ .
$$
Functions $\tau$ and $\tilde \tau$ are related via the following equations:
\begin{equation}\label{coupled-eq}
\tau(u,z) =\frac{-1}{z\left( 1+\int_0^1 \sigma^2(u,v) \tilde \tau(v,z) \,dv\right)}
\quad \textrm{and}\quad 
\tilde \tau(v,z) =\frac{-1}{z\left( 1 + c\int_0^1 \sigma^2(t,v)\tau(t,z)\,dt \right)}\ .
\end{equation}
\end{rem}

\begin{rem}
  We briefly indicate here how Theorems \ref{first-order} and
  \ref{prop-limit-var-profile-existence} above can be deduced from
  Lemma \ref{th-deter-approx-details}. As $\frac 1N \tr Q_n(z)$ is the
  Stieltjes transform of $F^{Y_n Y_n^*}$, Theorem
  \ref{th-deter-approx-details}--$(b)$ with $U_n = I_N$ yields $\frac
  1N \tr Q_n(z)- \frac 1N \tr T_n(z) \to 0$ almost surely. When a
  limit variance profile exists as described by 
    {\bf A-\ref{hypo-limit-variance-field}}, one can easily show 
  that $\frac 1N \tr T_n(z)$ converges to the Stieltjes
  transform $f(z)$ given by Theorem
  \ref{prop-limit-var-profile-existence} (Equation
  (\ref{eq-func-limit-var-profile}) is the ``continuous equivalent''
  of Equations (\ref{eq-systeme-approx-deter})).  Thanks to
  Proposition
  \ref{prop-properties-ST}--(\ref{prop-properties-ST-convergence}), we
  then obtain the almost sure weak convergence of $F^{Y_n Y_n^*}$ to
  $F$. In the case where {\bf A-\ref{hypo-limit-variance-field}} is not
  satisfied, one can prove similarly that $F^{Y_n Y_n^*}$ is
  approximated by $\pi_n$ as stated in Theorem \ref{first-order}-(2).
\end{rem}
\section{The Central Limit Theorem for ${\mathcal I}_n(\rho)$}
\label{sec-results}

When given a variance profile, one can consider the $t_i$'s defined in
Theorem \ref{first-order}-(1). Recall that 
$$T(z) = \diag(t_i(z), 1\le i \le N)\quad  \textrm{and}\quad  
D_j = \diag(\sigma^2_{ij}, 1 \le i \le N)\ . 
$$
We shall first define in Theorem \ref{th-variance} a non-negative real number that will play the
role of the variance in the CLT. We then state the CLT in Theorem
\ref{th-clt}. Theorem \ref{th-bias} deals with the bias term
$N(\E{\mathcal I} - V)$. 
 
\begin{theo}[Definition of the variance]
\label{th-variance}
Consider a variance profile $(\sigma_{ij})$ which fulfills assumptions 
{\bf A-\ref{hypo-variance-field}} and {\bf A-\ref{hypo-inf-trace-Dj}}
and the related $t_i$'s defined in Theorem \ref{first-order}-(1). Let $\rho>0$.  
\begin{enumerate}
\item Let $A_n=(a_{\ell, m})$ be the matrix defined by:
$$
a_{\ell,m} =\frac 1n \frac{\frac{1}{n} \tr D_{\ell} D_{m}
  T(-\rho)^2}{\left(1
    + \frac{1}{n} \tr D_{\ell} T(-\rho) \right)^2} \ ,\ 1\le \ell, m\le n\ ,
$$
then the quantity 
$
{\mathcal V}_n = - \log\det(I_n -A_n)
$
is well-defined. \\

\item Denote by 
$
{\mathcal W}_n = \tr A_n 
$
and let $\kappa$ be a real number\footnote{In the sequel, 
$\kappa$ is defined as $\kappa=\E|X_{11}|^4 -2$.} satisfying $\kappa \geq -1$.
The sequence $\left({\mathcal V}_n + \kappa {\mathcal W}_n\right)$
satisfies 
$$
0\quad  < \quad \liminf_n \left({\mathcal V}_n + \kappa {\mathcal W}_n \right)\quad \leq\quad  \limsup_n 
\left( {\mathcal V}_n + \kappa {\mathcal W}_n \right)\quad 
<\quad  \infty
$$ as $n \to \infty$ and $N/n \to c > 0$. We shall denote by:
$$
\Theta^2_n \stackrel{\triangle}{=} -\log \det(I-A_n) + \kappa \tr A_n\ .
$$ 
\end{enumerate}
\end{theo}

Proof of Theorem \ref{th-variance} is postponed to Section \ref{proof-variance}.

In the sequel and for obvious reasons, we shall refer to matrix $A_n$
as the {\bf variance matrix}. In order to study the CLT for
$N({\mathcal I}_n(\rho)-V_n(\rho))$, we decompose it into a random
term from which the fluctuations arise:
$$
\displaystyle{N \left( {\mathcal I}_n(\rho) - \E {\mathcal I}_n(\rho) \right)} 
=  
\displaystyle{\log \det(Y_n Y_n^* + \rho I_N) -
\E\log \det(Y_n Y_n^* + \rho I_N )} \ ,
$$ 
and into a deterministic one which yields to a bias in the CLT:
$$
\displaystyle{N \left( \E {\mathcal I}_n(\rho) - V_n(\rho) \right)} 
=  
\displaystyle{\E\log \det(Y_n Y_n^* + \rho I_N) - 
N \int \log(\lambda+\rho)\pi_n(\,d\lambda)} \ .
$$
We can now state the CLT.
\begin{theo}[The CLT]
\label{th-clt}
Consider the family of random matrices $(Y_n Y_n^*)$ and assume that
{\bf A-\ref{hypo-moments-X}}, {\bf A-\ref{hypo-variance-field}} and
{\bf A-\ref{hypo-inf-trace-Dj}} hold true. Let $\rho>0$, let
$\kappa=\E|X_{11}|^4 -2$, and let $\Theta_n^2$ be given by Theorem
\ref{th-variance}. Then
$$
\Theta_n^{-1} 
\Big(\log \det(Y_n Y_n^* + \rho I_N) - \E\log \det(Y_n
  Y_n^* + \rho I_N ) \Big) \quad \xrightarrow[n\rightarrow \infty,\
\frac Nn \rightarrow c]{\mathcal D}\quad {\mathcal N}(0,1)\ . 
$$
\end{theo}
Proof of Theorem \ref{th-clt} is postponed to Section \ref{proof-clt}. 

\begin{rem}
  In the case where the entries $X_{ij}$ are complex Gaussian ({\em
    i.e.} with independent normal real and imaginary parts, each of
  them centered with variance $2^{-1}$) then $\kappa=0$ and
  $\Theta_n^2$ reduces to the term ${\mathcal V}_n$. 
\end{rem}
The asymptotic bias is described in the following theorem: 
\begin{theo}[The bias]
\label{th-bias} 
Assume that the setting of Theorem \ref{th-clt} holds true. Then
\begin{enumerate}
\item
\label{th-bias-unique-solution} 
For every $\omega \in [\rho, +\infty)$, 
the system of $n$ linear equations with unknown parameters 
$({\bs w}_{\ell,n}(\omega); 1 \leq \ell \leq n)$: 
\begin{equation} 
\label{eq-def-w-bias} 
{\bs w}_{\ell,n}(\omega) = 
\frac 1n \sum_{m=1}^n 
\frac{ \frac 1n \tr D_{\ell} D_m T(-\omega)^2 }
{(1 + \frac 1n \tr D_{\ell} T(-\omega))^2} 
{\bs w}_{m,n}(\omega) + {\bs p}_{\ell,n}(\omega), \quad 1\le \ell \le n  
\end{equation} 
with
\begin{equation}
\label{eq-p(omega)} 
{\bs p}_{\ell,n}(\omega) = \kappa \ \omega^2 \ti t_\ell(-\omega)^2 
\left( 
\frac{\omega}n \sum_{i=1}^N \left( \frac{\sigma^2_{i\ell} t_i(-\omega)^3 }n
\tr \ti D_i^2 \ti T(-\omega)^2 \right) 
\ - \ 
\frac{\ti t_\ell(-\omega)}n \tr D_\ell^2 T(-\omega)^2 
\right) 
\end{equation} 
admits a unique solution for $n$ large enough. In particular if $\kappa=0$, then ${\bs p}_{\ell,n}=0$ 
and ${\bs w}_{\ell,n}=0$.
\item 
\label{th-bias-expression-bias} 
Let 
\begin{equation}
\label{eq-def-beta} 
\beta_n(\omega) = \frac 1n \sum_{\ell=1}^n {\bs w}_{\ell,n}(\omega) \ . 
\end{equation} 
Then ${\mathcal B}_n(\rho) \stackrel{\triangle}{=} \int_{\rho}^\infty \beta_n(\omega)\, d\omega$ is well-defined, moreover,
\begin{equation}
\label{eq-integrability-beta} 
\limsup_n 
\int_\rho^\infty \left| \beta_n(\omega) \right| d\omega \ < \ \infty.
\end{equation}
Furthermore, 
\begin{equation}
\label{eq-result-bias} 
N \left( \E {\mathcal I}_n(\rho) - V_n(\rho) \right)
- {\mathcal B}_n(\rho)
\quad \xrightarrow[n\rightarrow \infty,\ \frac Nn \rightarrow c]{\ } \quad 
0 \ .  
\end{equation} 
\end{enumerate} 
\end{theo} 
Proof of Theorem \ref{th-bias} is postponed to Section \ref{sec-proof-bias}. 

\section{The CLT for a limiting variance profile}
\label{section-limiting}
In this section, we shall assume that Assumption
{\bf A-\ref{hypo-limit-variance-field}} holds, {\em i.e.}
$\sigma^2_{ij}(n)=\sigma^2(i/N,j/n)$ for some continuous nonnegative
function $\sigma^2(x,y)$.  Recall the definitions
\eqref{eq-func-limit-var-profile} of function $\tau$ and of the $t_i$'s 
(defined in Theorem \ref{first-order}-(1)). In the sequel,
we take $\rho>0$, $z=-\rho$ and denote
$\tau(t)\stackrel{\triangle}{=}\tau(t,-\rho)$.
We first gather convergence results relating the $t_i$'s 
and $\tau$.

\begin{lemma}\label{con-con}
  Consider a variance profile $(\sigma_{ij})$ which fulfills
  assumption {\bf A-\ref{hypo-limit-variance-field}}. Recall the
  definitions of the $t_i$'s and $\tau$. Let $\rho>0$ and let $z=-\rho$ be fixed.
  Then, the following convergences hold true:
\begin{enumerate}
\item 
$
\frac 1N \sum_{i=1}^N t_i \delta_{\frac iN} \xrightarrow[n\rightarrow \infty]{w} \tau(u)\,du\ ,
$
where $\xrightarrow[]{w}$ stands for the weak convergence of measures.\\
\item  
$
\sup_{i\le N} \left| t_i -\tau(i/N) \right| \xrightarrow[n\rightarrow \infty]{} 0\ .
$ \\
\item 
$
\frac 1N \sum_{i=1}^N t^2_i \delta_{\frac iN} \xrightarrow[n\rightarrow \infty]{w} \tau^2(u)\,du\ ,
$
\end{enumerate}
\end{lemma}
\begin{proof}
The first item of the lemma follows from Lemma \ref{th-deter-approx-details}-(b) together with Theorem 2.3-(3) 
in \cite{HLN06}.

In order to prove item (2), one has to compute 
\begin{multline*}
t_i-\tau(i/N)=  \left(
\rho  + \frac 1n \sum_{j=1}^n 
\frac{\sigma^2(i/N,j/n)}
{1 + \frac{1}{n} \sum_{\ell=1}^N \sigma^2(\ell/N,j/n) t_{\ell}} \right)^{-1}\\
- \left( \rho +\int_0^1 \frac{\sigma^2(u,v)}{1+c\int_0^1 \sigma^2(x,v)\tau(x)\,dx}dv \right)^{-1}
\end{multline*}
and use the convergence proved in the first part of the lemma. In
order to prove the uniformity over $i\le N$, one may recall that
$C[0,1]^2=C[0,1]\otimes C[0,1]$ which in particular implies that
$\forall \varepsilon>0$, there exist $g_{\ell}$ and $h_{\ell}$ such
that $\sup_{x,y} |\sigma^2(x,y) -\sum_{\ell=1}^L g_{\ell}(x)
h_{\ell}(y)|\le \varepsilon$. Details are left to the reader.

The convergence stated in item (3) is a direct consequence of item (2).
\end{proof}

\subsection{A continuous kernel and its Fredholm determinant}
Let $K:[0,1]^2 \rightarrow \R$ be some non-negative continuous
function we shall refer to as a kernel.  Consider the associated
operator (similarly denoted with a slight abuse of notations):
\begin{eqnarray*}
K: C[0,1] & \rightarrow & C[0,1] \\
f & \mapsto & K f(x)= \int_{[0,1]} K(x,y)f(y)\, dy\ .
\end{eqnarray*}
Then one can define (see for instance \cite[Theorem 5.3.1]{Smi58}) 
the Fredholm determinant $ \det(1 +\lambda
K) $, where $1: f\mapsto f$ is the identity operator, as 
\begin{equation}\label{def-fredholm}
\det (1 -\lambda K) =\sum_{k=0}^\infty \frac {(-1)^k \lambda^k}{k!} \int_{[0,1]^k} 
K
\left( 
\begin{array}{ccc}
x_1 & \cdots & x_k \\
x_1 & \cdots & x_k 
\end{array}\right) \otimes_{i=1}^k d\,x_i
\end{equation}
where 
$$
K
\left( 
\begin{array}{ccc}
x_1 & \cdots & x_k \\
y_1 & \cdots & y_k 
\end{array}\right) = \det ( K(x_i, y_j),\ 1\le i,j\le n)\ ,
$$
for every $\lambda\in \mathbb{C}$. One can define the trace of the iterated kernel as:
$$
\tr K^k =\int_{[0,1]^k} K(x_1,x_2)\cdots K(x_{k-1},x_k)K(x_k,x_1) dx_1\cdots dx_k 
$$

In the sequel, we shall focus on the following kernel:
\begin{equation}\label{kernel}
K_{\infty}(x,y)=\frac{c \int_{[0,1]} \sigma^2(u,x) \sigma^2(u,y) \tau^2(u)\, du}
{\left( 1+c \int_{[0,1]} \sigma^2(u,x) \tau(u) \,du\right)^2}\ .
\end{equation}

\begin{theo}[The variance]\label{lem-limiting-variance}
Assume that assumptions {\bf A-\ref{hypo-moments-X}} and
{\bf A-\ref{hypo-limit-variance-field}} hold. Let $\rho>0$ and recall the definition of matrix $A_n$:
$$
a_{\ell,m}= \frac 1n \frac{\frac{1}{n} \sum_{i=1}^N \sigma^2\left(\frac iN, \frac{\ell}n\right)
\sigma^2\left(\frac iN, \frac m n\right) t_i^2}{\left(1
    + \frac{1}{n} \sum_{i=1}^N  \sigma^2\left(\frac iN, \frac{\ell} n\right) t_i \right)^2} \ ,\ 1\le \ell, m\le n\ .
$$
Then:
\begin{enumerate}
\item $ \tr A_n \xrightarrow[n\rightarrow \infty]{} \tr K_{\infty}\ .
  $\\
\item $ \det (I_n -A_n) \xrightarrow[n\rightarrow \infty]{}
  \det(1-K_{\infty})
  $  and $\det (1-K_{\infty})\neq 0$. \\
\item Let $\kappa=\E |X_{11}|^4 -2 $, then 
$$
0< -\log \det(1-K_{\infty}) +\kappa \tr K_{\infty} <\infty\ .
$$
\end{enumerate}
\end{theo}

\begin{proof}
  The convergence of $\tr A_n$ toward $\tr K_{\infty}$ follows
  from Lemma \ref{con-con}-(1),(3). Details of the proof are left to
  the reader.

Let us introduce the following kernel:
$$
K_n(x,y) = \frac{\frac 1n \sum_{i=1}^N \sigma^2(\frac iN, x)\sigma^2(\frac iN,y) t_i^2}
{\left(1 +\frac 1n \sum_{i=1}^N \sigma^2(\frac iN,x) t_i\right)^2}\ .
$$ 
One may notice in particular that $a_{\ell, m}=\frac 1n K_n(\frac \ell
n, \frac mn)$. Denote by $\|\cdot\|_{\infty}$ the supremum norm for a function over $[0,1]^2$
and by $\sigma_{\max}^2= \| \sigma^2 \|_{\infty}$, then:
\begin{equation}\label{maj-sup}
\| K_n \|_{\infty} \le \frac Nn
\frac{\sigma_{\max}^4}{\rho^2}\quad \textrm{and}\quad
\| K_{\infty}\|_{\infty} \le c \frac{\sigma_{\max}^4}{\rho^2}\ .
\end{equation}
The following facts (whose proof is omitted) can be established:
\begin{enumerate}
\item The family $(K_n)_{n\ge 1}$ is uniformly equicontinuous,
\item For every $(x,y)$, $K_n(x,y) \rightarrow K_{\infty}(x,y)$ as $n\rightarrow \infty$.
\end{enumerate}
In particular, Ascoli's theorem implies the uniform convergence of $K_n$ toward $K_{\infty}$.
It is now a matter of routine to extend these results and to get the following convergence:
\begin{equation}\label{conv-unif}
K_n
\left( 
\begin{array}{ccc}
x_1 & \cdots & x_k \\
y_1 & \cdots & y_k 
\end{array}\right) \xrightarrow[n\rightarrow \infty]{} K_{\infty} 
\left( 
\begin{array}{ccc}
x_1 & \cdots & x_k \\
y_1 & \cdots & y_k 
\end{array}\right)
\end{equation}
uniformly over $[0,1]^{2k}$. Using the uniform convergence
\eqref{conv-unif} and a dominated convergence argument, we obtain:
$$
\frac 1{n^k} \sum_{1\le i_1,\cdots, i_k\le n} K_n
\left( 
\begin{array}{ccc}
i_1/n & \cdots & i_k/n \\
i_1/n & \cdots & i_k/n 
\end{array}\right) \xrightarrow[n\rightarrow \infty]{} \int_{[0,1]^k} K_{\infty} 
\left( 
\begin{array}{ccc}
x_1 & \cdots & x_k \\
x_1 & \cdots & x_k 
\end{array}\right) \otimes_{i=1}^k d\,x_i\ .
$$
Now, writing the determinant $\det(I_n +\lambda A_n)$ explicitely and
expanding it as a polynomial in $\lambda$, we obtain:
$$
\det(I_n -\lambda A_n) = \sum_{k=0}^n \frac{(-1)^k \lambda^k}{k!}\left(\frac 1{n^k}\sum_{1\le i_1,\cdots, i_k\le n} K_n
\left( 
\begin{array}{ccc}
i_1/n & \cdots & i_k/n \\
i_1/n & \cdots & i_k/n 
\end{array}\right) \right)\ .
$$
Applying Hadamard's inequality (\cite[Theorem 5.2.1]{Smi58}) to the
determinants $K_n(\cdot)$ and $K_{\infty}(\cdot)$ yields:
$$
\frac 1{n^k}\sum_{1\le i_1,\cdots, i_k\le n} K_n
\left( 
\begin{array}{ccc}
i_1/n & \cdots & i_k/n \\
i_1/n & \cdots & i_k/n 
\end{array}\right)\quad  \le \quad {k^{\frac k2}\|K_n\|_{\infty}^k}
\quad \stackrel{(a)}{\le} \quad 
{k^{\frac k2}M^k}\ ,
 $$
where $(a)$ follows from \eqref{maj-sup}. Similarly,
$$
\int_{[0,1]^k} K_{\infty} 
\left( 
\begin{array}{ccc}
x_1 & \cdots & x_k \\
x_1 & \cdots & x_k 
\end{array}\right) \otimes_{i=1}^k d\,x_i \quad \le \quad 
{k^{\frac k2}M^k}\ .
$$
Since the series $\sum_k \frac{M^k k^{\frac k2}}{k!} |\lambda|^k$ converges, a dominated convergence argument yields
the convergence 
$$
\det(I_n +\lambda A_n) \xrightarrow[n\rightarrow \infty]{} \det(1+\lambda K_{\infty})\ ,
$$ and item (2) of the theorem is proved. Item (3) follows from Theorem \ref{th-variance}-(2) 
and the proof of the theorem is completed.

\end{proof}

\subsection{The CLT: Fluctuations and bias }

\begin{coro}[Fluctuations]\label{theo-limiting} Assume that (A-\ref{hypo-moments-X}) and
(A-\ref{hypo-limit-variance-field}) hold. Denote by 
$$
\Theta^2_{\infty} = -\log \det (1 -K_{\infty}) +\kappa \tr K_{\infty}\ ,
$$
then
\begin{eqnarray*}
  \lefteqn{\frac N{\Theta_{\infty}} \left( {\mathcal I}_n(\rho) - \mathbb{E} {\mathcal I}_n(\rho)\right)}\\
  &=& \Theta_{\infty}^{-1} \left( \log \det \left(Y_n Y_n^* +\rho I_N \right) 
    -\mathbb{E}\log \det \left(Y_n Y_n^* +\rho I_N\right)\right) \xrightarrow[n\rightarrow\infty]{\mathcal L} {\mathcal N}(0,1)\ .
\end{eqnarray*}

\end{coro}
\begin{proof} follows easily from Theorem \ref{th-clt} and Theorem \ref{lem-limiting-variance}.
\end{proof}

Recall the definition of $\tilde \tau$ (cf. Remark \ref{tau-tilde}).  

\begin{theo}[The bias]\label{th-bias-limiting-variance} Assume that the setting of Corollary \ref{theo-limiting} 
  holds true. Let $\ \omega \in [\rho,\infty)$ and denote by ${\bs
    p}:[0,1]\rightarrow \mathbb{R}$ the quantity:
\begin{multline*}
{\bs p}(x,\omega)= \kappa \omega^2 \tilde \tau^2(x,-\omega) \\
\times \Big\{ \omega c \int_0^1 \sigma^2(u,x) \tau^3(u) \left( \int_0^1 \sigma^2(s,u) \tilde \tau^2(s)\,ds \right)\,du\\
-\tilde \tau(x) c\int_0^1 \sigma^2(u,x) \tau^2(u) \, du\Big\}\ .
\end{multline*}
The following functional equation admits a unique solution:
$$
{\bs w}(x,\omega) =\int_0^1 \frac{c\int_0^1 \sigma^2(u,x) \sigma^2(u,y) \tau^2(u)\,du}
{\left(1 + c\int_0^1 \sigma^2(u,x) \tau(u)\,du\right)^2} {\bs w}(y,\omega) \, dy +{\bs p}(x,\omega)\ .
$$
Let 
$
\beta_{\infty}(\omega) = \int_0^1 {\bs w}(x,\omega)\, dx\ .
$
Then  $\int_{\rho}^\infty \left| \beta_{\infty}(\omega) \right|\, d\omega <\infty$.  
Moreover, 
\begin{equation}\label{conv-bias-limiting}
N \left( \E {\mathcal I}_n(\rho) - V_n(\rho) \right)
\quad \xrightarrow[n\rightarrow \infty,\ \frac Nn \rightarrow c]{} \quad 
{\mathcal B}_{\infty}(\rho)\stackrel{\triangle}{=} \int_\rho^\infty \beta_{\infty}(\omega) d\omega \ .  
\end{equation}
\end{theo}
Proof of Theorem \ref{th-bias-limiting-variance}, although technical,
follows closely the classical Fredholm theory as presented for instance
in \cite[Chapter 5]{Smi58}. We sketch it below.

\begin{proof}[Sketch of proof]
  The existence and unicity of the functional equation follows from
  the fact that the Fredholm determinant $\det(1-K_{\infty})$ differs
  from zero. In order to prove the convergence
  \eqref{conv-bias-limiting}, one may prove the convergence
  $\int_\rho^{\infty} \beta_n \rightarrow \int_\rho^{\infty}
  \beta_\infty$ (where $\beta_n$ is defined in Theorem \ref{th-bias})
  by using an explicit representation for $\beta_\infty$ relying on
  the explicit representation of the solution ${\bs w}$ via the
  resolvent kernel associated to $K_{\infty}$ (see for instance
  \cite[Section 5.4]{Smi58}) and then approximate the resolvent kernel
  as done in the proof of Theorem \ref{lem-limiting-variance}.
\end{proof}
\subsection{The case of a separable variance profile}
We now state a consequence of Corollary \ref{theo-limiting} in the
case where the variance profile is separable. Recall the definitions
of $\tau$ and $\tilde \tau$ given in \eqref{coupled-eq}.

\begin{coro}[Separable variance profile]\label{magic-coro}
  Assume that {\bf A-\ref{hypo-moments-X}} and
{\bf A-\ref{hypo-limit-variance-field}} hold. 
  Assume moreover that $\rho>0$ and that $\sigma^2$ is separable, {\em i.e.} that 
$$ 
\sigma^2(x,y)=d(x) \tilde d(y)\ , 
$$ 
where both $d:[0,1]\rightarrow (0,\infty)$ and $\tilde d: [0,1] \rightarrow
  (0,\infty)$ are continuous functions. Denote by 
$$
\gamma=c\int_0^1 d^2(t) \tau^2(t)\, dt\qquad \textrm{and}\qquad 
\tilde \gamma= \int_0^1 \tilde d^2(t) \tilde \tau^2(t)\, dt\ . 
$$
Then 
\begin{equation}\label{magic-variance}
\Theta^2_{\infty} = -\log\left( 1 - \rho^2 \gamma \tilde \gamma\right) +\kappa \rho^2 \gamma \tilde \gamma\ .
\end{equation}
\end{coro}

\begin{rem}
In the case where the random variables $X_{ij}$ are standard complex circular gaussian
(i.e. $X_{ij}= U_{ij} +\ii V_{ij}$ with $U_{ij}$ and $V_{ij}$ independent real centered 
gaussian random variables with variance $2^{-1}$) and where the variance profile is separable, 
then 
$$
N({\mathcal I}_n(\rho) - V_n(\rho)) \xrightarrow[n\rightarrow \infty]{\mathcal L} 
{\mathcal N}\left(0,-\log\left( 1 - \rho^2 \gamma \tilde \gamma\right)\right)\ .
$$
This result is in accordance with those in \cite{MSS03} and in \cite{HKLNP06pre}.
\end{rem}

\begin{proof} Recall the definitions of $\tau$ and $\tilde \tau$ given
  in \eqref{coupled-eq}. In the case where the variance profile is
  separable, the kernel $K_{\infty}$ writes:
$$
K_{\infty}(x,y)\quad =\quad \frac{c \tilde d(x) \tilde d(y) \int_{[0,1]} d^2(u)
  \tau^2(u) \, du}{\left(1 +c\tilde d(x) \int_{[0,1]} d(u) \tau(u) \,
    du \right)^2}\quad  = \quad \rho^2 \gamma \tilde d(x) \tilde d(y) \tilde
\tau^2(x)\ .
$$
In particular, one can readily prove that $\tr K_{\infty} = \rho^2
\gamma \tilde \gamma$.  Since the kernel $K_{\infty}(x,y)$ is itself a
product of a function depending on $x$ times a function depending on
$y$, the determinant 
$
K_{\infty}
{\tiny \left( 
\begin{array}{ccc}
x_1 & \cdots & x_k \\
y_1 & \cdots & y_k 
\end{array}\right)} 
$ is equal to zero for $k\ge 2$ and the Fredholm determinant writes
$
\det(1-K_{\infty})= 1-\int_{[0,1]} K_{\infty}(x,x)dx= 1-\rho^2 \gamma \tilde \gamma\ .
$
This yields
$$
-\log \det (1-A_{\infty})+\kappa \tr K_{\infty} \quad =\quad -\log(1-\rho^2 \gamma \tilde \gamma) +\kappa \rho^2 \gamma \tilde \gamma\ ,
$$ 
which ends the proof. 
\end{proof}

\section{Proof of Theorem \ref{th-variance}}
\label{proof-variance}
Recall the definition of the $n\times n$ variance matrix $A_n$:
$$
a_{\ell, m} = \frac{1}{n^2} \frac{\tr D_{\ell} D_m T(-\rho)^2}
{\left(1 + \frac{1}{n} \tr D_{\ell} T(-\rho)  \right)^2} \ , \quad 1\le \ell, m \le n.
$$
In the course of the proof of the CLT (Theorem
\ref{th-clt}), the quantity that will naturally pop up as a variance
will turn out to be: 
\begin{equation}\label{thetatilde}
\tilde \Theta_n^2 = \tilde{\mathcal V}_n +\kappa
{\mathcal W}_n
\end{equation}
(recall that ${\mathcal W}_n=\tr A_n$) where
$\tilde{\mathcal V}_n$ is introduced in the following lemma:

\begin{lemma}\label{system-aux} Consider a variance profile $(\sigma_{ij})$ which fulfills assumptions 
  {\bf A-\ref{hypo-variance-field}} and {\bf
    A-\ref{hypo-inf-trace-Dj}} and the related $t_i$'s defined in
  Theorem \ref{first-order}-(1). Let $\rho>0$ and consider the matrix
  $A_n$ as defined above.
\begin{enumerate}
\item For $1\le j\le n$, the system of $(n-j+1)$ linear equations with unknown parameters 
$(\bs{y_{\ell,n}^{(j)}},\ j\le\ell \le n)$:
\begin{equation}\label{TheSystem}
\bs{y_{\ell,n}^{(j)}}=
\sum_{m=j+1}^n  a_{\ell, m}\  
\bs{y_{m,n}^{(j)}}
+ a_{\ell, j}
\end{equation}
admits a unique solution for $n$ large enough. 
\end{enumerate}
Denote by $\tilde{\mathcal V}_n$ the sum of the first components 
of vectors $(\bs{y_{\ell,n}^{(j)}},\ j\le\ell \le n)$, i.e.:
$$
\tilde{\mathcal V}_n=\sum_{j=1}^n \bs{y_{j,n}^{(j)}}.
$$ 
\begin{itemize}
\item[2.] Let $\kappa$ be a real number satisfying $\kappa \geq -1$.
The sequence $\left(\tilde {\mathcal V}_n + \kappa {\mathcal W}_n\right)$
satisfies 
$$
0\quad  < \quad \liminf_n \left(\tilde {\mathcal V}_n + \kappa {\mathcal W}_n \right)\quad \leq\quad  \limsup_n 
\left( \tilde{\mathcal V}_n + \kappa {\mathcal W}_n \right)\quad 
<\quad  \infty
$$ as $n \to \infty$ and $N/n \to c > 0$. \\
\item[3.] 
The following holds true: 
$$
\tilde{\mathcal V}_n + \log\det(I_n -A_n) \xrightarrow[n\rightarrow \infty]{} 0.
$$
\end{itemize}
\end{lemma}

Obviously, Theorem \ref{th-variance} is a by-product of Lemma
\ref{system-aux}. The remainder of the section is devoted to the proof
of this lemma.

We cast the linear system \eqref{TheSystem} into a matrix framework and we denote by 
$A_n^{(j)}$ the $(n-j+1)\times (n-j+1)$ submatrix $A_n^{(j)}=(a_{\ell, m})_{\ell, m=j}^n$, by
$A_n^{0,(j)}$ the $(n-j+1)\times (n-j+1)$ matrix $A_n^{(j)}$ where the first column is replaced by zeros. 
Denote by $\bs{d_{n}^{(j)}}$ the $(n-j+1) \times 1$ vector :
$$
\bs{d_{n}^{(j)}} = \left( \frac 1n \frac{ \frac{1}{n} \tr D_{\ell} D_{j}
    T(-\rho)^2}{\left( 1+\frac 1n \tr D_{\ell} T(-\rho)\right)^2}
\right)_{\ell=j}^n \ .
$$
These notations being introduced, the system can be 
rewritten as:
\begin{equation}
\label{eq-x=Ax+v}
\bs{y_n^{(j)}} =A_n^{0,(j)} \bs{y_n^{(j)}}  + \bs{d_{n}^{(j)}}
\qquad \Leftrightarrow \qquad
(I-A_n^{0,(j)})\bs{y_n^{(j)}}= \bs{d_n^{(j)}} \ .
\end{equation} 
The key issue that appears is to study the invertibility of matrix
$(I-A_n^{0,(j)})$ and to get some bounds on its inverse.

\subsection{ Results related to matrices with nonnegative entries }

The purpose of the next lemma is to state some of the properties of
matrices with non-negative entriess that will appear to be satisfied
by matrices $A_n^{0,(j)}$. We shall use the following notations.
Assume that $M$ is a real matrix. By $M \succ 0$ (resp.~$M
\succcurlyeq 0$) we mean that $m_{ij} > 0$ (resp.~$m_{ij} \geq 0$) for
every element $m_{ij}$ of $M$. We shall write $M \succ M'$ (resp. $M
\succcurlyeq M'$) if $M-M' \succ 0$ (resp. $M -M'\succcurlyeq 0$).  If
$x$ and $y$ are vectors, we denote similarly $x \succ 0$, $x
\succcurlyeq 0$ and $x \succcurlyeq y$.

\begin{lemma}
\label{lm-(I-A)regular}
Let $A=(a_{\ell, m})_{\ell, m=1}^n$ be a $n \times n$ real matrix and 
$u = (u_{\ell}, 1 \leq \ell \leq n)$, 
$v = (v_{\ell}, 1 \leq \ell \leq n)$ be two real 
$n \times 1$ vectors. Assume that $A \succcurlyeq 0$, $u \succ 0$, and
$v \succ 0$. Assume furthermore that equation 
$$
u = Au + v 
$$ 
is satisfied. Then:
\begin{enumerate}
\item The spectral radius $r(A)$ of $A$ satisfies 
$r(A) \leq 1 - \frac{\min(v_{\ell})}{\max(u_{\ell})}  < 1$.
\item 
\label{it-elts-diag>1} 
Matrix $I_n -A$ is invertible and its inverse $\left( I_n - A \right)^{-1}$ 
satisfies:
$$
\left( I_n - A \right)^{-1}\succcurlyeq 0 \quad \textrm{and}\quad 
\left[ \left( I_n - A \right)^{-1} \right]_{\ell \ell} \geq 1 
$$ for every $1\le \ell \le n$.
\item 
\label{it-bound-maxrownorm}
The max-row norm of the inverse is bounded: $\leftrownorm \left( I_n - A \right)^{-1} \rightrownorm_\infty
\leq \frac{\max_{\ell}(u_{\ell})}{\min_{\ell}(v_{\ell})}$. 
\item\label{quatre} Consider the $(n-j+1) \times (n-j+1)$ submatrix
  $A^{(j)} = ( a_{\ell m} )_{\ell,m=j}^n$ and denote by $A^{0,(j)}$
  matrix $A^{(j)}$ whenever the first column is replaced by zeros.
  Then properties (1) and (2) are valid for $A^{0,(j)}$ and
$$
\leftrownorm \left( I_{(n-j+1)} - A^{(j)} \right)^{-1} \rightrownorm_\infty
\leq \frac{\max_{1\le \ell \le n}(u_{\ell})}
{\min_{1\le \ell \le n}(v_{\ell})}\ .
$$
\end{enumerate}  
\end{lemma}
\begin{proof}
Let $\alpha=  1 - \frac{\min(v_{\ell})}{\max(u_{\ell})}$.
Since $u \succ 0$ and $v \succ 0$, $\alpha$ readily satisfies $\alpha < 1$
and $\alpha u \succcurlyeq u - v = A u$ which in turn implies that  
$r(A) \leq \alpha < 1$ \cite[Corollary 8.1.29]{HorJoh94} and (1) is proved.
In order to prove (2), first note that $\forall m\ge 1,\ A^m \succcurlyeq 0$.
As $r(A) < 1$, the series $\sum_{m\geq 0} A^m$ converges, matrix $I_n - A$ is invertible and  
$(I_n - A)^{-1} = \sum_{m\geq 0} A^m\succcurlyeq
I_n \succcurlyeq 0$. This in particular implies that
$[ \left( I_n - A \right)^{-1} ]_{\ell \ell} \geq 1 $ for 
every $1\le \ell\le n$ and (2) is proved. 
Now $u = ( I_n - A)^{-1} v$ implies that for every $1\le k\le n$,
$$
u_k = 
\sum_{\ell=1}^n \left[ \left( I_n - A \right)^{-1} \right]_{k\ell} v_{\ell} 
\geq \min(v_{\ell}) \sum_{\ell=1}^n \left[ \left( I_n - A \right)^{-1} \right]_{k\ell}\ ,
$$
hence (3). 

We shall first prove (4) for matrix $A^{(j)}$, then show how
$A^{0,(j)}$ inherits from $A^{(j)}$'s properties. In \cite{HorJoh94},
matrix $A^{(j)}$ is called a principal submatrix of $A$.  In
particular, $r(A^{(j)}) \leq r(A)$ by \cite[Corollary
8.1.20]{HorJoh94}. As $A \succcurlyeq 0$, one readily has $A^{(j)}
\succcurlyeq 0$ which in turn implies property (2) for $A^{(j)}$.  Let
$\tilde A^{(j)}$ be the matrix $A^{(j)}$ augmented with zeros to reach
the size of $A$. The inverse $(I_{n-j+1} - A^{(j)})^{-1}$ is a
principal submatrix of $(I_{n} - \tilde A^{(j)})^{-1} \succcurlyeq 0$.
Therefore, $\leftrownorm \left( I_{(n-j+1)} - A^{(j)} \right)^{-1}
\rightrownorm_\infty \le \leftrownorm \left( I_n - \tilde A^{(j)}
\right)^{-1} \rightrownorm_\infty$.  Since $A^m \succcurlyeq (\tilde
A^{(j)})^m$ for every $m$, one has $\sum_{m\geq 0} A^m \succcurlyeq
\sum_{m\geq 0} (\tilde A^{(j)})^m$; equivalently $(I - A)^{-1}
\succcurlyeq (I - \ti A^{(j)})^{-1}$ which yields $\leftrownorm
\left( I- \tilde A^{(j)} \right)^{-1} \rightrownorm_\infty \le
\leftrownorm \left( I - A \right)^{-1} \rightrownorm_\infty$. Finally
(4) is proved for matrix $A^{(j)}$.

We now prove (4) for $A^{0,(j)}$. By \cite[Corollary
8.1.18]{HorJoh94}, $r(A^{0,(j)})\le r(A^{(j)})<1$ as $A^{(j)} \succcurlyeq A^{0,(j)}$.
Therefore, $(I-A^{0,(j)})$ is invertible and 
$$
(I-A^{0,(j)})^{-1} =\sum_{k=0}^\infty [A^{0,(j)}]^k \ .
$$
This in particular yields $(I-A^{0,(j)})^{-1} \succcurlyeq 0$ and $ (I-A^{0,(j)})^{-1}_{k k} \ge 1$.
Finally, as $A^{(j)} \succcurlyeq A^{0,(j)}$, one has
$$
\leftrownorm
\left( I- \tilde A^{0,(j)} \right)^{-1} \rightrownorm_\infty \le
\leftrownorm
\left( I- \tilde A^{(j)} \right)^{-1} \rightrownorm_\infty \ .
$$
Item (4) is proved and so is Lemma \ref{lm-(I-A)regular}.
\end{proof}

\subsection{Proof of Lemma \ref{system-aux}: Some preparation}
The following bounds will be needed: 
\begin{prop}\label{petite-mino}
Let $\rho>0$, consider a variance profile $(\sigma_{ij})$ which fulfills assumption
{\bf A-\ref{hypo-variance-field}} and 
consider the related $t_i$'s defined in Theorem \ref{first-order}-(1).
The following holds true:
$$
\frac 1\rho \quad \ge\quad  t_{\ell}(-\rho) \quad \ge\quad \frac{1}{\rho +\sigma_{\max}^2}\ .
$$
\end{prop}
\begin{proof}
  Recall that $t_{\ell}(z)\in {\mathcal S}(\Rplus)$ by Theorem
  \ref{first-order}.  In particular, $t_{\ell}(-\rho)=\int_{\Rplus}
  \frac{\mu_{\ell}(d\lambda)}{\lambda+\rho}$ for some probability
  measure $\mu_{\ell}$. This yields the upper bound
  $t_{\ell}(-\rho)\le \rho^{-1}$ and the fact that $t_{\ell}(-\rho)
  \ge 0$. Now the lower bound readily follows from Eq.
  \eqref{eq-systeme-approx-deter}.
\end{proof}

\begin{prop}
\label{lm-minorations}
Let $\rho>0$. Consider a variance profile $(\sigma_{ij})$ which fulfills assumptions
{\bf A-\ref{hypo-variance-field}} and {\bf A-\ref{hypo-inf-trace-Dj}};
consider the related $t_i$'s defined in Theorem \ref{first-order}-(1).
Then:
$$
 \liminf_{n\geq 1}\min_{1\le j\le n}
\frac 1n \tr D_{j} T_n(-\rho)^2 > 0   
\quad \mathrm{and} \quad 
\liminf_{n\geq 1}\min_{1\le j\le n}
\frac{1}{n} \tr D_{j}^2 T_n(-\rho)^2  > 0 \ .  
$$
\end{prop}
\begin{proof}
 Applying Proposition \ref{petite-mino} yields:
\begin{equation}
\label{eq-lower-trD2T2} 
\frac 1N \tr D_{j} T(-\rho)^2 = 
\frac 1N \sum_{i=1}^N \sigma^2_{ij} t_i^2(-\rho)
\ge \frac1{ (\rho +\sigma_{\max}^2)^2} \frac 1N \sum_{i=1}^N \sigma^2_{ij}\ ,
\end{equation} 
which is bounded away from zero by Assumption {\bf A-\ref{hypo-inf-trace-Dj}}.
Similarly,
\[ 
\frac 1N \tr D^2_{j} T(-\rho)^2 
\ge \frac1{ (\rho +\sigma_{\max}^2)^2} \frac 1N \sum_{i=1}^N \sigma^4_{ij}
\stackrel{(a)}{\ge}\frac1{ (\rho +\sigma_{\max}^2)^2}
\left(
\frac 1N \sum_{i=1}^N \sigma^2_{ij}
\right)^2\ ,
\] 
which remains bounded away from zero for the same reasons (notice that
$(a)$ follows from the elementary inequality $(n^{-1}\sum x_i)^2 \le
n^{-1} \sum x_i^2$).
\end{proof}
We are now in position to study matrix $A_n = A_n^{(1)}$.
\begin{prop}
\label{lm-properties-A}
Let $\rho>0$. Consider a variance profile $(\sigma_{ij})$ which
fulfills assumptions {\bf A-\ref{hypo-variance-field}} and {\bf
  A-\ref{hypo-inf-trace-Dj}}; consider the related $t_i$'s defined in
Theorem \ref{first-order}-(1) and let $A_n$ be the variance matrix.
Then there exist two $n\times 1$ real vectors $u_n=(u_{\ell n}) \succ
0$ and $v_n=(v_{\ell n}) \succ 0$ such that $u_n =A_n u_n +v_n$.
Moreover,
$$
\sup_n \max_{1\le \ell\le n} (u_{\ell n}) < \infty 
\quad \mathrm{and} \quad 
\liminf_n \min_{1\le \ell\le n} (v_{\ell n}) > 0 \ . 
$$ 
\end{prop}
\begin{proof}
Let $z = - \rho + \delta \ii$ with $\delta \in \R - \{ 0 \}$.  
An equation involving matrix $A_n$ will show up by 
developing the expression of $\im(T(z)) = \left( T(z) - T^*(z) \right)
/ 2 \ii$ and by using the expression of the $t_i(z)$'s given by Theorem
\ref{first-order}-(1). We first rewrite the system (\ref{eq-systeme-approx-deter}) as: 
$$
T(z) = 
\left( - z I_N + 
\frac{1}{n} \sum_{m=1}^n 
\frac{D_m}{1+\frac 1n \tr D_m T} 
\right)^{-1} \ . 
$$
We then have 
\begin{eqnarray*}
\im(T) &=& 
\frac{1}{2\ii}\left(T - T^* \right) \ = \ 
\frac{1}{2\ii} T T^* \left( {T^*}^{-1} - T^{-1} \right)\ , \\
&=&
\frac 1n \sum_{m=1}^n 
\frac{ D_m T T^*}
{\left| 1 + \frac 1n \tr D_m T \right|^2} 
 \im\left( \frac 1n \tr D_m T \right) 
+  \delta T T^*  \ . 
\end{eqnarray*} 
This yields in particular, for any $1\le \ell\le n$:
\begin{equation}\label{eq-system}
\frac 1\delta \im\left(\frac 1n \tr D_{\ell} T \right) = 
\frac{1}{n^2} \sum_{m=1}^n 
\frac{\tr D_{\ell} D_m T T^*} 
{\left| 1 + \frac 1n \tr D_m T \right|^2} 
\frac 1\delta \im\left( \frac 1n \tr D_m T \right) 
+ \frac1n \tr D_{\ell} T T^* \ . 
\end{equation} 
Recall that for every $1\le i\le N$, $t_i(z)\in {\mathcal S}(\Rplus)$.
Denote by $\mu_i$ the probability measure associated with $t_i$ {\em
  i.e.} $t_i(z)=\int_{\Rplus} \frac{\mu_i(d\lambda)}{\lambda-z}$.
Then
\[
\frac 1\delta \im\left(\frac 1n \tr D_{\ell} T \right) = 
\frac 1n \sum_{i=1}^N \sigma^2_{i\ell} 
\int_0^\infty \frac{\mu_i(d\lambda)}{| \lambda - z |^2}  \xrightarrow[\delta\rightarrow 0]{} 
\frac 1n \sum_{i=1}^N \sigma^2_{i\ell}  \int_0^\infty \frac{\mu_i(d\lambda)}{( \lambda + \rho)^2}  \ .
\]
Denote by $\tilde u_{\ell n}$ the right handside of the previous limit
and let $u_{\ell n}= \frac{\tilde u_{\ell n}}{(1+\frac 1n \tr D_{\ell}
  T(-\rho))^2}$. Plugging this expression into
\eqref{eq-system} and letting $\delta \to 0$, we end up
with equation:
$$
u_n = A_n u_n +v_n
$$
$A_n \succcurlyeq 0$ is given in the statement of the lemma, 
$u_n = ( u_{\ell,n}; 1 \leq \ell \leq n)$ and 
$v_n = ( v_{\ell n}; 1 \leq \ell \leq n)$ 
are the $n \times 1$ vectors with elements 
\begin{equation}
\label{eq-u-v} 
u_{\ell,n} = \frac{
\frac 1n \sum_{i=1}^N \sigma^2_{i\ell}
\int_0^\infty \frac{\mu_i(d\lambda) }{( \lambda + \rho )^2} }{\left( 1 + \frac1n \tr D_{\ell} T(-\rho)\right)^2}
\quad \mathrm{and} \quad 
v_{\ell,n} =  \frac{\frac 1n \tr D_{\ell} T^2(-\rho)}{\left( 1 +\frac 1n \tr D_{\ell} T(-\rho) \right)^2} \ .
\end{equation} 
It remains to notice that $u_n \succ 0$ and $v_n \succ 0$ for $n$
large enough due to {\bf A-\ref{hypo-inf-trace-Dj}}, that the
numerator of $u_{\ell,n}$ is lower than $(N \sigma_{\max}^2) / (n
\rho^2)$ and that its denominator is bounded away from zero (uniformly
in $n$) by Propositions \ref{petite-mino} and \ref{lm-minorations}.
Similar arguments hold to get a uniform upper bound for $v_{\ell,n}$.
This concludes the proof of Proposition \ref{lm-properties-A}.
\end{proof}

\subsection{Proof of Lemma \ref{system-aux}: End of proof}

\begin{proof}[Proof of Lemma \ref{system-aux}-(1)]
Proposition \ref{lm-properties-A} together with Lemma \ref{lm-(I-A)regular}-(4) yield that $I- A^{0,(j)}$ is invertible, 
therefore the system \ref{eq-x=Ax+v} admits a unique solution given by:
$$
\bs{y_n^{(j)}}= (I-A_n^{0,(j)})^{-1}\bs{d_n^{(j)}} \ ,
$$
and (1) is proved.
\end{proof}

\begin{proof}[Proof of Lemma \ref{system-aux}-(2)]
Let us first prove the upper bound. 
Proposition \ref{lm-properties-A} together with Lemma \ref{lm-(I-A)regular} yield
\begin{eqnarray*}
\limsup_{n}\max_{j} \leftrownorm ( I - A^{0,(j)} )^{-1} \rightrownorm_\infty
&\le& 
\limsup_{n\ge 1} \frac{\max_{1\le \ell\le n} (u_{\ell n})}
{\min_{1\le \ell\le n} (v_{\ell n})} \quad <\quad \infty\ .
\end{eqnarray*}
Each component of vector $\bs {d_n^{(j)}}$ satisfies 
$\bs{d_{\ell,n}^{(j)}} \leq  \frac{N \sigma_{\max}^4}{ n^2 \rho^2}$
{\em i.e.} $\sup_{1\le j\le n} \| \bs{d_n^{(j)}} \|_\infty < \frac{K}n$. Therefore, 
vector $\bs{y_n^{(j)}}$ satisfies:
$$
\sup_{j} \| \bs{y_n^{(j)}} \|_\infty \leq 
\sup_{j} |\!|\!| ( I - A_n^{0,(j)} )^{-1} |\!|\!|_\infty 
\| \bs{d_n^{(j)}} \|_\infty  < \frac Kn \ .
$$ 
Consequently,
$$
0\ \le\  \check{\mathcal V}_n = \sum_{j=1}^n \bs{\check{y}_{j,n}^{(j)}}
\ \le\  \sum_{j=1}^n \| \bs{\check{y}_n^{(j)}} \|_{\infty}
$$ 
satisfies $\limsup_n \check{\mathcal V}_n <\infty$. Moreover, Proposition \ref{petite-mino} yields
${\mathcal W}_n \leq n^{-2} \sum_{j=1}^n \tr D_j^2 T^2 \leq
\sigma_{\max}^4 N (\rho^2 n)^{-1}$. In particular, ${\mathcal W}_n$   is also bounded and $\limsup_n
(\check{\mathcal V}_n +\kappa {\mathcal W}_n) \le \limsup_n(\check{\mathcal V}_n
+|\kappa| {\mathcal W}_n) < \infty$.

We now prove the lower bound.
$$
\check{\mathcal V}_n +\kappa {\mathcal W}_n =\sum_{j=1}^n \bs{y_{j,n}^{(j)}} +\kappa \bs{d^{(j)}_{j,n}}
\ge \sum_{j=1}^n \bs{y_{j,n}^{(j)}} - \bs{d^{(j)}_{j,n}}\ .
$$
Recall that $\bs{y^{(j)}_n} =(I-A^{0,(j)})^{-1} \bs{d_n^{(j)}}$. We therefore have:
\begin{eqnarray*}
\bs{y_{j,n}^{(j)}} - \bs{d^{(j)}_{j,n}} &=& \left[ \bs{y_{n}^{(j)}} - \bs{d^{(j)}_{n}}\right]_1
\quad =\quad  \left[ \left( (I-A_n^{0,(j)})^{-1} -I\right)\bs{d_n^{(j)}}\right]_1\\
& = & \left[ (I-A_n^{0,(j)})^{-1} A_n^{0,(j)} \bs{d_n^{(j)}}\right]_1\ . 
\end{eqnarray*}
As $(I-A_n^{0,(j)})^{-1} \succcurlyeq I$, we have:
\begin{eqnarray*}
\bs{y_{j,n}^{(j)}} - \bs{d^{(j)}_{j,n}} &\ge & \left[ A_n^{0,(j)} \bs{d_n^{(j)}}\right]_1
\quad =\quad \sum_{\ell =j+1}^n \frac 1{n^2} \frac{\left( \frac 1n \tr D_{\ell} D_j T^2(-\rho)\right)^2}
{\left( 1 +\frac 1n \tr D_j T(-\rho)\right)^4}\\
&\stackrel{(a)}{\ge} & K \sum_{\ell =j+1}^n \frac 1{n^2} \left( \frac 1n \tr D_{\ell} D_j\right)^2,
\end{eqnarray*}
where $(a)$ follows from Proposition \ref{petite-mino}, which is used
both to get a lower bound for the numerator 
and an upper bound for the denominator:
$(1+\frac 1n \tr D_j T)^4 \le (1 + Nn^{-1} \sigma_{\max}^2
\rho^{-1})^4$. Some computations remain to be done in order to take
advantage of {\bf A-\ref{hypo-inf-trace-Dj}} to get the lower bound.
Recall that $\frac 1m \sum_{k=1}^m x_k^2 \ge \left( \frac 1m
  \sum_{k=1}^m x_k\right)^2$. We have:
\begin{eqnarray*}
  \sum_{j=1}^n \bs{y^{(j)}_{j,n}} - \bs{d^{(j)}_{j,n}} &\ge & \sum_{j=1}^n \sum_{\ell = j+1}^n \frac 1{n^2} \left( \frac 1n \tr D_{\ell} D_j\right)^2 \\
  &=& \frac 1{n^2} \times \frac{n(n-1)}{2}\times\frac{2}{n(n-1)} \sum_{j< \ell} \left( \frac 1n \tr D_{\ell} D_j \right)^2\\
  &\stackrel{(a)}{\ge}& \frac 13 \left( \frac 2{n(n-1)} \sum_{j<\ell} \frac 1n \tr D_{\ell} D_j \right)^2\\
  &\stackrel{(b)}{=} & \frac 13 \left( \frac 1{n(n-1)} \sum_{1\le j,\ell\le n} 
  \frac 1n \tr D_{\ell} D_j \right)^2  + o(1)\\ 
&\ge& \frac 13 \left( \frac 1{n^3} \sum_{i=1}^N \left( \sum_{j=1}^n \sigma_{ij}^2 \right)^2 \right)^2 + o(1)\\
&\ge& \frac 13 \left( \frac N{n^3} \left( \frac 1N \sum_{i=1}^N \sum_{j=1}^n \sigma_{ij}^2 \right)^2 \right)^2 + o(1)\\
&\ge& \frac 13 \left( \frac N{n^3} \left( \sum_{j=1}^n \frac 1N \sum_{i=1}^N \sigma_{ij}^2 \right)^2 \right)^2 +o(1)
\end{eqnarray*}
where $(a)$ follows from the bound $\frac{n(n-1)}{2n^2} \ge \frac
13$ valid for $n$ large enough.  The term $o(1)$ at step $(b)$ goes to zero
as $n\rightarrow \infty$ and takes into account the diagonal terms in
the formula $2 \sum_{j< \ell} \alpha_{j\ell} +\sum_j \alpha_{jj}=
\sum_{j,\ell} \alpha_{j\ell}$. It remains now to take the $\liminf$ to
obtain:
$$
\liminf_{n\rightarrow \infty } \left( \sum_{j=1}^n \bs{y^{(j)}_{j,n}} +\kappa  \bs{d^{(j)}_{j,n}}\right)
\ge \frac {c^2 \sigma_{\min}^8 }3\ .
$$ 
Item (2) is proved.
\end{proof}

\begin{proof}[Proof of Lemma \ref{system-aux}-(3)]
We first introduce the following block-matrix notations:
$$
A_n^{(j)}=
\left(
\begin{array}{cc}
\bs{d_{j,n}^{(j)}} & \bar{a}_n^{(j)} \\
\bs{\bar{d}_{n}^{(j)}} & A_n^{(j+1)} \\
\end{array}
\right)\quad \textrm{and}\quad 
A_n^{0,(j)}=
\left(
\begin{array}{cc}
0 & \bar{a}_n^{(j)} \\
0 & A_n^{(j+1)} \\
\end{array}
\right)\ .
$$
We can now express the following inverse:  
$$
(I-A_n^{0,(j)})=
\left(
\begin{array}{cc}
1 & - \bar{a}_n^{(j)}\\
0 & (I-A_n^{(j+1)}) 
\end{array}\right)
\quad \textrm{as}\quad 
(I-A_n^{0,(j)})^{-1}=
\left(
\begin{array}{cc}
1 & \bar{a}_n^{(j)} (I-A_n^{(j+1)})^{-1}\\
0 & (I-A_n^{(j+1)})^{-1} 
\end{array}\right)\ .
$$
This in turn yields
$
\bs{y_{j,n}^{(j)}}= \bs{d_{j,n}^{(j)}} + \bar{a}_n^{(j)} (I-A_n^{(J+1)})^{-1}\bs{\bar{d}_{n}^{(j)}}
$ and one can easily check that $\bs{y_{j,n}^{(j)}}\le \frac Kn$, where $K$ does not depend on $j$ and $n$, as 
$$
|\bs{y_{j,n}^{(j)}}| \le |\bs{d_{j,n}^{(j)}}| + n\|
\bar{a}_n^{(j)}\|_{\infty} \leftrownorm (I-A_n^{0,(j)})^{-1}
\rightrownorm_{\infty} \| \bs{\bar{d}_{n}^{(j)}}\|_{\infty}\ .
$$
Remark that 
\begin{eqnarray*}
\lefteqn{\log\det(I-A_n^{(j)}) -\log\det(I-A_n^{(j+1)} )}\\
&=& \log\det\left(
\left[
\begin{array}{cc}
1-\bs{d_{j,n}^{(j)}} & -\bar{a}_n^{(j)}\\
-\bs{\bar{d}_{n}^{(j)}} & I - A_n^{(j+1)}
\end{array}\right]
\left[
\begin{array}{cc}
1 & 0\\
0 & (I-A_n^{(j+1)})^{-1}
\end{array}\right]\right)\\
&=& \log \det \left[ 
\begin{array}{cc}
1-\bs{d_{j,n}^{(j)}} & -\bar{a}_n^{(j)}(I-A_n^{(j+1)})^{-1}\\
-\bs{\bar{d}_{n}^{(j)}} & I
\end{array}
\right]\\
&=& \log\left(1- \bs{d_{j,n}^{(j)}} -\bar{a}_n^{(j)}(I-A_n^{(j+1)})^{-1}\bs{\bar{d}_{n}^{(j)}}\right)\ 
\end{eqnarray*}
and write $\log \det (I-A_n)$ as:
\begin{eqnarray*}
  \log \det (I-A_n) &=&\sum_{j=1}^{n-1} \left(\log \det (I-A_n^{(j)}) -\log\det(I-A_n^{(j+1)})\right) +\log(1-a_{nn})\\
  &=& \sum_{j=1}^{n-1} \log \left( 1- \bs{d_{j,n}^{(j)}} -\bar{a}_n^{(j)}(I-A_n^{(j+1)})^{-1}\bs{\bar{d}_{n}^{(j)}}\right)
+\log(1-a_{nn})\\
  &=& - \sum_{j=1}^{n-1} \left( \bs{d_{j,n}^{(j)}} + \bar{a}_n^{(j)}(I-A_n^{(j+1)})^{-1}\bs{\bar{d}_{n}^{(j)}}\right) + o\left(1\right)\\
  &=& - \sum_{j=1}^{n-1} \bs{y^{(j)}_{j,n}} + o\left(1\right)\quad 
=\quad -\sum_{j=1}^{n} \bs{y^{(j)}_{j,n}} + o\left(1\right)\\
 &=& - \check{\mathcal V}_n +o(1)\ .
\end{eqnarray*}
This concludes the proof of Lemma \ref{system-aux}.

\end{proof}

\section{Proof of Theorem \ref{th-clt}}
\label{proof-clt}
\subsection{More notations; outline of the proof; key lemmas}
\label{sec-outline-CLT}
\subsubsection*{More notations}
Recall that $Y_n=(Y_{ij}^n)$ is a $N\times n$ matrix where $Y_{ij}^n
=\frac{\sigma_{ij}}{\sqrt{n}} X_{ij}$ and that
$Q_n(z)=(q_{ij}(z)\,)=(Y_n Y_n^* -zI_N)^{-1}$. We denote
\begin{enumerate}
\item by $\ti Q_n(z)=(\ti q_{ij}(z)\,)=(Y_n^* Y_n -zI_n)^{-1}$, 
\item by $y_j$ the column number $j$ of $Y_n$,
\item by $Y_n^j$ the $N\times (n-1)$ matrix that remains after deleting 
column number $j$ from $Y_n$,
\item by $Q_{j,n}(z)$ (or $Q_j(z)$ for short when there is no confusion with
$Q_n(z)$) the $N \times N$ matrix 
$$
Q_j(z)=(Y^j Y^{j\, *} -zI_N)^{-1},
$$
\item by $\xi_i$ the row number $i$ of $Y_n$,
\item by $Y_{i,n}$ (or $Y_i$ for short when there is no confusion with $Y_n$) 
the $(N-1)\times n$ matrix that remains after deleting row $i$ from $Y$,
\item by $\ti Q_{i,n}(z)$ (or $\ti Q_i(z)$) the $n \times n$ matrix 
$$
\ti Q_i(z)=(Y_i^* Y_i -zI_n)^{-1}.
$$
\end{enumerate} 
Recall that we use both notations $q_{ij}$ or $[Q]_{ij}$ for the
individual element of $Q(z)$ depending on the context (same for other
matrices). The following formulas are well-known (see for instance
Sections 0.7.3 and 0.7.4 in \cite{HorJoh94}):
\begin{equation}
\label{inversion-lemma} 
Q = Q_j -\frac{Q_j y_j y_j^* Q_j}{1+y_j^* Q_j y_j}, \quad
\ti Q = \ti Q_i -\frac{\ti Q_i \xi_i^* \xi_i \ti Q_i}
{1+  \xi_i \ti Q_i \xi_i^*}
\end{equation} 
\begin{equation} 
\label{eq-qii} 
q_{ii}(z) = \frac{-1}{z(1+\xi_i \ti Q_i(z) \xi_i^*)}, \quad 
\ti q_{jj}(z) = \frac{-1}{z(1+y_j^* Q_j(z) y_j)}. 
\end{equation}

For $1\le j\le n$, denote by ${\mathcal F}_j$ the $\sigma$-field
${\mathcal F}_j = \sigma(y_j, \cdots, y_n)$ 
generated by the random vectors $(y_j,\cdots,y_n)$. Denote by
$\E_j$ the conditional expectation with respect to ${\mathcal
  F}_j$, i.e. $\E_j=\E(\cdot\mid {\mathcal F_j})$. By
  convention, ${\mathcal F}_{n+1}$ is the trivial $\sigma$-field; in particular,
$\E_{n+1}=\E$. 
\subsubsection*{Outline of the proof}
In order to prove the convergence of $\Theta_n^{-1} (\log \det (Y_n Y_n^* + \rho I_N) -
\E \log \det (Y_n Y_n^* + \rho I_N))$ toward the
standard gaussian law ${\mathcal N}(0,1)$, we shall rely on the
following CLT for martingales:
\begin{theo}[CLT for martingales, Th. 35.12 in \cite{Bil95}] 
\label{th-clt-martingales}
Let $\gamma^{(n)}_n, \gamma^{(n)}_{n-1}, \ldots, \gamma^{(n)}_1$ be a 
martingale difference sequence with respect to the increasing filtration 
${\mathcal F}^{(n)}_n, \ldots, {\mathcal F}^{(n)}_1$. 
Assume that there exists a sequence of real positive numbers 
${\Theta}_n^2$ such that 
\begin{equation}
\label{eq-cond-variance-clt} 
\frac{1}{\Theta_n^2}
\sum_{j=1}^n \E_{j+1} {\gamma_j^{(n)}}^2 \cvgP{n\to\infty} 1 \ .
\end{equation} 
Assume further that the Lindeberg condition holds: 
$$
\forall \epsilon > 0, \ 
\frac{1}{\Theta_n^2}
\sum_{j=1}^n \E\left( {\gamma_j^{(n)}}^2 
{\bf 1}_{\left| \gamma_j^{(n)} \right| \geq \epsilon \Theta_n } \right) 
\xrightarrow[n\rightarrow \infty]{} 0\ .
$$
Then $\Theta_n^{-1} \sum_{j=1}^n \gamma^{(n)}_j $ converges in 
distribution to 
${\mathcal N}(0,1)$. 
\end{theo}
\begin{rem} The following condition:
\begin{equation}
\label{eq-lyapounov}
\exists \delta > 0, \quad
\frac{1}{\Theta_n^{2(1+\delta)}}
\sum_{j=1}^n \E\left| {\gamma_j^{(n)}} \right|^{2+\delta} 
\xrightarrow[n\rightarrow \infty]{} 0\ ,
\end{equation} 
known as Lyapounov's condition implies Lindeberg's condition and is easier to establish 
(see for instance \cite{Bil95}, Section 27, page 362).
\end{rem}
The proof of the CLT will be carried out following three steps:
\begin{enumerate}
\item We first show that 
$\log \det (Y_n Y_n^* + \rho I) -
\E \log \det (Y_n Y_n^* + \rho I)$
can be written as 
$$
\log \det (Y_n Y_n^* + \rho I) -
\E \log \det (Y_n Y_n^* + \rho I)=\sum_{j=1}^n \gamma_j,
$$
where $(\gamma_j)$ is a martingale difference sequence. 
\item We then prove that $(\gamma_j)$ satisfies Lyapounov's condition 
(\ref{eq-lyapounov}) where $\Theta_n^2$ is given by
Theorem \ref{th-variance}. 
\item We finally prove (\ref{eq-cond-variance-clt}) which implies the CLT.
\end{enumerate}
\subsubsection*{Key Lemmas} The two lemmas stated below will be of constant use in the sequel.
The first lemma describes the asymptotic
behaviour of quadratic forms related to random matrices.

\begin{lemma}
\label{lm-approx-quadra}
Let $\bs{x}=(x_1,\cdots, x_n)$ be a $n \times 1$ vector where the $x_i$
are centered i.i.d.~complex random variables with unit variance. 
Let $M$ be a $n\times n$ 
deterministic complex matrix. 
\begin{enumerate}
\item (Bai and Silverstein, Lemma 2.7 in \cite{BaiSil98})  \label{silverstein} Then, for any $p\ge 2$, there exists a
constant $K_p$ for which 
$$
\E|\bs{x}^* M \bs{x} -\tr M|^p
\le K_p \left( \left( \E|x_1|^4 \tr M M^* \right)^{p/2} 
+ \E |x_1|^{2p} \tr (M M^*)^{p/2}
\right)\ .
$$
\item \label{identite-cumu} (see also Eq. (1.15) in \cite{BaiSil04})
  Assume moreover that $\mathbb{E}\, x_1^2 =0$ and that $M$ is real,
  then
$$
\mathbb{E}\left( \bs{x}^* M \bs{x} -\tr M \right)^2 = \tr M^2 +\kappa \sum_{i=1}^n m_{ii}^2\ ,
$$
where $\kappa= \mathbb{E}|x_1|^4-2$.
\end{enumerate}
\end{lemma} 
As a consequence of the first part of this lemma, there exists a
constant $K$ independent of $j$ and $n$ for which
\begin{equation}
\label{eq-approx-yi*Q*yi} 
\E\left| y_j^* Q_j(-\rho) y_j  - 
\frac{1}{n} \tr D_j Q_j(-\rho) \right|^p \leq K n^{-p/2}
\end{equation} 
for $p \le 4$. 


We introduce here various intermediate quantities: 
\begin{eqnarray}
c_i(z)&=& \frac{-1}{z\left( 1 +
\frac 1n \tr \ti D_i \E \ti Q(z)\right)},\ 1\le i\le N; \qquad 
C(z)=\diag(c_i(z);\ 1\le i\le N)\ , \nonumber \\
\ti c_j(z)&=& 
\frac{-1}{z\left( 1 +
\frac 1n \tr D_j \E Q(z)\right)},\ 1\le j\le n;\qquad 
\ti C(z)=\diag(\ti c_j(z);\ 1\le j\le n)\ , \nonumber \\
b_i(z)&=& \frac{-1}{z\left( 1 +\frac 1n \tr \ti D_i \ti C(z)\right)},\ 
1\le i\le N;\ \qquad B(z)=\diag(b_i(z);\ 1\le i\le N) \nonumber \\ 
\ti b_j(z)&=& \frac{-1}{z\left( 1 +\frac 1n \tr D_j C(z)\right)},\ 
1\le j\le n;\ 
\qquad \ti B(z)=\diag(\ti b_j(z);\ 1\le j\le n)\ . 
\label{eq-intermediate} 
\end{eqnarray}
The following lemma provides various bounds and approximation results.
\begin{lemma}
\label{lm-approximations-EQ-C-B} Consider the family of random matrices $(Y_n Y_n^*)$ and assume that
{\bf A-\ref{hypo-moments-X}} and {\bf A-\ref{hypo-variance-field}}
hold true. Let $z=-\rho$ where $\rho>0$. Then,
\begin{enumerate}
\item \label{superlemma-item1}
Matrices $C_n$ satisfy
$
\|C_n\| \leq \frac{1}{\rho}
\quad \textrm{and}\quad 0 <  c_{i} \leq \frac{1}{\rho} \ .
$  
These inequalities remain true when $C$ is replaced with $B$ or $\ti C$.  \\ 
\item \label{superlemma-item2} 
Let $U_n$ and $\ti U_n$ be two sequences of real diagonal deterministic 
$N\times N$ and $n \times n$ matrices.
Assume that 
$
\sup_{n\ge 1} \max\left( \| U_n\|, \|\ti U_n\|\right) <\infty\ ,
$ then the following hold true: \\
\begin{enumerate}
\item \label{eq-approx:EQ-T} 
$
\frac 1n \tr U(\E Q -T) \xrightarrow[n\rightarrow\infty]{} 0\ $
and $\ 
\frac 1n \tr 
\ti U(\E\ti Q -\ti T) \xrightarrow[n\rightarrow\infty]{} 0\ , 
$\\
\item \label{eq-approx:B-T} 
$
\frac 1n \tr U(B -T) \xrightarrow[n\rightarrow\infty]{} 0 \ ,\\
$
\item \label{eq-approx:(Q-EQ)2} 
$
\sup_n 
\E \left( \tr U \left( Q - \E Q \right) 
\right)^2 <  \infty
$,\\
\item \label{eq-approx:(Q-EQ)4} 
$
\sup_n  \frac 1{n^2} \E \left( \tr U \left( Q - 
\E Q \right) \right)^4 < \infty 
$, \\
\end{enumerate}
\item 
\label{superlemma-item:rank1} [Rank-one perturbation inequality]
The resolvent $Q_j$ satisfies 
$
\left| \tr M \left( Q - Q_j
\right) \right| 
\leq \frac{\left\| M \right\|}{\rho}
$
for any $N \times N$ matrix $M$ (see Lemma 2.6 in \cite{SilBai95}).
\end{enumerate}
\end{lemma}

Proof of Lemma \ref{lm-approximations-EQ-C-B} is postponed to Appendix 
\ref{anx-proof-lm-approximations-EQ-C-B}.

Finally, we shall frequently use the following identities which are obtained 
from the definitions of $c_i$ and $\ti c_j$ together with Equations 
(\ref{eq-qii}): 
\begin{eqnarray}
\label{eq-QiiX}
[Q(z)]_{ii} &=& c_i  + z c_i [Q]_{ii}
\left( \xi_i \tilde{Q}_i \xi_i^* - \frac{1}{n} \tr \tilde{D}_i \E \tilde{Q}
\right)  \\
\ \! [ \tilde{Q}(z) ]_{jj}  &=&
\tilde{c}_j  + z \tilde{c}_j [ \tilde{Q} ]_{jj}
\left( y_j^* {Q}_j y_j - \frac{1}{n} \tr {D}_j \E {Q}
\right) 
\label{eq-identite-tildeQX}
\end{eqnarray}

\subsection{Proof of Step 1: The sum of a martingale difference sequence}\label{step-one}
Recall that $\mathbb{E}_j=\mathbb{E}(\cdot \mid {\mathcal F}_j)$ where
${\mathcal F}_j=\sigma(y_{\ell},\ j\le \ell\le n)$. We have:
\begin{eqnarray*}
\lefteqn{\log \det (Y Y^* + \rho I_N) -
\E \log \det (Y Y^* + \rho I_N )}\\
&=& \sum_{j=1}^n 
(\E_j -\E_{j+1}) \log \det (Y Y^* + \rho I_N) \\
&\stackrel{(a)}{=}&- \sum_{j=1}^n 
\left( \E_j - \E_{j+1} \right) 
\log \left( 
\frac     
{ \det\left( Y^j {Y^j}^* + \rho I_N \right)  } 
{ \det\left( Y Y^* + \rho I_N \right)  } 
\right), \\ 
&\stackrel{(b)}{=}& 
- \sum_{j=1}^n 
\left( \E_j - \E_{j+1} \right) 
\log \left( 
\frac     
{ \det\left( {Y^j}^* Y^j + \rho I_{n-1} \right)  } 
{ \det\left( Y^* Y + \rho I_n \right)  } 
\right)\\
&\stackrel{(c)}{=}& - \sum_{j=1}^n \left( \E_j - \E_{j+1} \right) \log 
\left[ \tilde{Q}\left(-\rho\right) \right]_{jj}\ ,\\
&\stackrel{(d)}{=}& \sum_{j=1}^n \left( \E_j - \E_{j+1} \right) \log 
\left( 1 + y_j^* Q_j\left(-\rho\right) y_j\right) 
\end{eqnarray*}
where $(a)$ follows from the fact that $Y^j$ does not depend upon $y_j$, in particular 
$
\E_j \log \det\left( Y^j {Y^j}^* + \rho I_{N} \right)  
= 
\E_{j+1} \log \det\left( Y^j {Y^j}^* + \rho I_{N} \right) ;
$ $(b)$ follows from the fact that $\det(Y^{j*} Y^j +\rho I_{n-1})=\det (Y^j Y^{j*} +\rho I_N) \times \rho^{n-1-N}$
(and a similar expression for $\det(Y^* Y +\rho I_n)$); $(c)$ follows from the equality
$$
[\ti Q(-\rho)]_{jj} =
\frac{\det(Y^{j\,*} Y^j + \rho I_{n-1})}{\det (Y^* Y + \rho I_n)}
$$
which is a consequence of the general inverse formula
$A^{-1} =\frac 1{\det(A)} \mathrm{adj}(A)$ where 
$\mathrm{adj}(A)$ is the transposed matrix of cofactors of $A$ (see
Section 0.8.2 in \cite{HorJoh94}); and $(d)$ follows from  
(\ref{eq-qii}). We therefore have
\begin{eqnarray*}
\lefteqn{\log \det (Y Y^* + \rho I_N) -
\E \log \det (Y Y^* + \rho I_N )}\\
& =& \sum_{j=1}^n \left( \E_j - \E_{j+1} \right) \log 
\left( 1 + y_j^* Q_j\left(-\rho\right) y_j\right)\quad \stackrel{\triangle}{=}\quad  \sum_{j=1}^n \gamma_j\ .
\end{eqnarray*} 
As the following identity holds true, 
$$
\E_j \log\left( 1 + \frac{1}{n} \tr D_j Q_j \right) = 
\E_{j+1} \log\left( 1 + \frac{1}{n} \tr D_j Q_j \right)\ ,
$$
one can express $\gamma_j$ as:
\begin{eqnarray}
\gamma_j &=& \left( \E_j - \E_{j+1} \right) 
\log\left( 1 + 
\frac{y_j^* Q_j y_j - \frac 1n \tr D_j Q_j}
{1 + \frac 1n \tr D_j Q_j} 
\right) \nonumber \\
&=& \left( \E_j - \E_{j+1} \right) \log\left( 1 + \Gamma_j \right) \qquad \textrm{where}
\quad  \Gamma_j = 
\frac{y_j^* Q_j y_j - \frac 1n \tr D_j Q_j}
{1 + \frac 1n \tr D_j Q_j}\ . \label{def-Aj}
\end{eqnarray}
The sequence $\gamma_n, \ldots, 
\gamma_1$ is a martingale difference sequence with respect to the increasing
filtration ${\mathcal F}_n , \ldots, {\mathcal F}_1$ and Step 1 is established.

\subsection{Proof of Step 2: Validation of Lyapounov's condition
(\ref{eq-lyapounov})} 
\label{subsec-lyapounov} 
In the remainder of this section, $z=-\rho$. Let $\delta>0$ be a fixed positive number that will be specified below. As 
$\liminf \Theta_n^2 > 0$ by Theorem \ref{th-variance}, we only
need to prove that 
$\sum_{j=1}^n \E|\gamma_j|^{2+\delta} \to_n 0$. We have
$
\E|\gamma_j|^{2+\delta} =\E \left| (\E_j - \E_{j+1}) \log(1+\Gamma_j)\right|^{2+\delta}
$; Minkowski and Jensen inequalities yield:
\begin{eqnarray*}
\left(\E|\gamma_j|^{2+\delta}\right)^{\frac 1{2+\delta}} &\le &
\left( \E \left| \E_j \log(1+\Gamma_j)\right|^{2+\delta}\right)^{\frac 1{2+\delta}}
+ \left( \E \left| \E_{j+1} \log(1+\Gamma_j)\right|^{2+\delta} \right)^{\frac 1{2+\delta}}\\
&\le& 2 \left( \E \left|  \log(1+\Gamma_j)\right|^{2+\delta} \right)^{\frac 1{2+\delta}}.
\end{eqnarray*}
Otherwise stated, 
\begin{equation}\label{majo-lyapounov}
\E|\gamma_j|^{2+\delta} \le K_0\, \E \left|  \log(1+\Gamma_j)\right|^{2+\delta}
\end{equation}
where $K_0=2^{2+\delta}$. Since $y_j^* Q_j y_j \ge 0$, $\Gamma_j$ (defined in (\ref{def-Aj})) is lower bounded:
$$
\Gamma_j \ge \frac{-\frac 1n \tr D_j Q_j}{1+\frac 1n \tr D_j Q_j}.
$$ 
Now, since  
$$
0 \leq 
\frac{1}{n} \tr D_j Q_j\left( - \rho \right) 
\leq 
\frac{\| D_j \|}{n} \tr\, Q_j\left( - \rho \right) 
\leq K_1 \eqdef \frac{ \sigma_{\mathrm{max}}^2}{\rho} \sup_n \left(\frac{N}{n} \right) 
$$
and since $x\mapsto \frac{x}{1+x}$ is non-decreasing, we have:
\begin{equation}\label{defK2}
\frac{\frac{1}{n} \tr D_j Q_j}{1+\frac{1}{n} \tr D_j Q_j} 
\le K_2 \stackrel{\triangle}{=}\frac{K_1}{1+K_1} <1.
\end{equation}
In particular, $\Gamma_j\ge -K_2>-1.$ The function $(-1,\infty)\ni x\mapsto
\frac{\log(1+x)}{x}$ is non-negative, non-increasing. Therefore, 
$\frac {\log(1+x)}{x} \le \frac{\log(1-K_2)}{K_2} $ for $x\in [-K_2,\infty).$
Plugging this into (\ref{majo-lyapounov}) yields
\begin{eqnarray*}
\E |\gamma_j|^{2+\delta} &\le & K_0 K_2^{2+\delta}\, \E |\Gamma_j|^{2+\delta}\\
& \stackrel{\triangle}{=}& K_3\, \E |\Gamma_j|^{2+\delta}\ 
\le K_3\ \E\left| y_j^* Q_j y_j - \frac 1n \tr D_j Q_j \right|^{2+\delta}.
\end{eqnarray*}
By lemma \ref{lm-approx-quadra}-(\ref{silverstein}), the right hand side of the 
last inequality is
lower than $K_4\, n^{-(1+\delta/2)}$ as soon as 
$\E|X_{11}|^{2 + \delta} < \infty$. 
This is ensured by {\bf A-\ref{hypo-moments-X}} for 
$\delta\le 6$. Therefore, Lyapounov's condition \eqref{eq-lyapounov} holds and Step 2 is proved. 

\subsection{Proof of Step 3: Convergence of the normalized sum of 
conditional variances}
This section, by far the most involved in this article, is devoted to
establish the convergence (\ref{eq-cond-variance-clt}), hence the CLT.
In an attempt to guide the reader, we divide 
it into five stages. Recall that $z=-\rho$ and 
$$
\gamma_j = \left( \E_j - \E_{j+1} \right) 
\log\left( 1 + \Gamma_j \right)\quad \textrm{where}\quad \Gamma_j=
\frac{y_j^* Q_j y_j - \frac 1n \tr D_j Q_j}
{1 + \frac 1n \tr D_j Q_j}\ .
$$
In order to apply Theorem \ref{th-clt-martingales}, we shall prove that 
$\Theta_n^{-2} \sum_{j=1}^n \E_{j+1} \gamma_j^2 \cvgP{} 1$ 
where $\Theta_n^2$ is given by Theorem \ref{th-variance}. Since 
$\liminf \Theta_n^2 > 0$,
it is sufficient to establish the following convergence:
\begin{equation}
\label{eq-cvg-gamma-minus-theta} 
\sum_{j=1}^n \E_{j+1} \gamma_j^2 \ - \ \Theta^2_n
\ \cvgP{n\to\infty}  \ 0\ . 
\end{equation} 
Instead of working with $\Theta_n$, we shall work with $\tilde \Theta_n$ 
(introduced in Section \ref{proof-variance}, see Eq. \eqref{thetatilde}) and prove:
\begin{equation}
\label{conv-theta-tilde} 
\sum_{j=1}^n \E_{j+1} \gamma_j^2 \ - \ \tilde \Theta^2_n
\ \cvgP{n\to\infty}  \ 0\ . 
\end{equation} 
In the sequel, $K$ will denote a constant whose value may change
from line to line but which will neither depend on $n$ nor on $j\le n$.

Here are the main steps of the proof:
\begin{enumerate}
\item  The following convergence holds true:
\begin{equation}
\label{eq-cvg-proba-gamma-A}
\sum_{j=1}^n 
\E_{j+1} \gamma_j^2 -
\sum_{j=1}^n \E_{j+1} \left( \E_j \Gamma_j \right)^2
\cvgP{n\to\infty} 0 \ . 
\end{equation}
This convergence roughly follows from a first order approximation, as
we informally discuss: Recall that $\gamma_i= (\E_j -\E_{j+1}) \log
(1+\Gamma_j)$ and that $\Gamma_j \rightarrow 0$ by Lemma \ref{lm-approx-quadra}-(\ref{silverstein}). A
first order approximation of $\log(1+x)$ yields $\gamma_j \approx
(\mathbb{E}_j -\mathbb{E}_{j+1})\Gamma_j$. As $\mathbb{E}_{j+1} (y_j^* Q_j
y_j) = \frac 1n \tr D_j \mathbb{E}_{j+1} Q_j$, one has
$\mathbb{E}_{j+1} \Gamma_j=0$, hence $\gamma_j \approx \mathbb{E}_j \Gamma_j$
and one may expect $\mathbb{E}_{j+1} \gamma_j^2 \approx
\mathbb{E}_{j+1}(\mathbb{E}_j \Gamma_j)^2$ and even
\eqref{eq-cvg-proba-gamma-A} as we shall demonstrate.
\\
\item Recall that $\kappa= \E|X_{11}|^4 -2$. The following equality holds true 
\begin{multline}
\label{eq-Ej+1(EjAj)^2}
\E_{j+1} \left( \E_j \Gamma_j \right)^2 = 
\frac{1}{n^2 \left(1+\frac 1n \tr D_j \E Q \right)^2} 
\left( \tr D_j \left( \E_{j+1} Q_j \right) 
D_j \left( \E_{j+1} Q_j \right) \phantom{\sum_{i=1}^N } \right.\\
+ \left. \kappa \sum_{i=1}^N \sigma^4_{ij} \left( \E_{j+1} [Q_j]_{ii} \right)^2 
\right) + \boldsymbol{\varepsilon_{2,j}}
\end{multline}
where 
$$
\max_{j\le n} \E | \boldsymbol{\varepsilon_{2,j}} | 
\le \frac{K}{n^{3/2}}
$$ for some given $K$. 
\end{enumerate}
A closer look to the right hand side of \eqref{eq-Ej+1(EjAj)^2} yields
the following comments: By Lemma
\ref{lm-approximations-EQ-C-B}-(\ref{eq-approx:EQ-T}), the denominator
$(1 + \frac 1n \tr D_j \E Q)^2$ can be approximated by $(1 + \frac 1n
\tr D_j T)^2$; moreover, it is possible to prove that $[ Q_j ]_{ii} \approx [ T
]_{ii}$ (some details are given in the course of the proof of step (5)
below). Hence,
\[ 
 \frac{\kappa}n \sum_{i=1}^N \sigma^4_{ij} 
\left( \E_{j+1} [Q_j]_{ii} \right)^2 
\approx 
 \frac {\kappa}n \tr D_j^2 T^2  \ . 
\] 
Therefore
it remains to study the asymptotic behaviour of the term
$\frac 1n \tr D_j (\E_{j+1} Q_j) D_j (\E_{j+1} Q_j )$ in order to understand \eqref{eq-Ej+1(EjAj)^2}. This is the
purpose of step (3) below.
  
\begin{enumerate}
\item[(3)] In order to evaluate 
$\frac 1n \tr D_j (\E_{j+1} Q_j) D_j (\E_{j+1} Q_j )$ 
for large $n$, we introduce the random variables  
\begin{equation}\label{DefChi}
\bs{\chi_{\ell, n}^{(j)}} =\frac 1n \tr D_{\ell} (\E_{j+1} Q) D_j Q,\quad  j\le \ell\le n.
\end{equation}
Note that, up to rank-one perturbations,  
$\E_j \bs{\chi_{j,n}^{(j)}}$ is very close to the term of interest. We prove here that   
$\bs{\chi_{\ell, n}^{(j)}}$ satisfies the following equation: 
\begin{equation}
\label{eq-xi-lj}
\boldsymbol{\chi_{\ell, n}^{(j)}} = 
\frac{1}{n} 
\sum_{m=j+1}^n 
\frac{\frac{1}{n} \tr (D_\ell B D_m \E Q) }{\left( 1 + \frac{1}{n} \tr D_m \E Q \right)^2} 
 \boldsymbol{\chi_{m, n}^{(j)} }
+  \frac 1n \tr D_\ell B D_j \E Q + \boldsymbol{\varepsilon_{3,\ell j}} ,\quad j\le \ell\le n\ ,
\end{equation}
where $B$ is defined in Section \ref{sec-outline-CLT} and where 
$$
\max_{\ell, j \le n}  \E | \boldsymbol{\varepsilon_{3,\ell j}} | \le \frac K{\sqrt{n}}\ .
$$

\item[(4)] Recall that we have proved in Section \ref{proof-variance}
  (Lemma \ref{system-aux}) that the following (deterministic) system:
$$
\bs{y_{\ell,n}^{(j)}}=
\sum_{m=j+1}^n  a_{\ell, m}\  
\bs{y_{m,n}^{(j)}}
+ a_{\ell, j} \quad \textrm{for}\quad  j\le \ell\le n\ ,
$$
where $ a_{\ell, m} = \frac{1}{n^2} \frac{\tr D_{\ell} D_m T^2}
{\left(1 + \frac{1}{n} \tr D_{\ell} T  \right)^2}$ admits a unique solution. Denote by 
$\bs{x_{\ell, n}^{(j)}}=n\left(1 +\frac 1n \tr D_{\ell} T\right)^2 \bs{y_{\ell,n}^{(j)}}$, then 
$(\bs{x_{\ell, n}^{(j)}},\ j\le \ell\le n)$ readily satisfies the following system:
$$
\bs{x_{\ell, n}^{(j)}} =\frac 1n \sum_{m=j+1}^n  \frac{\frac 1n \tr D_{\ell} D_m T^2}{\left( 1 +\frac 1n \tr D_m T \right)^2}
\bs{x_{m, n}^{(j)}} + \frac 1n \tr D_{\ell} D_j T^2,\quad j\le \ell \le n\ .
$$
As one may
notice, \eqref{eq-xi-lj} is a perturbated version of the system above
and we shall indeed prove that:
\begin{equation}
\label{eq-zeta_jj-x_jj}
\bs{\chi_{j,n}^{(j)}} = \bs{x_{j,n}^{(j)}} + \boldsymbol{\varepsilon_{41,j}} + \boldsymbol{\varepsilon_{42,j}}\quad \textrm{where}\quad
\max_{j\le n} \E  |\boldsymbol{\varepsilon_{41,j}}|\le \frac K{\sqrt{n}}\quad
\textrm{and}\quad \max_{j\le n}|\boldsymbol{\varepsilon_{42,j}}|\le \delta_n,
\end{equation}
the sequence $(\delta_n)$ being deterministic with $\delta_n \rightarrow 0$ as $n\rightarrow \infty$.\\

\item[(5)] 
Combining the previous results, we finally prove that
\begin{equation}
\label{eq-Ej+1(EjAj)^2-theta}
\sum_{j=1}^n \E_{j+1}(\E_j \Gamma_j)^2 \ - \ \tilde \Theta^2_n 
\ \cvgP{n\to\infty} \ 0 \ . 
\end{equation}
This, together with (\ref{eq-cvg-proba-gamma-A}), yields convergence
\eqref{conv-theta-tilde} and (\ref{eq-cvg-gamma-minus-theta}) which in
turn proves (\ref{eq-cond-variance-clt}), ending the proof of Theorem
\ref{th-clt}.

\end{enumerate}

\begin{proof}[Proof of (\ref{eq-cvg-proba-gamma-A})] Recall that
$
\frac{\frac{1}{n} \tr D_j Q_j}{1+\frac{1}{n} \tr D_j Q_j} 
\le K_2 <1\ 
$
by \eqref{defK2}. In particular, $\Gamma_j\ge -K_2>-1.$ We first prove that 
$$
\E_j \log(1+\Gamma_j) = 
\E_j \Gamma_j + \boldsymbol{\varepsilon_{11,j}} + \boldsymbol{\varepsilon_{12,j}}
$$ 
where
\begin{equation}
\left\{
\begin{array}{ll}
\boldsymbol{\varepsilon_{11,j}} &= 
\E_j \log(1+\Gamma_j) {\bs 1}_{|\Gamma_j|\le K_2}- \E_j \Gamma_j \\ 
\boldsymbol{\varepsilon_{12,j}} &= 
\E_j  \log(1+\Gamma_j) {\bs 1}_{(K_2, \infty)}(\Gamma_j) 
\end{array}\right. \ 
\textrm{and} \quad\left\{ 
\begin{array}{l}
 \max_{j\le n} \E\,\boldsymbol{\varepsilon_{11,j}}^2 
 \le \frac{K}{n^2}\\ 
 \max_{j\le n} \E\,\boldsymbol{\varepsilon_{12,j}}^2 
 \le \frac{K}{n^2} 
\end{array}\right. .\label{inequality-epsilon-1}
\end{equation}
In the sequel, we shall often omit subscript $j$ while dealing with the
$\boldsymbol{\varepsilon}$'s. As $0<K_2<1$, we have:
\begin{eqnarray*}
\left| \boldsymbol{\varepsilon_{11}} \right| &=& \left| \E_j \left( 
\sum_{k=1}^\infty  \frac{(-1)^{k-1}}{k} \Gamma_j^k 
{\bs 1}_{|\Gamma_j|\le K_2}
\ - \ \Gamma_j \right) \right| \ ,\\
&\leq& \E_j \Gamma_j {\bs 1}_{\Gamma_j>K_2}+
\sum_{k=2}^\infty \E_j \left| \Gamma_j \right|^k 
{\bs 1}_{|\Gamma_j|\le K_2} \quad
\leq\quad \E_j \Gamma_j {\bs 1}_{\Gamma_j>K_2} +
\frac{\E_j \Gamma_j^2 {\bs 1}_{|\Gamma_j|\le K_2}}{1-K_2}\ .
\end{eqnarray*} 
Therefore,  
\begin{eqnarray*}
\E\, \boldsymbol{\varepsilon_{11}}^2 
&\stackrel{(a)}{\le}&2\left(  \E \Gamma_j^2 {\bs 1}_{\Gamma_j>K_2} + 
\frac{ \E \Gamma_j^4 
{\bs 1}_{|\Gamma_j|\le K_2}}{(1-K_2)^2} \right)\\
&\stackrel{(b)}{\le}& \frac{2\E \Gamma_j^4}{K_2^2} + 
\frac{ 2 \E \Gamma_j^4}{(1-K_2)^2} \\
&\stackrel{(c)}{\leq}& \left( \frac{2}{K_2^2} + \frac{2}{(1-K_2)^2} \right) 
\E \left( y_j^* Q_j y_j - \frac 1n \tr D_j Q_j \right)^{4}  
\quad \stackrel{(d)}{\le}\quad \frac{K}{n^2}
\end{eqnarray*} 
where $(a)$ follows from $(a+b)^2 \le 2(a^2 + b^2)$, $(b)$ from the inequality 
$\Gamma_j^2 {\bs 1}_{\Gamma_j>K_2} \le \Gamma_j^2 \left(\frac{\Gamma_j}{K_2}\right)^2 {\bs 1}_{\Gamma_j>K_2}$,
$(c)$ from the fact that the denominator of $\Gamma_j$ is larger than one, and $(d)$  
from Lemma \ref{lm-approx-quadra}-(\ref{silverstein}) as $X_{11}$ has 
a finite $8^{\textrm{th}}$ moment by {\bf A-\ref{hypo-moments-X}}.

Now, $0\le \boldsymbol{\varepsilon_{12}} 
\le \E_j \Gamma_j {\bs 1}_{\Gamma_j>K_2}$. Thus, $\mathbb{E} \boldsymbol{\varepsilon_{12}}^2
\le \E \Gamma_j^2 {\bs 1}_{\Gamma_j>K_2}\le K_2^{-2} \E \Gamma_j^4 {\bs 1}_{\Gamma_j>K_2}$. Lemma
\ref{lm-approx-quadra}-(\ref{silverstein}) yields again:
$$
\E \,\boldsymbol{\varepsilon_{12}}^2  \le \frac{K}{n^2}
$$ 
and (\ref{inequality-epsilon-1}) is proved. Similarly, 
we can prove
\[ 
\E_{j+1} \log(1+\Gamma_j) = \E_{j+1} \Gamma_j + \boldsymbol{\varepsilon_{13,j}}\quad 
\textrm{with}\quad \max_{j\le n} 
\E\,\boldsymbol{\varepsilon_{13,j}}^2 \le \frac{K}{n^2} \ . 
\] 
Note that since $y_j$ and 
${\mathcal F}_{j+1}$ are independent, we have
$\mathbb{E}_{j+1} (y_j^* Q_j
y_j) = \frac 1n \tr D_j \mathbb{E}_{j+1} Q_j$ which yields $\E_{j+1} \Gamma_j =0$.  
Gathering all the previous estimates, we obtain:
\begin{eqnarray*}
\gamma_j &=&\E_j \Gamma_j + \boldsymbol{\varepsilon_{11,j}} + \boldsymbol{\varepsilon_{12,j}} - \boldsymbol{\varepsilon_{13,j}}\\
&\stackrel{\triangle}{=}& \E_j \Gamma_j + \boldsymbol{\varepsilon_{14,j}}\ ,
\end{eqnarray*}
where $\max_{j\le n}\E\, \bs{\varepsilon_{14,j}}^2 \le K\,n^{-2}$ 
by Minkowski's inequality. We therefore have $\E_{j+1}(\gamma_j)^2 = \E_{j+1} ( \E_j \Gamma_j + 
\bs{\varepsilon_{14,j}} )^2$. Let 
$$
\bs{\varepsilon_{1,j}} \eqdef 
\E_{j+1}(\gamma_j)^2 - \E_{j+1}(\E_j \Gamma_j)^2 = 
\E_{j+1} \bs{\varepsilon_{14,j}}^2 + 
2 \E_{j+1} ( \bs{\varepsilon_{14,j}} \, \E_j \Gamma_j ).
$$ 
Then
\begin{eqnarray*}
\E | \bs{\varepsilon_{1,j}} | &\leq& 
\E \bs{\varepsilon_{14,j}}^2 + 
2 \E | \bs{\varepsilon_{14,j}} \, \E_j \Gamma_j | \ , \\
&\stackrel{(a)}\le&
\E \bs{\varepsilon_{14,j}}^2 + 
2 ( \E \bs{\varepsilon_{14,j}}^2 )^{1/2} ( \E \Gamma_j^2 )^{1/2}\quad 
\stackrel{(b)}{\leq}\quad \frac{K}{n^{3/2}}\ ,
\end{eqnarray*}
where $(a)$ follows from Cauchy-Schwarz inequality and $(\E_j \Gamma_j)^2
\leq \E_j \Gamma_j^2$, and $(b)$ follows from Lemma \ref{lm-approx-quadra}-(\ref{silverstein})
which yields $\E \Gamma_j^2 \leq Kn^{-1}$.  Finally, we have $\sum_{j=1}^n \E |
\E_{j+1}(\gamma_j)^2 - \E_{j+1}(\E_j \Gamma_j)^2 | \leq K n^{-\frac 12}$ which
implies (\ref{eq-cvg-proba-gamma-A}).
\end{proof}

\begin{proof}[Proof of (\ref{eq-Ej+1(EjAj)^2})] 
We have: 
\begin{eqnarray*} 
\E_{j} \Gamma_j &=&  
\E_{j} \left( \frac{y_j^* Q_j y_j - 
\frac 1n \tr D_j Q_j}{1+\frac 1n \tr D_j Q_j}
\right)\ ,\\ 
&=& \frac{1}{1+\frac 1n \tr D_j \E Q} \Bigg\{ 
\E_{j} \left( y_j^* Q_j y_j - \frac 1n \tr D_j Q_j \right) 
\phantom{\frac{\frac 1n}{\frac 1n}}  \\
& & \phantom{\frac{1}{1+\frac 1n \tr D_j \E Q}}
- \E_{j} \left(  
\frac{y_j^* Q_j y_j - \frac 1n \tr D_j Q_j}
{1+\frac 1n \tr D_j Q_j} \left( \frac 1n \tr D_j Q_j - 
\frac 1n \tr D_j \E Q \right) \right) \Bigg\} \ .
\end{eqnarray*} 
Hence,
\begin{eqnarray}
\label{eq-(EjAj)^2-extended}
\E_{j+1} (\E_{j} \Gamma_j)^2 &=&  
\frac{1}{(1+\frac 1n \tr D_j \E Q)^2} 
\E_{j+1}
\left( 
\left( y_j^* (\E_j Q_j) y_j - \frac 1n \tr D_j \E_j Q_j \right)^2 
+ 
\bs{\varepsilon_{21,j}} + 
\bs{\varepsilon_{22,j}} \right)\nonumber \\
&=&  \frac{1}{(1+\frac 1n \tr D_j \E Q)^2} 
\E_{j+1}
\left( y_j^* (\E_j Q_j) y_j - \frac 1n \tr D_j \E_j Q_j \right)^2 
+\bs{\varepsilon_{2,j}}
\end{eqnarray}
where
\begin{eqnarray*}
\bs{\varepsilon_{21,j}} &=&
\left[ 
\E_{j} \left(  
\frac{y_j^* Q_j y_j - \frac 1n \tr D_j Q_j}
{1+\frac 1n \tr D_j Q_j} \left( \frac 1n \tr D_j Q_j - 
\frac 1n \tr D_j \E Q \right) \right) 
\right]^2\ , \\ 
\bs{\varepsilon_{22,j}} &=&
-2 \ \E_{j} 
\left( y_j^* Q_j y_j - \frac 1n \tr D_j Q_j \right)\times  \\
& & 
\phantom{-2 \ \E_{j} 
\left( y_j^* Q_j y_j\right)}
 \E_{j} \left(  
\frac{y_j^* Q_j y_j - \frac 1n \tr D_j Q_j}
{1+\frac 1n \tr D_j Q_j} \left( \frac 1n \tr D_j Q_j - 
\frac 1n \tr D_j \E Q \right) \right) \ ,\\
\bs{\varepsilon_{2,j}} &=& \frac{\E_{j+1} ( \bs{\varepsilon_{21,j}} +
\bs{\varepsilon_{22,j}})}{(1+\frac 1n \tr D_j \E
  Q)^2} \ .
\end{eqnarray*}
As
$\frac 1n \tr D_j Q_j \ge 0$, 
standard inequalities yield:
$$
\E \bs{\varepsilon_{21,j}} \leq 
\left[ \E
\left( y_j^* Q_j y_j - \frac 1n \tr D_j Q_j \right)^4 \right]^{\frac 12} 
\left[ \E \left( \frac 1n \tr D_j Q_j - \frac 1n \tr D_j \E Q \right)^4 
\right]^{\frac 12}\ . 
$$
By Lemma \ref{lm-approx-quadra}-(\ref{silverstein}),  $\E
\left( y_j^* Q_j y_j - \frac 1n \tr D_j Q_j \right)^4 
\le K n^{-2}$. Due to the convex inequality $(a+b)^4\le 2^3(a^4 +b^4)$, we obtain:
\begin{eqnarray*}
\E \left( \frac 1n \tr D_j(Q_j-\E Q)\right)^4 &=& 
\E \left( \frac 1n \tr D_j(Q_j-\E Q_j)+ \frac 1n \tr D_j(\E Q_j-\E Q)\right)^4\\
&\le& K\left\{ \E \left( \frac 1n \tr D_j(Q_j-\E Q_j)\right)^4
+ \E\left( \frac 1n \tr D_j(\E Q_j-\E Q)\right)^4\right\}\ ,
\end{eqnarray*}
where the first term of the right hand side is bounded by $K n^{-2}$
by \eqref{eq-approx:(Q-EQ)4} in Lemma \ref{lm-approximations-EQ-C-B}
and the second one is bounded by $K n^{-4}$ due to the rank-one
perturbation inequality [Lemma
\ref{lm-approximations-EQ-C-B}-(\ref{superlemma-item:rank1})].
Therefore $\E \bs{\varepsilon_{21,j}} \le K n^{-2}$ and similar
derivations yield $\E | \bs{\varepsilon_{22,j}} | \le K n^{-\frac
  32}$. Gathering these two results yields the bound $\E |
\bs{\varepsilon_{2,j}} | \leq K n^{-\frac 32}$. Let us now
expand the term $\E_{j+1} \left(
  y_j^* \E_j Q_j y_j -\frac 1n \tr D_j \E_j Q_j \right)^2$ in the
right hand side of (\ref{eq-(EjAj)^2-extended}).

Recall that $\E_j Q_j =\E_{j+1} Q_j$ and that $y_j = D_j^{\frac 12}
\left(\frac{X_{1j}}{\sqrt{n}}, \cdots,
  \frac{X_{Nj}}{\sqrt{n}}\right)^T$. Note also that $\E_{j+1}\left(y_j^* \E_j Q_j
  y_j\right)= \frac 1n \tr D_j \E_{j+1} Q_j$. Then Lemma \ref{lm-approx-quadra}-(\ref{identite-cumu})
immediatly yields:
\begin{eqnarray*} 
\lefteqn{\E_{j+1} \left( y_j^* \E_j Q_j y_j - \frac 1n \tr D_j \E_j Q_j \right)^2 }\\
&=&
\frac{1}{n^2} \left( \tr D_j \left( \E_{j+1} Q_j \right) 
D_j \left( \E_{j+1} Q_j \right) 
+ \kappa \sum_{\ell=1}^N \sigma^4_{\ell j} \left( \E_{j+1} [Q_j]_{\ell \ell} \right)^2 
\right) 
\ .
\end{eqnarray*}
Equation (\ref{eq-Ej+1(EjAj)^2}) is proved. 
\end{proof}

\begin{proof}[Proof of (\ref{eq-xi-lj})] 
Recall that $\bs{\chi_{\ell, n}^{(j)}}= \frac 1n \tr D_{\ell} (\E_{j+1} Q) D_j Q$.
The outline of the proof of (\ref{eq-xi-lj}) is given by the following set 
of equations, the ${\bs \chi}$'s and $\boldsymbol{\varepsilon}$'s being introduced 
as and when required.
\begin{alignat}{11}
\bs{\chi_{\ell, n}^{(j)}}= &{\bs \chi_1} +&{\bs \chi_2} -& {\bs \chi_3} &&&&&&&\label{eq:1}\\
&&&{\bs \chi_3} = &{\bs \chi'_3} +&\boldsymbol{\varepsilon}_{3}&&&&&\label{eq:2}\\
&&&\phantom{\chi_3 =} &{\bs \chi'_3} = &\ {\bs \chi_4} +&{\bs \chi_5} +& \boldsymbol{\varepsilon}'_{3}&&&\label{eq:3}\\
&&&&&&{\bs \chi_5} = &{\bs \chi_6} -&{\bs \chi_7} &+ \boldsymbol{\varepsilon}_{6} &- \boldsymbol{\varepsilon}_{7}\label{eq:4}
\end{alignat}
Gathering the previous equations, we will end up with 
\begin{equation}
\label{eq-xi-lj-sum}
\bs{\chi_{\ell, n}^{(j)}} = {\bs \chi_1} +{\bs \chi_2} -{\bs \chi_4} - {\bs \chi_6} +{\bs \chi_7} +\boldsymbol{\varepsilon}\quad \textrm{where}\quad
\boldsymbol{\varepsilon}= -\boldsymbol{\varepsilon}_3 +\boldsymbol{\varepsilon}'_3 
-\boldsymbol{\varepsilon}_6 +\boldsymbol{\varepsilon}_7.
\end{equation} 
Let us first give decomposition (\ref{eq:1}) and introduce ${\bs \chi_1}$, ${\bs \chi_2}$ 
and ${\bs \chi_3}$. Recall that $B$ (defined in Section \ref{sec-outline-CLT}) is the $N\times N$ diagonal matrix
$B=\mathrm{diag}(b_i)$ where 
$b_i=\left(\rho(1+\frac 1n \tr \ti D_i \ti C)\right)^{-1}$. 
The following identity yields:   
\begin{eqnarray*}
Q&=& B+B(B^{-1} -Q^{-1})Q \\
&=& B + B\left(\rho \, 
\mathrm{diag}\left( \frac 1n \tr \ti D_i \ti C\right) 
-Y Y^*\right)Q.
\end{eqnarray*}
Therefore, 
\begin{eqnarray*}
\bs{\chi_{\ell, n}^{(j)}} &=& \frac 1n \tr D_{\ell} (\E_{j+1} Q) D_j Q\\
&=&  \frac 1n \tr D_{\ell} B D_j Q + 
 \frac {\rho}n \tr D_{\ell} B 
\diag \left( \frac{1}{n} \tr \tilde D_i \tilde C\right) 
\left( \E_{j+1} Q \right) D_j Q \\
&& \qquad \qquad \qquad - \frac 1n \tr D_{\ell} B \left( 
\sum_{m=1}^n \E_{j+1} y_m y_m^* Q \right) D_j Q \\
&\stackrel{\triangle}{=}& {\bs \chi_1} +{\bs \chi_2} -{\bs \chi_3}\ ,
\end{eqnarray*}
and (\ref{eq:1}) is established. We now turn to 
decomposition (\ref{eq:2}). Identities (\ref{inversion-lemma}) and 
(\ref{eq-qii}) yield:
$$
y^*_m Q = y^*_m Q_m - y^*_m \frac{Q_m y_m y_m^* Q_m}{1+y_m^* Q_m y_m} \ =\ \frac{y^*_m Q_m}{1+y^*_m Q_m y_m}
= \rho [\ti Q]_{m m} y^*_m Q_m\ .
$$
Using this equation, we have 
\begin{eqnarray*}
{\bs \chi_3 }&=& \frac 1n \tr D_{\ell} B \left( \sum_{m=1}^n \E_{j+1} y_m y_m^* Q\right) D_j Q \\
      &=& \frac{\rho}{n} \tr D_{\ell} B \left( 
\sum_{m=1}^n \E_{j+1} \left( [\ti Q]_{m m} y_m y_m^* Q_m \right) \right)D_j Q\\
&\stackrel{(a)}{=}& \frac{\rho}{n} \tr D_{\ell} B \left( 
\sum_{m=1}^n \tilde{c}_m \E_{j+1} \left( y_m  y_m^* Q_m \right) \right)D_j Q
\\
& & 
- \frac{\rho^2}{n} \tr D_{\ell} B \left( 
\sum_{m=1}^n \tilde{c}_m 
\E_{j+1} \left( [\ti Q]_{m m} 
\left( y_m^* Q_m y_m - \frac 1n \tr D_m \E Q \right) 
y_m  y_m^* Q_m \right) \right)D_j Q \\ 
&\stackrel{\triangle}{=}& {\bs \chi'_3} + \boldsymbol{\varepsilon}_{3}\ ,
\end{eqnarray*}
where $(a)$ follows directly from \eqref{eq-identite-tildeQX}. 
We are now in position to prove that :
\begin{equation}
\label{eq-epsilon3}
\max_{\ell, j\le n} \E |\boldsymbol{\varepsilon}_{3}| \le \frac K{\sqrt{n}}
\ . 
\end{equation}
Using the fact that $|\tr A y y ^* B | = |y^* B A y| \le \| A B\|\, \|y\|^2$ together with the norm inequality 
$\| A B\| \le \|A \|\, \|B\|$, we obtain:  
\begin{eqnarray*}
|\boldsymbol{\varepsilon}_{3}| &\le & 
\frac{\rho^2}{n} \sum_{m=1}^n \| D_j Q D_{\ell} B\| 
\ti c_m \E_{j+1} \left( [\ti Q]_{m m} \left| y^*_m Q_m y_m 
-\frac 1n \tr D_m \E Q \right| \| y_m \|^2 \| Q_m \| \right),\\
&\stackrel{(a)}{\le} & 
\frac{\sigma_{\max}^4}{\rho^3} 
\frac 1n \sum_{m=1}^n 
\E_{j+1} \left( \left| y^*_m Q_m y_m 
-\frac 1n \tr D_m \E Q \right| \| y_m \|^2 \right),
\end{eqnarray*}
where $(a)$ follows from the fact that $\|D_j Q D_{\ell} B\| \ti c_m
\le \sigma_{\max}^4 \rho^{-3}$ and $[\ti Q]_{mm} \| Q_m \| \leq
\rho^{-2}$.  Writing $\frac{1}{n} \tr D_m \E Q = \frac 1n \tr D_m Q +
\frac 1n \tr D_m (\E Q -Q)$ and replacing in the previous inequality,
we obtain:
\begin{multline*}
\E \left( \left|y^*_m Q_m y_m -\frac 1n \tr D_m \E Q \right|
\| y_m \|^2 \right) \\
\le \left( \left[ 
\E \left|y^*_m Q_m y_m -\frac 1n \tr D_m Q \right|^2 \right]^{\frac 12}
+ \left[ \E \left| \frac 1n \tr D_m (Q-\E Q) \right|^2 \right]^{\frac 12}
\right)
\left( \E \| y_m \|^4 \right)^{\frac 12}\ ,
\end{multline*}
where $\left[ \E \left|y^*_m Q_m y_m -\frac 1n \tr D_m Q \right|^2
\right]^{\frac 12} \le K n^{-\frac 12}$ by Lemma
\ref{lm-approx-quadra}-(\ref{silverstein}) combined with Lemma
\ref{lm-approximations-EQ-C-B}-(\ref{superlemma-item:rank1}), $\left[ \E
  \left| \frac 1n \tr D_m (Q-\E Q) \right|^2 \right]^{\frac 12} \le
Kn^{-1}$ by Lemma
\ref{lm-approximations-EQ-C-B}-(\ref{eq-approx:(Q-EQ)2}) and $\E \|y_m
\|^4 \le \sigma_{\max}^4 \E | X_{11} |^4 (Nn^{-1})^2$.  This in
particular yields $\max_{\ell, j\le n} \E
|\boldsymbol{\varepsilon}_{3}| \le K n^{-\frac 12}$ and proves
(\ref{eq-epsilon3}).

Recall that if $m\le j$, then 
$$\mathbb{E}_{j+1}(y_m y_m^* Q_m)=\mathbb{E}_{j+1}(y_m y_m^*) \mathbb{E}_{j+1}(Q_m)=\frac{D_m}n \mathbb{E}_{j+1} (Q_m).
$$
We now turn to Equation (\ref{eq:3}) and introduce ${\bs \chi_4}$, ${\bs \chi_5}$ and 
$\boldsymbol{\varepsilon}_{3}'$. 
\begin{eqnarray*}
{\bs \chi'_3} & = & \frac{\rho}{n} \tr D_{\ell} B \left( 
\sum_{m=1}^n \tilde{c}_m \E_{j+1} \left( y_m  y_m^* Q_m \right) \right)D_j Q \ , \\
&=& \frac{\rho}{n^2} \tr D_{\ell} B \left( 
\sum_{m=1}^j \tilde{c}_m D_m \E_{j+1} Q_m  \right)D_j Q\\
&&\qquad + \frac{\rho}{n} \tr D_{\ell} B \left( 
\sum_{m=j+1}^n \tilde{c}_m \E_{j+1} \left( y_m  y_m^* Q_m \right) \right)D_j Q \ ,\\
&=& \frac{\rho}{n^2} \tr D_{\ell} B \left( 
\sum_{m=1}^j \tilde{c}_m D_m \E_{j+1} Q  \right)D_j Q\\
&&\qquad + \frac{\rho}{n} \tr D_{\ell} B \left( 
\sum_{m=j+1}^n \tilde{c}_m \E_{j+1} \left( y_m  y_m^* Q_m \right) \right)D_j Q\ ,\\
&& \qquad \qquad + \frac{\rho}{n^2} \tr D_{\ell} B \left( 
\sum_{m=1}^j \tilde{c}_m D_m \E_{j+1} \left( Q_m-Q \right) 
\right)D_j Q
\quad \stackrel{\triangle}{=}\quad {\bs\chi_4} + {\bs \chi_5} +\bv_3' \ ,
\end{eqnarray*}
and decomposition (\ref{eq:3}) is introduced. In order to estimate
$\bv_3'$, recall that given two square matrices $R$ and $S$, one has
$|\tr RS | \le \|R\| \tr S$ for $S$ non-negative and Hermitian. As matrix
$Q_m - Q$ is non-negative and hermitian by \eqref{inversion-lemma}, we obtain:
\begin{equation}
\label{eq-epsilon-prime3}
| \bv_3'| \le 
\frac{1}{n^2} 
\| D_{\ell} B D_j Q \| 
\sum_{m=1}^j \E \left( \tr D_m \left( Q_m-Q \right) \right)
\le 
\frac{\sigma_{\max}^6}{n \rho^3} 
\end{equation} 
by Lemma \ref{lm-approximations-EQ-C-B}-(\ref{superlemma-item:rank1}).  \\

We now turn to ${\bs \chi_5}$ and provide decomposition (\ref{eq:4}).
Recall that $\tr Ay y^* B= y^* B A y$. 
Combining (\ref{inversion-lemma}) and \eqref{eq-qii}, we get
$Q=Q_m -\rho [\ti Q]_{m m} Q_m y_m y_m^* Q_m$. Plugging this expression
into the definition ${\bs \chi_5}$ and using the fact that $y_m$ is
measurable with respect to ${\mathcal F}_{j+1}$ (since $m \ge j+1$), we
obtain:
\begin{eqnarray*}
\chi_5 &=& \frac{\rho}{n} \tr D_{\ell} B \left( 
\sum_{m=j+1}^n \tilde{c}_m \E_{j+1} \left( y_m  y_m^* Q_m \right) 
 \right)D_j Q\\
      &=& \frac{\rho}{n}   
\sum_{m=j+1}^n \tilde{c}_m y_m^* \left( \E_{j+1} Q_m \right) 
 D_j Q_m D_{\ell} B y_m \\
&& \quad - \frac{\rho^2}{n} \sum_{m=j+1}^n \tilde c_m [\ti Q]_{m m} y^*_m  \left( \E_{j+1} Q_m \right) 
D_j Q_m y_m y_m^* Q_m D_{\ell} B y_m\ .
\end{eqnarray*}
In order to understand the forthcoming decomposition, recall that
asymptotically $y^*_m A_m y_m \sim \frac 1n \tr D_m A_m$ as long as
$y_m$ and $A_m$ are independent, and that $\frac 1n \tr D_m A_m\sim \frac 1n \tr D_m A$
if $A_m$ is a rank-one perturbation of $A$. We can now introduce ${\bs \chi_6}$ and ${\bs \chi_7}$:
\begin{eqnarray*}
{\bs \chi_5} &=& \frac{\rho}{n}   
\sum_{m=j+1}^n \frac{\tilde{c}_m}{n}  \tr D_{\ell} B D_m \left( \E_{j+1} Q \right)D_j Q \\
&& \quad - \frac{\rho^2}{n} \sum_{m=j+1}^n \frac{\tilde c_m^2}{n} \tr D_m \left( \E_{j+1} Q \right) 
D_j Q \times \frac 1n \tr ( D_m Q D_{\ell} B ) + \bv_6 -\bv_7\\
&\stackrel{\triangle}{=}& {\bs \chi_6} -{\bs \chi_7} + \bv_6 -\bv_7\ ,
\end{eqnarray*}
where
\begin{eqnarray*}
\bv_6 &=& 
 \frac{\rho}n 
\sum_{m=j+1}^n \tilde{c}_m \ \ y_m^* \left( \E_{j+1} Q_m \right) D_j 
Q_m D_{\ell} B y_m \\
&&\qquad - \frac{\rho}{n}   
\sum_{m=j+1}^n \frac{\tilde{c}_m}{n}  \tr D_{\ell} B D_m \left( \E_{j+1} Q \right)D_j Q\\
\bv_7 &=& 
 \frac{\rho^2}{n} 
\sum_{m=j+1}^n \tilde{c}_m [ \tilde Q ]_{mm} \ \ 
y_m^* \left( \E_{j+1} Q_m \right) D_j Q_m y_m \ \ 
y_m^* Q_m D_{\ell} B y_m \\
&&\qquad - \frac{\rho^2}{n} \sum_{m=j+1}^n \frac{\tilde c_m^2}{n} \tr D_m \left( \E_{j+1} Q \right) 
D_j Q \times \frac 1n \tr ( D_m Q D_{\ell} B ) \ . 
\end{eqnarray*} 
It is now a matter of routine to check that:
\begin{equation}
\label{eq-epsilon-6-7}
\E| \bv_6 | \leq \frac{K}{\sqrt{n}} 
\qquad \mathrm{and} \qquad 
\E| \bv_7 | \leq \frac{K}{\sqrt{n}} \ . 
\end{equation} 
Let us provide some details. 

Recall that $y_m$ is independent from $\mathbb{E}_{j+1}(Q_m)$. To
obtain the bound on $\E |\bv_6|$, we use the facts that $\E ( y_m^*
\left( \E_{j+1} Q_m \right) D_j Q_m D_{\ell} B y_m - \frac 1n \tr
D_{\ell} B D_m \left( \E_{j+1} Q_m \right)D_j Q_m )^2 \leq K n^{-1}$
by Lemma \ref{lm-approx-quadra}-(\ref{silverstein}), $ | \frac 1n \tr
D_{\ell} B D_m \left( \E_{j+1} Q_m \right)D_j ( Q_m - Q ) | \leq
Kn^{-1}$ by Lemma
\ref{lm-approximations-EQ-C-B}-(\ref{superlemma-item:rank1}), etc.

In order to prove that $\E | \bv_7 | \leq
Kn^{-\frac 12}$, we use similar arguments but we also need two additional estimates. The control
$[\ti Q ]_{mm}-\ti c_m$ which has already been done while estimating 
$\bv_3$ relies on \eqref{eq-identite-tildeQX}.  
The bounded character of $\E (y_m^* A_m y_m)^2$ where $A_m$ is
independent of $y_m$ and of finite spectral norm.  This is a
by-product of Lemma \ref{lm-approx-quadra}-(\ref{silverstein}).

We now put the pieces together and provide Eq. (\ref{eq-xi-lj-sum})
satisfied by ${\bs \chi_{\ell j}}$. Recall that
\begin{eqnarray*}
{\bs \chi_1} &=&  \frac 1n \tr D_{\ell} B D_j Q \\
{\bs \chi_2} &=&
 \frac {\rho}n \tr D_{\ell} B\  
\diag \left( \frac{1}{n} \tr \tilde D_i \tilde C\right) 
\left( \E_{j+1} Q \right) D_j Q \\
{\bs \chi_4} &=&
\frac{\rho}{n^2} \tr D_{\ell} B \left( 
\sum_{m\le j} \tilde{c}_m D_m  \right) (\E_{j+1} Q ) D_j Q\\
{\bs \chi_6} &=&
\frac{\rho}{n^2} \tr D_{\ell} B \left( 
\sum_{m=j+1}^n \tilde{c}_m D_m  \right) (\E_{j+1} Q ) D_j Q\\
{\bs \chi_7} &=&
\frac{\rho^2}{n} \sum_{m=j+1}^n \frac{\tilde c_m^2}{n} 
\tr D_m \left( \E_{j+1} Q \right) 
D_j Q \times \frac 1n \tr ( D_m Q D_{\ell} B )  \\
&=&
\frac{1}{n} \sum_{m=j+1}^n 
\frac{\frac 1n \tr ( D_{\ell} B D_m Q)}{(1 + \frac 1n \tr D_m \E Q)^2}
 {\bs \chi_{m,j}}  \ . 
\end{eqnarray*}
As $ \frac{1}{n} \sum_{m=1}^n \tilde c_m D_m = 
\diag\left( \frac{1}{n} \tr \tilde D_1 \tilde C, \ldots, 
\frac{1}{n} \tr \tilde D_N \tilde C \right) $, 
we have ${\bs \chi_2} - {\bs \chi_4} - {\bs \chi_6} = 0$, and (\ref{eq-xi-lj-sum}) becomes
\[
\bs{\chi_{\ell, n}^{(j)}} =  \frac 1n \tr D_{\ell} B D_j Q 
+
\frac{1}{n} \sum_{m=j+1}^n 
\frac{\frac 1n \tr ( D_{\ell} B D_m Q)}{(1 + \frac 1n \tr D_m \E Q)^2}
 \bs{\chi_{m, n}^{(j)}}  
\ + \ \bv\ , 
\]
where $\E | \bv | \leq Kn^{-\frac 12}$ thanks to Inequalities
(\ref{eq-epsilon3}), (\ref{eq-epsilon-prime3}), and
(\ref{eq-epsilon-6-7}). Small adjustments need to be done in order to
obtain (\ref{eq-xi-lj}). Now replace $\frac 1n \tr D_{\ell} B D_p Q$
by $\frac 1n \tr D_{\ell} B D_p \E Q$ (use Lemma
\ref{lm-approximations-EQ-C-B}-(\ref{eq-approx:(Q-EQ)2})). The new
error term $\bs{\varepsilon_{3,\ell j}}$ still satisfies $\max_{\ell,
  j\le n} \E |\bs{\varepsilon_{3,\ell,j}}| \leq Kn^{-\frac 12}$. Eq.
(\ref{eq-xi-lj}) is proved.
\end{proof}

\begin{proof}[Proof of (\ref{eq-zeta_jj-x_jj})] 
  Recall that $\boldsymbol{\chi_{\ell, j}}$ and
  $\boldsymbol{\varepsilon_{3,\ell j}}$ have been introduced above.

Following the matrix framework introduced to express the system
satisfied by the $\bs{y}$'s (matrices $A_n$, $A_n^{(j)}$ and
$A_n^{0,(j)}$), we introduce matrix $G_n=A_n^T=(g_{\ell m})_{\ell,
  m=1}^n$, its $(n-j+1)\times (n-j+1)$ principal submatrix
$G_n^{(j)}=(g_{\ell, m})_{\ell, m=j}^n$ and the matrix $G_n^{0,(j)}$
which differs from matrix $G_n^{(j)}$ by its first column, equal to
zero. Denoting by $\bs{\delta_n^{(j)}}=\left( \frac 1n \tr D_{\ell} D_j T^2;\quad j\le \ell \le n\right)$, we have:
$$
\bs{x_n^{(j)}} = G_n^{0,(j)} \bs{x_n^{(j)}} + \bs{\delta_n^{(j)}}\ .
$$

Define here the $(n-j+1) \times 1$ vector
  $\boldsymbol{\varepsilon_3^{(j)}} = (\boldsymbol{\varepsilon_{3,\ell
      j}};\ j\le \ell\le n) $ and the $(n-j+1) \times 1$ vectors:
\begin{eqnarray*}
\boldsymbol{\chi^{(j)}} &=& (\boldsymbol{\chi_{\ell, n}^{(j)}};\ j\le \ell\le n)\ ,\\
\bs{\breve{\delta}^{(j)}} & =&  \left( \frac 1n \tr D_{\ell} B D_j \E Q;\  j \leq \ell \leq n \right)\ .
\end{eqnarray*}
Define now the $(n-j+1) \times (n-j+1)$ matrix:
\begin{eqnarray*}
\breve{G}^{(j)} &= &\left( \frac{\frac{1}{n^2} \tr D_{\ell} B D_m \E Q }{\left( 
1 + \frac 1n \tr D_{m} \E Q\right)^2}
\right)_{\ell,m=j}^n\ ,\\
\end{eqnarray*}
and $\breve{G}^{0,(j)}$ which is egal to $\breve{G}^{(j)}$ exept for
its first column equal to zero. With these notations, Eq.
(\ref{eq-xi-lj}), valid for for $j\le \ell \le n$, can take the
following matrix form:
$$
\boldsymbol{\chi^{(j)}} = \breve{G}^{0,(j)} \boldsymbol{\chi^{(j)}} + 
\bs{\breve{\delta}^{(j)}} + \boldsymbol{\varepsilon_3^{(j)}} \ .
$$
We will heavily rely on the following: 
$$
\limsup_n \leftrownorm \left( I - G^{0,(j)} \right)^{-1} \rightrownorm_\infty <\infty\ .
$$
which can be proved as in Lemma \ref{lm-(I-A)regular}-(\ref{quatre})
and Lemma \ref{lm-properties-A}. We drop superscript $^{0,(j)}$ in the
equation below for the sake of readability.
\begin{eqnarray}
  \lefteqn{{\bs \chi} = \breve{G} {\bs \chi} + \bs{\breve{\delta}} + 
    \boldsymbol{\varepsilon_3}} \nonumber \\
  &\Leftrightarrow& {\bs \chi} 
= G {\bs \chi} + \bs{\delta} + 
\boldsymbol{\varepsilon_3} + (\breve{G} - G){\bs \chi} 
+ (\bs{\breve{\delta}} -\bs{\delta})\ , \nonumber \\
  &\Leftrightarrow& {\bs \chi} = (I-G)^{-1} {\bs \delta}  + 
  (I - G)^{-1} \boldsymbol{\varepsilon_3} 
+ (I - G)^{-1} (\breve{G} - G) {\bs \chi} 
+ (I - G)^{-1} (\bs{\breve{\delta}} -\bs{\delta}) \ , \nonumber \\
  &\Leftrightarrow& {\bs \chi} = \bs{x}  + (I - G)^{-1} \boldsymbol{\varepsilon_3} 
+ (I - G)^{-1} (\breve{G} - G) {\bs \chi} 
+ (I - G)^{-1} (\bs{\breve{\delta}} -\bs{\delta}) 
\label{eq-chi-close-x} 
\end{eqnarray}
The first component of the previous equation writes:
\begin{eqnarray*}
\bs{\chi_{j}^{(j)}}&=& 
\bs{x_{j}^{(j)}} + [(I-G^{0,(j)})^{-1} \boldsymbol{\varepsilon_3}]_1 \\
&\phantom{=}& \qquad +\big[ (I-G^{0,(j)})^{-1} (\breve{G}^{0,(j)} - G^{0,(j)}) {\bs \chi} 
+ (I-G^{0,(j)})^{-1} (\bs{\breve{\delta}} -\bs{\delta}) \big]_1\ ,\\
&\stackrel{\triangle}{=}& \bs{x_{j}^{(j)}} +  \boldsymbol{\varepsilon_{41,j}} + \boldsymbol{\varepsilon_{42,j}}\ .
\end{eqnarray*}
Due to Lemma \ref{lm-(I-A)regular}-(\ref{quatre}) which applies to $G^{0,(j)}$ and to the fact that 
$\max_{\ell, j\le n} \E |\boldsymbol{\varepsilon_{3,\ell  j}} | \le 
K n^{-\frac 12}$, we have:
\[
\E |\boldsymbol{\varepsilon_{41,j}}  | 
\leq \sum_{m=1}^{n-j+1} 
[ (I - G^{0,(j)})^{-1} ]_{1,m}
 \ \E |  \boldsymbol{\varepsilon_{3,\ell j}}|
\leq \frac{K}{\sqrt{n}}  \ .
\]
The second error term $\boldsymbol{\varepsilon_{42,j}}$ is the sum of the following terms:
$$
\boldsymbol{\varepsilon_{42,j}}\quad  =\quad 
[ (I-G^{0,(j)})^{-1} (\breve{G}^{0,(j)} - G^{0,(j)}) \boldsymbol{\chi} ]_1  + [  (I-G^{0,(j)})^{-1} 
(\bs{\breve{\delta}} - \bs{\delta} ) ]_1 
$$
Let us first prove that 
$[ (I-G^{0,(j)})^{-1} (\breve{G}^{0,(j)} - G^{0,(j)}) \boldsymbol{\chi} ]_1$  is dominated 
by a sequence independent of $j$ that converges to zero 
as $n\rightarrow \infty$.
The mere definition of $\boldsymbol{\chi_{\ell, n}^{(j)}}$  
(see \eqref{DefChi}) yields
$\| \boldsymbol{\chi^{(j)}} \|_\infty \leq 
(N \sigma_{\max}^4)(n \rho^{2})^{-1}$, where $\|\cdot \|_{\infty}$ stands for the $\ell_{\infty}$-norm. Hence
\begin{multline*}
| [  (I - G^{0,(j)})^{-1} (\breve{G}^{0,(j)} - G^{0,(j)}) \boldsymbol{\chi} ]_1 | \\
\leq 
\leftrownorm (I-G^{0,(j)})^{-1} \rightrownorm_\infty 
\leftrownorm ( \breve{G^{0,(j)}} - G^{0,(j)} )^T \rightrownorm_\infty 
\| \boldsymbol{\chi} \|_\infty 
\leq K 
\leftrownorm ( \breve{G}^{0,(j)} - G^{0,(j)} )^T \rightrownorm_\infty \ .
\end{multline*}
Let us prove that
\begin{equation}
\label{eq-bA-A-to-zero} 
\leftrownorm (\breve{G}^{0,(j)} - G^{0,(j)})^T \rightrownorm_\infty 
\xrightarrow[n\to\infty]{} 0 
\end{equation} 
uniformly in $j$. To this end, 
let us evaluate the $(\ell,m)$-element of matrix $\breve{G}^{0,(j)} - G^{0,(j)}$ $(m>j)$: 
\begin{eqnarray}
n | [ \breve{G}^{0,(j)} - G^{0,(j)} ]_{\ell m} | &=&
\left|
\frac{\frac 1n \tr D_{\ell} B D_m \E Q}{(1+\frac 1n \tr D_m \E Q)^2} 
-
\frac{\frac 1n \tr D_{\ell} D_m T^2}{(1+\frac 1n \tr D_m T)^2} 
\right| \nonumber \\
&\leq&
\left| (1+\frac 1n \tr D_m T)^2 \frac 1n \tr D_{\ell} B D_m \E Q
-
(1+\frac 1n \tr D_m \E Q)^2 \frac 1n \tr D_{\ell} D_m T^2
\right| \nonumber \\
&\leq&
\phantom{+} \left| (1+\frac 1n \tr D_m T)^2 
\frac 1n \tr D_{\ell} B D_m (\E Q - T ) \right| \nonumber \\ 
& & 
+ 
\left| (1+\frac 1n \tr D_m T)^2 
\frac 1n \tr D_{\ell} T D_m (B - T ) \right| \nonumber \\
& &
+
\left| 
\left( (1+\frac 1n \tr D_m T)^2 - (1+\frac 1n \tr D_m \E Q)^2 \right) 
\frac 1n \tr D_{\ell} D_m T^2  \right| \ .
\label{eq-G-A} 
\end{eqnarray}
The first term of the right hand side of (\ref{eq-G-A})
satisfies:
$$
\left| (1+\frac 1n \tr D_m T)^2 
\frac 1n \tr D_{\ell} B D_m (\E Q - T ) \right| 
\leq
\left( 1 + \frac{\sigma_{\max}^2}{\rho} \right)^2 
\frac 1n \tr U ( \E Q - T ) \ ,
$$
where $U$ is the $N \times N$ diagonal matrix 
$U = \sigma_{\max}^4\,\rho^{-1} \diag\left( \mathrm{sign} \left( \E [Q]_{ii} 
- t_i \right), 1 \le i \le N \right)$. 
By Lemma \ref{lm-approximations-EQ-C-B}-(\ref{eq-approx:EQ-T}), the 
right hand side of this inequality converges to zero as $n\to \infty$. 

The second and third terms of the right hand side of (\ref{eq-G-A}) can 
be handled similarly with the help of Lemma \ref{lm-approximations-EQ-C-B} 
and one can prove that elements of $n ( \breve{G}^{0,(j)} - G^{0,(j)} )$ are dominated by a sequence 
independent of $j$ that converges to zero. This implies
that $\leftrownorm ( \breve{G}^{0,(j)} - G^{0,(j)} )^T \rightrownorm_\infty$ converges 
to zero uniformly in $j$ and (\ref{eq-bA-A-to-zero}) is proved.  
As a consequence, 
$[  (I - G^{0,(j)})^{-1} (\breve{G}^{0,(j)} - G^{0,(j)}) \chi ]_1$ is dominated by
a sequence independent of $j$ that converges to zero. The other term 
$[ (I-G^{0,(j)})^{-1} (\bs{\breve{\delta}} - \bs{\delta}) ]_1$  
in the expression of $\boldsymbol{\varepsilon_{42,j}}$ is handled similarly.
Eq. (\ref{eq-zeta_jj-x_jj}) is proved.  
\end{proof}

\begin{proof}[Proof of (\ref{eq-Ej+1(EjAj)^2-theta})] 
We rewrite Equation (\ref{eq-Ej+1(EjAj)^2}) as
$\E_{j+1} \left( \E_j \Gamma_j \right)^2 = \eta_{1,j} +\kappa \eta_{2,j} + 
\boldsymbol{\varepsilon_{2,j}}$ with 
\begin{eqnarray*} 
\eta_{1,j} &=& 
\frac{1}{n^2} 
\frac{1}{\left(1+\frac 1n \tr D_j \E Q \right)^2} 
\tr D_j \left( \E_{j+1} Q_j \right) 
D_j \left( \E_{j+1} Q_j \right) \phantom{\sum_{i=1}^N }\ , \\
\eta_{2,j} &=& 
\frac{1}{n^2\left(1+\frac 1n \tr D_j \E Q \right)^2} 
\sum_{i=1}^N \sigma^4_{ij} \left( \E_{j+1} [Q_j]_{ii} \right)^2\ , 
\end{eqnarray*} 
and we prove that $\sum_{j=1}^n \eta_{1,j} - \tilde{\mathcal V}_n \cvgP{ } 0$ and 
$\sum_{j=1}^n \eta_{2,j} - {\mathcal W}_n \cvgP{ }  0$ where $\tilde{\mathcal V}_n$ and ${\mathcal W}_n$ are defined in Section \ref{proof-variance}. To prove the first assertion, we first notice that
$$\tr D_j ( \E_{j+1} Q_j ) D_j ( \E_{j+1} Q_j ) 
= 
\E_{j+1} ( \tr D_j (\E_{j+1} Q_j) D_j Q_j ) 
= 
\E_{j+1} ( \tr D_j (\E_{j+1} Q) D_j Q ) + \bv
$$ 
with $|\bv| \leq 2 \sigma_{\max}^4 \rho^{-2}$ by 
Lemma \ref{lm-approximations-EQ-C-B}-(\ref{superlemma-item:rank1}).
Therefore, 
$$
\eta_{1,j} = \frac{\E_{j+1} \bs{\chi_{j, n}^{(j)}}}{\left( 1+\frac 1n
    \tr D_j \E Q\right)^2} + \frac{\bv}{\left( 1+\frac 1n \tr D_j \E
    Q\right)^2}\ .
$$ 
It remains to control the difference $\left( 1+\frac 1n \tr D_j \E
  Q\right)^{-2} -\left( 1+\frac 1n \tr D_j T\right)^{-2}$, to plug
(\ref{eq-zeta_jj-x_jj}) and one easily
obtains $\sum_{j=1}^n \eta_{1,j} - \tilde{\mathcal V}_n \cvgP{ } 0$. 

We now sketch the proof of $\sum_{j=1}^n \eta_{2,j} - {\mathcal W}_n \cvgP{ } 0$.
As in (\ref{eq-qii}), $[ Q_j ]_{ii}$ satisfies $ [ Q_j ]_{ii}
= -( z(1 + \xi_{i}^j \ti Q_i^j {\xi_i^j}^*))^{-1}$ where $\xi_i^j$
is the row $\xi_i$ without element $j$, and $\ti Q_i^j = ( {Y_i^j}^*
Y_i^j + \rho I_{n-1} )^{-1}$ where $Y_i^j$ is matrix $Y$ without row
$i$ and column $j$.  Using this identity and Lemmas
\ref{lm-approx-quadra}-(\ref{silverstein}) and \ref{lm-approximations-EQ-C-B}, we can show
that $[ Q_j ]_{ii}$ is approximated by $t_i$, which is key to prove
$\sum_{j=1}^n \eta_{2,j} - {\mathcal W}_n \cvgP{ } 0$.
\end{proof} 

\section{Proof of Theorem \ref{th-bias}} 
\label{sec-proof-bias}


We first provide an expression of the bias that involves the Stieltjes
transforms $\frac 1N \mathrm{Tr}\, Q$ and $\frac 1N \mathrm{Tr}\,T$.
By writing $\log\det(Y_n Y_n^* + \rho I_N) = N \log\rho +
\log\det(\frac 1\rho Y_n Y_n^* + I_N)$ and by taking the derivative of
$\log\det(\frac 1\rho Y_n Y_n^* + I_N)$ with respect to $\rho$, we
obtain
\[
\log\det(Y_n Y_n^* + \rho I_N) = N \log \rho 
+ N \int_\rho^\infty 
\left( \frac 1 \omega - \frac 1N \tr Q(-\omega) \right) d\omega  \ . 
\]
Since $\frac 1N \tr Q(z) \in {\mathcal S}(\R^+)$, we have 
$\frac{1}{\omega} - \frac 1N \tr Q(-\omega) \geq 0$ for $\omega > 0$. In fact, recall 
that $\|Q(-\omega)\|\le \omega^{-1}$ by Proposition \ref{prop-resolvent-properties}.  
Therefore, by Fubini's Theorem, 
\[
\E \log\det(Y_n Y_n^* + \rho I_N) = N \log \rho 
+ N \int_\rho^\infty 
\left( \frac 1 \omega - \frac 1N \tr \E Q(-\omega) \right) d\omega  \ . 
\]
Similarly,
\[
N V_n(\rho) = N \int \log(\lambda + \rho) \pi_n(d\lambda) 
= N \log \rho 
+ N \int_\rho^\infty 
\left( \frac 1 \omega - \frac 1N \tr T(-\omega) \right) d\omega .
\]
Hence the bias term is given by: 
\begin{equation}
\label{eq-Zn2-resolvent} 
{\mathcal B}_n(\rho) \stackrel{\triangle}{=} \E \log\det(Y_n Y_n^* + \rho I_N) - N V_n(\rho) =  
 \int_\rho^\infty 
\tr \left( T(-\omega) - \E Q(-\omega) \right) d\omega \ .
\end{equation} 
In Appendix \ref{anx-prf-Q-Qtilde}, we prove that:
\begin{equation}
\label{eq-tr(T-Q)=tr(T-Q)tilde} 
\tr ( T - Q ) = \tr (\ti T - \ti Q) \ . 
\end{equation} 
Therefore, we can also write the bias as: 
\begin{equation}
\label{eq-Zn^2-integral} 
{\mathcal B}_n(\rho) = \int_\rho^\infty 
\tr \left( \ti T(-\omega) - \E \ti Q(-\omega) \right) d\omega  \ . 
\end{equation} 
For technical reasons (and in order to rely on results of Section
\ref{proof-variance}), we use representation \eqref{eq-Zn^2-integral}
of the bias instead of \eqref{eq-Zn2-resolvent}. The proof of Theorem
\ref{th-bias} will rely on the study of the asymptotic behaviour of
the integrand in the right hand side of this equation.

As a by-product of Section \ref{proof-variance}, the existence and uniqueness of the solution of the
system of equations (\ref{eq-def-w-bias}) is straightforward. Indeed,  
define the $n \times 1$ vectors ${\bs w}$ and ${\bs p}$ as
\begin{eqnarray*}
{\bs w} &=& \left( {\bs w}_{j,n}; 1 \leq j \leq n \right) \ , \\ 
{\bs p} &=& \left( {\bs p}_{j,n}; 1 \leq j \leq n \right) \ .
\end{eqnarray*}
Then the system (\ref{eq-def-w-bias}) can be written in a matrix form as
\begin{equation}
\label{eq-system-w} 
{\bs w} = A {\bs w} + {\bs p}  \ .
\end{equation} 
Since $(I-A)$ is invertible for $n$ large enough, this proves Theorem
\ref{th-bias}--(\ref{th-bias-unique-solution}).

The rest of the proof will be carried out into four steps: 

\begin{enumerate}
\item
\label{sec-prf-bias-phi-psi}
We first introduce a perturbated version of the system \eqref{eq-system-w}.
For the reader's convenience, we recall the following notations: 
\[
\begin{array}{lclclcl} 
t_i &=& \displaystyle{
\frac{1}{\omega\left( 1 +\frac 1n \tr \ti D_i \ti T \right)}},  & & 
\ti t_j &=& 
\displaystyle{\frac1{\omega\left( 1 +\frac 1n \tr D_j T\right)}},  \\ 
c_i &=& \displaystyle{
\frac1{\omega\left( 1 +\frac 1n \tr \ti D_i \E \ti Q \right)}},  & & 
\ti c_j &=& 
\displaystyle{\frac1{\omega\left( 1 +\frac 1n \tr D_j \E Q \right)}}, \\ 
b_i &=& \displaystyle{
\frac1{\omega\left( 1 +\frac 1n \tr \ti D_i \ti C \right)}},  & & 
\ti b_j&=& 
\displaystyle{\frac1{\omega\left( 1 +\frac 1n \tr D_j C \right)}} ,
\end{array}
\] 
where $z$ is equal to $-\omega$ with $\omega\ge 0$. Write the integrand in 
\eqref{eq-Zn^2-integral} as 
\begin{equation}
\label{eq-def-varphi} 
\tr \left( \ti T(-\omega) - \E \ti Q(-\omega) \right) = 
\frac{1}{n} \sum_{j=1}^n {\bs \varphi}_j(\omega) 
\quad 
\mathrm{with} \quad {\bs \varphi}_j(\omega) \eqdef 
n ( \ti t_j(-\omega)  - \E [\ti Q(-\omega)]_{jj})\ . 
\end{equation} 
Let ${\bs \psi}^{(j)}(\omega) \eqdef n(\ti b_j(-\omega) - \E [\ti
Q(-\omega)]_{jj})$ and define the $n \times 1$ vectors ${\bs \varphi}$
and ${\bs \psi}$ and the $n \times n$ matrix $\breve{A}$ as
\begin{eqnarray*}
{\bs \varphi} &=& \left( {\bs \varphi}_j; 1 \leq j \leq n \right) \ , \\ 
{\bs \psi} &=& \left( {\bs \psi}^{(j)}; 1 \leq j \leq n \right) \ , \\ 
\breve{A} &=& \left(
\frac{ \frac{1}{n^2} \tr D_j D_m C T }
{ ( 1 + \frac 1n \tr D_j T) ( 1 + \frac 1n \tr D_j C) }
\right)_{j,m=1}^n 
\end{eqnarray*}
We first prove that
\begin{equation}
\label{eq-varphi-matrix} 
{\bs \varphi} = \breve{A} {\bs \varphi} + {\bs \psi} \ . 
\end{equation} 

\item
\label{sec-prf-bias-expression-psi}
We prove that 
\begin{equation}
\label{eq-psi-i} 
{\bs \psi}^{(j)} = \kappa \ \omega^2 \ti b_j \ti c_j \left( 
 \frac{\omega}{n} \sum_{i=1}^N \left( \sigma^2_{ij} c_i^3 
\frac 1n \sum_{m=1}^n \sigma^4_{im} \E [\ti Q_i]_{mm}^2 \right) 
-
\frac{\ti c_j}n \sum_{i=1}^N \sigma^4_{ij} \E [ {Q}_j ]_{ii}^2  \right) 
+ {\bs \varepsilon}^{(j)}  \ ,
\end{equation} 
with $|{\bs \varepsilon}^{(j)} | \leq K n^{-1/2}$ where $K$ is a constant that does 
not depend on $n$ nor on $j$ (but may depend on $\omega$). \\

\item
\label{sec-prf-bias-(phi-w)-to-0} 
Matrix $\breve{A}$ readily approximates $A$ and vector 
${\bs \psi}$ approximates ${\bs p}$ for large $n$ by Step \ref{sec-prf-bias-expression-psi}. 
Therefore, by inspecting Equations \eqref{eq-system-w} and 
\eqref{eq-varphi-matrix}, one may expect ${\bs \varphi}$ to be close to 
${\bs w}$. We prove here that 
\begin{equation}
\label{eq-(phi-w)-to-0}
\| {\bs \varphi} - {\bs w} \|_\infty 
\xrightarrow[n\to\infty, N/n \to c]{} 0 \ . 
\end{equation}

\item
\label{sec-prf-bias-dct} 
Let 
$\beta_n(\omega) = \frac 1n \sum_{j=1}^n {\bs w}_{j,n}(\omega)$. Eq.
\eqref{eq-(phi-w)-to-0} yields
$\frac 1n \sum_{j=1}^n {\bs \varphi}_j(\omega) - \beta_n(\omega)
\rightarrow 0 \ .$ 
In order to prove \eqref{eq-result-bias}, it remains to integrate and to provide a
Dominated Convergence Theorem argument. To this end, we shall prove that: 
\begin{equation}
\label{eq-dct-beta} 
\left| \beta_n(\omega) \right| \leq \frac{K'}{\omega^3} 
\end{equation} 
for $n$ large enough. This will establish \eqref{eq-integrability-beta}. We will also prove that
\begin{equation}
\label{eq-dct-phi} 
\left| \frac 1n \sum_{j=1}^n {\bs \varphi}_j(\omega) \right| \leq 
\frac{K'}{\omega^2} 
\end{equation} 
for $\omega \in [ \rho, + \infty)$, where $K'$ does not to depend on
$n$ nor on $\omega$. This will yield \eqref{eq-result-bias} and 
the proof of Theorem \ref{th-bias} will be completed.

\end{enumerate} 

\subsection{Proof of step \ref{sec-prf-bias-phi-psi}: Equation 
(\ref{eq-varphi-matrix})}  
\label{subsec-th-bias-expression-bias} 
Recall that ${\bs \psi}^{(j)} = n( \ti b_j - \E [\ti Q]_{jj})$. 
Using these expressions, we have for $1\le j\le n$:
\begin{eqnarray*} 
{\bs \varphi}_j &=& n( \ti t_j - \ti b_j)  + {\bs \psi}^{(j)} 
\quad = \quad 
  n \ti b_j \ti t_j \left( \ti b_j^{-1} - \ti t_j^{-1} \right) 
 \ + \ {\bs \psi}^{(j)} 
  \\ 
&=& 
\omega \ti b_j \ti t_j \tr D_j \left( C - T \right) 
\ + \ {\bs \psi}^{(j)}  \\ 
&=& 
\omega \ti b_j \ti t_j \sum_{i=1}^N \sigma^2_{ij} c_i t_i 
\left( t_i^{-1} - c_i^{-1} \right) \ + \ {\bs \psi}^{(j)} 
 \\
&=& 
\frac{ \omega^2 \ti b_j \ti t_j}{n^2} \sum_{i=1}^N \sum_{m=1}^n
\sigma^2_{ij} \sigma^2_{im} c_i t_i \ 
{\bs \varphi}_m \ + \ {\bs \psi}^{(j)} \\
&=& \omega^2 \ti b_j \ti t_j \sum_{m=1}^n \frac 1{n^2} \mathrm{Tr} (D_j D_m C T)\  {\bs \varphi}_m \ + \ {\bs \psi}^{(j)} , \\
\end{eqnarray*} 
which yields Eq. \eqref{eq-varphi-matrix}. 

\subsection{Proof of step \ref{sec-prf-bias-expression-psi}: Expression
of ${\bs \psi}^{(j)}$} 
We shall develop ${\bs \psi}^{(j)}$ as 
\begin{eqnarray}
{\bs \psi}^{(j)} &=& {\bs \psi}_1 \ + \ {\bs \psi}_2 \ - \ {\bs \psi}_3 
\label{eq-psi_i-123} \\
 & & {\bs \psi}_1 = {\bs \psi}_4 + {\bs \varepsilon}_1 
 \label{eq-psi1-4} \\
 & & 
 \ \ \ \ \ \ \ \ \ \, {\bs \psi}_2 = - {\bs \psi}_5 + {\bs \psi}_6 
 \label{eq-psi2-56} \\
& &
\ \ \ \ \ \ \ \ \ \ \ \ \ \ \ \ \ \ \ \, 
 {\bs \psi}_5 = {\bs \psi}_7
 + {\bs \varepsilon}_5 \label{eq-psi5-7} \\
 & & 
\ \ \ \ \ \ \ \ \ \ \ \ \ \ \ \ \ \ \ \ \ \ \ \ \ \ \ 
 {\bs \psi}_6 = {\bs \psi}_8 + {\bs \varepsilon}_6 \label{eq-psi6-8} \\
& &  
\ \ \ \ \ \ \ \ \ \ \ \ \ \ \ \ \ \ \ 
{\bs \psi}_3 = {\bs \psi}_9 + {\bs \varepsilon}_3
 \label{eq-psi3-9} 
\end{eqnarray} 
where the ${\bs \psi}_k$'s and the ${\bs \varepsilon}_k$'s will be introduced
when required. We shall furthermore  
prove that $| {\bs \varepsilon}_k | \leq K n^{-1/2}$ for $k=1,3,5,6$. 
This will yield
\begin{equation}
\label{eq-expression-psi_i} 
{\bs \psi}^{(j)} = {\bs \psi}_4 - {\bs \psi}_7 + {\bs \psi}_8 - {\bs \psi}_9 +
{\bs \varepsilon}^{(j)} \quad \mathrm{with} \quad 
| {\bs \varepsilon}^{(j)} | = | {\bs \varepsilon}_1 - {\bs \varepsilon}_3 
- {\bs \varepsilon}_5 + {\bs \varepsilon}_6 | \leq \frac{K}{n^{1/2}} \ . 
\end{equation}
Let us begin with decomposition (\ref{eq-psi_i-123}): 
\begin{eqnarray}
{\bs \psi}^{(j)} &=& n \ti b_j \E \left( 
[\ti Q]_{jj} \left( [\ti Q]_{jj}^{-1} - \ti b_j^{-1} \right) \right) 
\nonumber \\
&\stackrel{(a)}{=}& 
n \omega \ti b_j \E \left( [\ti Q]_{jj} \left( 
y_j^* {Q}_j y_j - 
\frac{1}{n} \tr D_j C \right) \right) \nonumber \\
&\stackrel{(b)}{=}& 
n \omega \ti b_j \ti c_j \E \left( y_j^* {Q}_j y_j - 
\frac{1}{n} \tr D_j C \right) \nonumber \\
& & \ \ \ \ \ \ \ \ \ \ 
-  n \omega^2 \ti b_j \ti c_j \E \left( [\ti Q]_{jj} \left( 
y_j^* Q_j y_j - \frac{1}{n} \tr D_j \E Q \right) 
\left( y_j^* Q_j y_j - 
\frac{1}{n} \tr D_j C \right) \right)   
\nonumber  \\
&\stackrel{(c)}{=}& 
\omega \ti b_j \ti c_j \tr D_j \E \left( Q_j - Q \right) 
+ \omega \ti b_j \ti c_j \tr D_j \left(\E Q - C \right)
\nonumber \\ 
& & \ \ \ \ \ \ \ \ \ \ 
-  n \omega^2 \ti b_j \ti c_j \E \left( [\ti Q]_{jj} \left( 
y_j^* Q_j y_j - \frac{1}{n} \tr D_j \E Q \right) 
\left( y_j^* Q_j y_j - 
\frac{1}{n} \tr D_j C \right) \right)   
\nonumber \\
&\eqdef& 
{\bs \psi}_1  +  {\bs \psi}_2 -  {\bs \psi}_3  \nonumber 
\end{eqnarray} 
where $(a)$ follows from (\ref{eq-qii}) and the definition of $\ti b_j$, 
$(b)$ follows from identity (\ref{eq-identite-tildeQX}), and $(c)$ 
follows from the following equality:
\[ 
\E \left( y_j^* Q_j y_j - \frac{1}{n} \tr D_j C \right)
= 
\frac 1n \tr D_j \left( \E Q_j -  C \right) 
= 
\frac 1n \tr D_j \E \left( Q_j - Q \right) 
+ \frac 1n \tr D_j \left( \E Q - C \right) \ .
\] 
Eq. (\ref{eq-psi_i-123}) is established.

We now turn to the decomposition (\ref{eq-psi1-4}). Combining  
(\ref{inversion-lemma}) and (\ref{eq-qii}), we obtain  
$Q = Q_j - \omega [\ti Q ]_{jj} Q_j y_j y_j^* Q_j$, hence 
$ {\bs \psi}_1 = 
\omega^2 \ti b_j \ti c_j \E \left( [ \ti Q ]_{jj} 
y_j^* Q_j D_j Q_j y_j \right)$. 
Using identity (\ref{eq-identite-tildeQX}) and the fact that  
$\E ( y_j^* Q_j D_j Q_j y_j ) = \frac 1n \E (\tr D_j {Q}_j {D}_j {Q}_j)$, 
we obtain: 
\begin{eqnarray*} 
{\bs \psi}_1  &=& 
\frac{\omega^2}{n} \ti b_j \ti c_j^2  
\E \left( \tr D_j {Q}_j {D}_j {Q}_j \right) \\
& & 
\ \ \ \ \ \ \ \ \ \ \ \ \ \ \ \ 
- \omega^3 \ti b_j \ti c_j^2 
\E \left( 
[ \ti Q ]_{jj} \left( y_j^* Q_j y_j - \frac 1n \tr D_j \E Q \right) 
\left( y_j^* Q_j D_j Q_j y_j \right) \right) \\ 
&\eqdef& 
{\bs \psi}_4 + {\bs \varepsilon}_1 \ . \nonumber  
\end{eqnarray*}
We have:
\[ 
| {\bs \varepsilon}_1 | \leq 
\frac{1}{\omega} 
\E \left( y_j^* Q_j D_j Q_j y_j 
\left| 
{\bs \varepsilon}_{11} + {\bs \varepsilon}_{12} + {\bs \varepsilon}_{13} 
\right| \right) \ ,
\]
with 
${\bs \varepsilon}_{11} = 
y_j^* Q_j y_j - \frac 1n \tr D_j Q_j$, 
${\bs \varepsilon}_{12} = 
\frac 1n \tr {D}_j \left( {Q}_j - \E {Q}_j \right)$, and
${\bs \varepsilon}_{13} = 
\frac 1n \tr {D}_j \E \left( {Q}_j - {Q} \right)$. 
By Lemmas 
\ref{lm-approx-quadra}-(\ref{silverstein}), 
\ref{lm-approximations-EQ-C-B}--(\ref{eq-approx:(Q-EQ)2}), and 
\ref{lm-approximations-EQ-C-B}--(\ref{superlemma-item:rank1}), we have
$\E | {\bs \varepsilon}_{11} |^2 \leq K n^{-1}$, 
$\E | {\bs \varepsilon}_{12} |^2 \leq K n^{-2}$, and 
$| {\bs \varepsilon}_{13} |^2 \leq K n^{-2}$ respectively. 
By Cauchy-Schwarz inequality, we therefore have:
\[
| {\bs \varepsilon}_1 | \leq \frac{K \left(\mathbb{E} (y^*_j Q_j D_j Q_j y_j)^2  \right)^{\frac 12}}{\sqrt{n}}
\le \frac{K'}{\sqrt{n}}\ ,  
\]
and (\ref{eq-psi1-4}) is established. 

We now establish decomposition (\ref{eq-psi2-56}): 
\begin{eqnarray*} 
{\bs \psi}_2 &=& 
\omega \ti b_j \ti c_j \tr D_j  \left(\E Q - C \right)  \\
&=& 
\omega \ti b_j \ti c_j \sum_{i=1}^N \sigma^2_{ij} {c}_i 
\E \left( [Q]_{ii} \left( c_i^{-1} - [Q]_{ii}^{-1}
\right) \right)  \\ 
&=& 
- \omega^2 \ti b_j \ti c_j \sum_{i=1}^N \sigma^2_{ij} c_i 
\E \left( 
[Q]_{ii} \left( \xi_i \ti Q_i \xi_i^* - \frac 1n \tr \ti D_i \E \ti Q \right) 
\right)  \\ 
&\stackrel{(a)}{=}&
- \omega^2 \ti b_j \ti c_j \sum_{i=1}^N \sigma^2_{ij} c_i^2  
\left( \E \left( \xi_i \ti Q_i \xi_i^* \right) - 
\frac 1n \tr \ti D_i \E \ti Q  \right)  \\ 
& & \ \ \ \ \ 
+ \ \omega^3 \ti b_j \ti c_j \sum_{i=1}^N \sigma^2_{ij} c_i^2 
\E \left( 
[Q]_{ii} 
\left( \xi_i \ti Q_i \xi_i^* - \frac 1n \tr \ti D_i \E \ti Q \right)^2
\right) \\
&\eqdef& - {\bs \psi}_5 + {\bs \psi}_6   
\end{eqnarray*}
where $(a)$ follows from (\ref{eq-QiiX}). Equation  
(\ref{eq-psi2-56}) is established. 

Let us now turn to decomposition
(\ref{eq-psi5-7}). We have 
${\bs \psi}_5 = 
 \frac{\omega^2 \ti b_j \ti c_j}{n} \sum_{i=1}^N \sigma^2_{ij} c_i^2  
\tr \ti D_i \E \left( \ti Q_i - \ti Q \right)$. 
By similar arguments as those used for ${\bs \psi}_1$, we have:
\begin{eqnarray*} 
{\bs \psi}_5 &=& 
\frac{\omega^3 \ti b_j \ti c_j }{n^2} \sum_{i=1}^N \sigma^2_{ij} c_i^3  
\E \left( \tr \ti D_i \ti Q_i \ti D_i \ti Q_i \right)\\
&\phantom{=}& + \frac{\omega^3 \ti b_j \ti c_j }{n^2} \sum_{i=1}^N \sigma^2_{ij} c_i^2  
\E ([Q]_{ii} -c_i) \tr \ti D_i \ti Q_i \xi_i^* \xi_i \ti Q_i \\
&\eqdef&  
{\bs \psi}_7 + {\bs \varepsilon}_{5} 
\end{eqnarray*}
where $| {\bs \varepsilon}_{5} | \leq K n^{-\frac 12}$ and (\ref{eq-psi5-7}) is established. 

Turning to (\ref{eq-psi6-8}), 
we have 
\begin{eqnarray}
{\bs \psi}_6 &=& 
\omega^3 \ti b_j \ti c_j \sum_{i=1}^N \sigma^2_{ij} c_i^2 
\E \left( [Q]_{ii} 
\left( \xi_i \ti Q_i \xi_i^* - \frac 1n \tr \ti D_i \E \ti Q 
\right)^2 \right) \nonumber \\
&=& 
\omega^3 \ti b_j \ti c_j \sum_{i=1}^N \sigma^2_{ij} c_i^3 
\E \left( 
\xi_i \ti Q_i \xi_i^* - \frac 1n \tr \ti D_i \E \ti Q \right)^2 \nonumber \\
& & \ \ \ \ \ - \ 
\omega^4 \ti b_j \ti c_j \sum_{i=1}^N \sigma^2_{ij} c_i^3 
\E \left( [ Q ]_{ii}
\left( \xi_i \ti Q_i \xi_i^* - \frac 1n \tr \ti D_i \E \ti Q 
\right)^3 \right) \nonumber \\
&\eqdef& {\bs \psi}'_6  + {\bs \varepsilon}_{61} \ ,
\label{eq-psi6-psi6'} 
\end{eqnarray} 
using again (\ref{eq-QiiX}). The term ${\bs \varepsilon}_{61}$ 
satisfies: 
\begin{eqnarray*} 
| {\bs \varepsilon}_{61} | 
&\leq&  
\frac{1}{\omega^2} \sum_{i=1}^N \sigma^2_{ij} 
\E \left| {\bs \varepsilon}_{611,i} + {\bs \varepsilon}_{612,i} +  
{\bs \varepsilon}_{613,i} \right|^3 \\
&\leq&  
\frac{9}{\omega^2} \sum_{i=1}^N \sigma^2_{ij} 
\left( \E \left| {\bs \varepsilon}_{611,i} \right|^3 
+ \E \left| {\bs \varepsilon}_{612,i} \right|^3 +  
\left| {\bs \varepsilon}_{613,i} \right|^3 \right)\ ,  
\end{eqnarray*} 
where ${\bs \varepsilon}_{611,i} = 
\xi_i \ti Q_i \xi_i^* - \frac 1n \tr \ti D_i \ti Q_i$, 
${\bs \varepsilon}_{612,i} = \frac 1n \tr \ti D_i 
\left( \ti Q_i - \E \ti Q_i \right)$, and 
${\bs \varepsilon}_{613,i} = \frac 1n \tr \ti D_i 
\E \left( \ti Q_i - \ti Q \right)$. 
By Lemma \ref{lm-approx-quadra}-(\ref{silverstein}), $\E \left| {\bs \varepsilon}_{611,i} \right|^3
\leq K n^{-3/2}$. By Lemma 
\ref{lm-approximations-EQ-C-B}--(\ref{eq-approx:(Q-EQ)4}), 
$\E \left| {\bs \varepsilon}_{612,i} \right|^3 \leq 
\left(\E \left| {\bs \varepsilon}_{612,i} \right|^4\right)^{3/4} \leq 
K n^{-3/2}$. By Lemma 
\ref{lm-approximations-EQ-C-B}--(\ref{superlemma-item:rank1}), 
$\left| {\bs \varepsilon}_{613,i} \right|^3 \leq K n^{-3}$, hence 
\[
| {\bs \varepsilon}_{6,1} | \leq \frac{K}{\sqrt{n}} \ .
\] 
We now handle the term ${\bs \psi}_6'$ in (\ref{eq-psi6-psi6'}). We have:
\begin{eqnarray*}
{\bs \psi}'_6 &=& 
\omega^3 \ti b_j \ti c_j \sum_{i=1}^N \sigma^2_{ij} c_i^3 
\E \left( \xi_i \ti Q_i \xi_i^* - \frac 1n \tr \ti D_i \ti Q_i + 
{\bs \varepsilon}_{612,i} + {\bs \varepsilon}_{613,i}
\right)^2 \\
&=& 
\omega^3 \ti b_j \ti c_j \sum_{i=1}^N \sigma^2_{ij} c_i^3 
\E \left( \xi_i \ti Q_i \xi_i^* - \frac 1n \tr \ti D_i \ti Q_i \right)^2 + 
{\bs \varepsilon}_{62} \\
&\eqdef&
{\bs \psi}_8 + {\bs \varepsilon}_{62} \ ,
\end{eqnarray*} 
where 
\[
{\bs \varepsilon}_{62} = 
\omega^3 \ti b_j \ti c_j \sum_{i=1}^N \sigma^2_{ij} c_i^3 
\left(
\E \left( {\bs \varepsilon}_{612,i} + {\bs \varepsilon}_{613,i} \right)^2
+ 
2 \E \left(
\left( \xi_i \ti Q_i \xi_i^* - \frac 1n \tr \ti D_i \ti Q_i \right)
\left( {\bs \varepsilon}_{612,i} + {\bs \varepsilon}_{613,i} \right)
\right)
\right) . 
\]
Using Lemmas \ref{lm-approx-quadra}-(\ref{silverstein}) and \ref{lm-approximations-EQ-C-B}, it 
is easy to show that 
\[
\left| {\bs \varepsilon}_{62} \right| \leq \frac{K}{\sqrt{n}} \ . 
\]
Furthermore, the terms $\E \left( \phantom{Q} \right)^2$ in the expression of 
${\bs \psi}_8$ has a more explicit form. Indeed, applying Lemma \ref{lm-approx-quadra}-(\ref{identite-cumu}) yields:
\[
{\bs \psi}_8 = 
\frac{\omega^3 \ti b_j \ti c_j }{n^2} \sum_{i=1}^N \sigma^2_{ij} c_i^3 
\left( \E \left( \tr \ti D_i \ti Q_i \ti D_i \ti Q_i  \right) + 
\kappa \sum_{m=1}^n \sigma^4_{im} \E \left( [\ti Q_i]_{mm}^2 \right) 
\right) \ . 
\] 
Decomposition (\ref{eq-psi6-8}) is established with 
${\bs \varepsilon}_{6} = {\bs \varepsilon}_{61} + {\bs \varepsilon}_{62}$.
 
It remains to give decomposition (\ref{eq-psi3-9}). Using 
\eqref{eq-identite-tildeQX}, we have 
\begin{eqnarray*} 
{\bs \psi}_3 &=&  
 n \omega^2 \ti b_j \ti c_j \E \left( [\ti Q]_{jj} \left( 
y_j^* Q_j y_j - \frac 1n \tr D_j \E Q \right) 
\left( y_j^* Q_j y_j - \frac{1}{n} \tr {D}_j {C} \right) \right)   \\
&=&
 n \omega^2 \ti b_j \ti c_j^2  \E \left( \left( 
y_j^* Q_j y_j - \frac 1n \tr D_j \E Q \right) 
\left( y_j^* Q_j y_j - 
\frac{1}{n} \tr D_j C \right) \right)   \\
& & 
- n \omega^3 \ti b_j \ti c_j^2  
\E \left( [\ti Q]_{jj} \left( 
y_j^* Q_j y_j - \frac 1n \tr D_j \E Q \right)^2  
\left(  y_j^* Q_j y_j - 
\frac{1}{n} \tr D_j C \right) \right)   \\
&\eqdef& 
{\bs \psi}'_3 + {\bs \varepsilon}_{31} \ . 
\end{eqnarray*} 
The term ${\bs \varepsilon}_{31}$ satisfies 
\begin{multline*} 
\left| {\bs \varepsilon}_{31} \right| \leq 
\frac{n}{\omega} 
\E \left(  
\left| y_j^* Q_j y_j - \frac 1n \tr D_j Q_j 
+ {\bs \varepsilon}_{311} + {\bs \varepsilon}_{312} 
\right|^2 \right. \\ 
\times \left. \left| y_j^* Q_j y_j - \frac 1n \tr D_j Q_j + 
{\bs \varepsilon}_{311} + {\bs \varepsilon}_{312}+ {\bs \varepsilon}_{313}
\right| \right) 
\end{multline*}
with 
${\bs \varepsilon}_{311} = 
\frac 1n \tr D_j ( Q_j - \E Q_j)$,
${\bs \varepsilon}_{312} = 
\frac 1n \tr {D}_j \E \left( {Q}_j - {Q} \right)$, and
${\bs \varepsilon}_{313} = 
\frac 1n \tr {D}_j ( \E {Q} - C )$.  \\ 
The terms ${\bs \varepsilon}_{311}$, ${\bs \varepsilon}_{312}$, and 
$y_j^* Q_j y_j - \frac 1n \tr D_j Q_j$ can  
be handled by Lemmas \ref{lm-approx-quadra}-(\ref{silverstein}) and \ref{lm-approximations-EQ-C-B}.
The term ${\bs \varepsilon}_{313}$ coincides with
${\bs \psi}_2 (n \omega \ti b_j \ti c_j)^{-1}$. The derivations made on 
${\bs \psi}_2$ above (decompositions (\ref{eq-psi2-56}--\ref{eq-psi6-8})) 
show that $| {\bs \psi}_2 (\omega \ti b_j \ti c_j)^{-1} |  \leq K$ therefore 
$| {\bs \varepsilon}_{313} | \leq K n^{-1}$. \\

Using these results, we obtain after some standard manipulations: 
\[
\left| {\bs \varepsilon}_{31} \right| \leq \frac{K}{\sqrt{n}} \ . 
\]
The term ${\bs \psi}'_3$ can be written as:
\begin{eqnarray*}
{\bs \psi}'_3 &=&  
n \omega^2 \ti b_j \ti c_j^2 \E \left( \left( 
y_j^* Q_j y_j - \frac 1n \tr D_j Q_j
+ {\bs \varepsilon}_{311} + {\bs \varepsilon}_{312} \right) \right.  \\
& & 
\ \ \ \ \ \ \ \ \ \ \ \ \ \ \ \ \ \  \ \ \ \ \ \ \ \ 
\left.  \left( y_j^* Q_j y_j - \frac 1n \tr D_j Q_j
+ {\bs \varepsilon}_{311} + {\bs \varepsilon}_{312} + {\bs \varepsilon}_{313} 
 \right) \right)   \\
 &=& 
n \omega^2 \ti b_j \ti c_j^2 \E 
 \left( y_j^* Q_j y_j - \frac 1n \tr D_j Q_j \right)^2 
+ {\bs \varepsilon}_{32} \\
&\eqdef& {\bs \psi}_9 + {\bs \varepsilon}_{32} 
\end{eqnarray*} 
with $| {\bs \varepsilon}_{32} | \leq K n^{-1/2}$. Similarly to 
${\bs \psi}_8$, we can develop ${\bs \psi}_9$ to obtain
\begin{equation}
\label{eq-expression-psi9}
{\bs \psi}_9 = 
\frac{
\omega^2 \ti b_j \ti c_j^2 }{n} 
\left( \E ( \tr {D}_j {Q}_j {D}_j {Q}_j ) +
\kappa \sum_{i=1}^N \sigma^4_{ij} \E \left( [ {Q}_j ]_{ii}^2 \right)
\right) 
\end{equation} 
Decomposition (\ref{eq-psi3-9}) is established with 
${\bs \varepsilon}_3 = {\bs \varepsilon}_{31} + {\bs \varepsilon}_{32}$. 

We now put the pieces together and provide Eq. (\ref{eq-expression-psi_i})
satisfied by ${\bs \psi}^{(j)}$. We recall that
\begin{eqnarray*}
{\bs \psi}_4 &=& 
\frac{ \omega^2 \ti b_j \ti c_j^2}{n}  
\E \left( \tr D_j {Q}_j {D}_j {Q}_j \right)\ , \\
{\bs \psi}_7  &=&  
\frac{\omega^3 \ti b_j \ti c_j }{n^2} \sum_{i=1}^N \sigma^2_{ij} c_i^3  
\E \left( \tr \ti D_i \ti Q_i \ti D_i \ti Q_i \right)\ , \\ 
{\bs \psi}_8 &=&  
\frac{\omega^3 \ti b_j \ti c_j }{n^2} \sum_{i=1}^N \sigma^2_{ij} c_i^3 
\left( \E \left( \tr \ti D_i \ti Q_i \ti D_i \ti Q_i  \right) + 
\kappa \sum_{m=1}^n \sigma^4_{im} \E \left( [\ti Q_i]_{mm}^2 \right) 
\right)\ , \\ 
{\bs \psi}_9 &=& 
\frac{ 
\omega^2 \ti b_j \ti c_j^2}{n} 
\left( \E ( \tr {D}_j {Q}_j {D}_j {Q}_j ) +
\kappa \sum_{i=1}^N \sigma^4_{ij} \E \left( [ {Q}_j ]_{ii}^2 \right)
\right) \ . 
\end{eqnarray*} 
When computing the right hand side of (\ref{eq-expression-psi_i}), all 
terms of the form $\E \tr D_j Q_j D_j Q_j$ and 
$\E \tr \tilde{D}_i \tilde{Q}_i \tilde{D}_i \tilde{Q}_i$ cancel out and we
end up with Equation (\ref{eq-psi-i}). 
Step \ref{sec-prf-bias-expression-psi} is established. 

\subsection{Proof of step \ref{sec-prf-bias-(phi-w)-to-0}: 
$\| {\bs \varphi} - {\bs w} \|_\infty \to 0$} 
In order to prove \eqref{eq-(phi-w)-to-0}, we need the following facts:  
\begin{eqnarray} 
& & \leftrownorm ( \breve{A} - A )^T \rightrownorm_\infty 
\xrightarrow[n\rightarrow\infty]{} 0 ,
\label{eq-btiA-tiA-to-0} \\
& & \limsup_n \leftrownorm ( I - A )^{-1} \rightrownorm_\infty
< \infty, 
\label{eq-I-tiA-bounded} \\
& & I- \breve{A} \ \mathrm{is \ invertible \ for \ } n \ 
\mathrm{large \ enough}, 
\label{eq-I-btiA-invertible} \\
& & \limsup_n \leftrownorm ( I - \breve{A} )^{-1} \rightrownorm_\infty
< \infty \ . 
\label{eq-I-btiA-bounded}
\end{eqnarray} 
The proof of (\ref{eq-btiA-tiA-to-0}) is close to the proof of 
(\ref{eq-bA-A-to-zero}) above and is therefore omitted. The bound
(\ref{eq-I-tiA-bounded}) follows from Lemma 
\ref{lm-(I-A)regular}--(\ref{it-bound-maxrownorm}). We now prove 
(\ref{eq-I-btiA-invertible}) and (\ref{eq-I-btiA-bounded}). 
Recall that by Lemma \ref{lm-properties-A}, there exist two vectors 
$u_n = ( u_{\ell,n}) \succ 0$ and $v_n = ( v_{\ell,n}) \succ 0$ such that
$u_n = A u_n + v_n$, $\sup_n \| u_n \|_\infty < \infty$ and 
$\liminf_n \min_{\ell} (v_{\ell,n}) > 0$. 
Matrix $\breve{A}$ satisfies the equation 
$u_n = \breve{A} u_n + \breve{v}_n$ with 
$\breve{v}_n = ( \breve{v}_{\ell n} ) = v_n + (A - \breve{A} ) u_n$. 
Combining (\ref{eq-btiA-tiA-to-0}) with inequalities 
$\sup_n \| u_n \|_\infty < \infty$ and 
$\liminf_n ( \min_\ell v_{\ell n} )  > 0$, we have
$\liminf_n ( \min_\ell \breve{v}_{\ell n} )  > 0$. Therefore, Lemma
\ref{lm-(I-A)regular} applies to matrix $\breve{A}$ for $n$ large
enough which implies (\ref{eq-I-btiA-invertible}) and 
(\ref{eq-I-btiA-bounded}). \\

We are now in position to prove $\| {\bs \varphi} - {\bs w} \|_\infty \to 0$. 
Working out Eq. (\ref{eq-varphi-matrix}) and (\ref{eq-system-w}), we obtain:
$$
{\bs \varphi} = {\bs w} + (I-A)^{-1} ( \breve{A} - A ) 
{\bs \varphi} + (I-A)^{-1} ( {\bs \psi} - {\bs p} ),
$$ 
hence 
\[
\| {\bs \varphi} - {\bs w} \|_\infty  \leq 
\leftrownorm (I-A)^{-1} \rightrownorm_\infty 
\leftrownorm ( \breve{A} - A ) \rightrownorm_\infty  
\| {\bs \varphi} \|_\infty  + 
\leftrownorm (I-A)^{-1} \rightrownorm_\infty 
\| {\bs \psi} - {\bs p} \|_\infty  \ . 
\] 
Thanks to (\ref{eq-I-btiA-invertible}), we have 
${\bs \varphi} = ( I - \breve{A} )^{-1} {\bs \psi}$ for $n$ large enough.
One can check from (\ref{eq-psi-i}) that $\sup_n \| {\bs \psi} \|_\infty 
< \infty$. Therefore, by (\ref{eq-I-btiA-bounded}), we have 
$\sup_n \| {\bs \varphi} \|_\infty < \infty$. Using 
(\ref{eq-btiA-tiA-to-0}) and (\ref{eq-I-tiA-bounded}), we then have
$\leftrownorm (I-A)^{-1} \rightrownorm_\infty
\leftrownorm ( \breve{A} - A )   \rightrownorm_\infty
\| {\bs \varphi} \|_\infty \to 0$. 

It remains to prove that $\| {\bs \psi} - {\bs p} \|_\infty \to 0$. In
Step 3, it has been established that ${\bs \psi}$ is a perturbated
version of ${\bs p}$ as defined in \eqref{eq-p(omega)} in the sense of
Eq. \eqref{eq-psi-i}. Using the arguments developed in the course of
the proof of \eqref{eq-zeta_jj-x_jj}, it is a matter of routine to check $\| {\bs
  \psi} - {\bs p} \|_\infty \to 0$. Details are omitted. Hence
$$
\leftrownorm (I-A)^{-1} \rightrownorm_\infty
\| {\bs \psi} - {\bs p} \|_\infty \to 0.
$$ 
Consequently, 
$\| {\bs \varphi} - {\bs w} \|_\infty \to 0$ and Step 
\ref{sec-prf-bias-(phi-w)-to-0} is proved. 

\subsection{Proof of step \ref{sec-prf-bias-dct}: Dominated Convergence}  
In this section, constant $K'$ does not depend on $n$ neither on
$\omega$ but its value is allowed to change from line to line.  We
first prove \eqref{eq-dct-beta}. We have
\[
\left| \beta_n \right| \leq 
\| {\bs w} \|_\infty 
\leq \leftrownorm ( I - A)^{-1} \rightrownorm_\infty
\| {\bs p} \|_\infty 
\]
by \eqref{eq-system-w}. By inspecting (\ref{eq-p(omega)}) one obtains 
$\| {\bs p} \|_\infty \leq | \kappa | (N/n) 
( \frac{\sigma_{\max}^6}{\omega^4}
+ \frac{\sigma_{\max}^4}{\omega^3} ) \leq K' \omega^{-3}$. 
We need now to bound $\leftrownorm ( I - A )^{-1} \rightrownorm_\infty$ in 
terms of $\omega \in [ \rho, \infty)$. 
Lemma 
\ref{lm-(I-A)regular}--(\ref{it-bound-maxrownorm}) yield:
\[
\leftrownorm ( I - A )^{-1} \rightrownorm_\infty 
\leq 
\frac{\max_{\ell}(u_{\ell,n})}{\min_{\ell}(v_{\ell,n})}  
\] 
where 
$u_n = ( u_{\ell n} )$ and $v_n = ( v_{\ell n} )$ are the vectors given in 
the statement of Lemma \ref{lm-properties-A}. We now inspect the expressions
of $u_{\ell n}$ and $v_{\ell n}$. Eq. \eqref{eq-lower-trD2T2} yields: 
$$
\min_{\ell} ( v_{\ell,n} ) \geq 
\frac{1}{(\omega + \sigma_{\max})^2} \min_j \frac 1N \tr D_j ,
$$ 
and $\max_{\ell}(u_{\ell,n})  \leq (N \sigma_{\max}^2)(n \omega^2)^{-1} $ by 
\eqref{eq-u-v}. Gathering all these estimates, we obtain $|\beta_n|\le K' \omega^{-3}$.
and Inequality \eqref{eq-dct-beta} is proved.  

We now prove \eqref{eq-dct-phi}. We have 
\begin{equation}
\label{eq-bound-phi} 
| \frac 1n \sum_{j=1}^n {\bs \varphi}_{j} | \leq 
\| {\bs \varphi} \|_\infty 
\leq \leftrownorm ( I - \breve{A} )^{-1} \rightrownorm_\infty
\| {\bs \psi} \|_\infty 
\end{equation} 
by \eqref{eq-varphi-matrix} and \eqref{eq-I-btiA-invertible}. We know
that the right hand side is bounded as $n \to \infty$. However, not
much is known about the behaviour of the bound with respect to
$\omega$.  
Using Inequality \eqref{eq-bound-phi} and relying on the derivations
that lead to (\ref{eq-varphi-matrix}--\ref{eq-psi-i}),
one can prove that $\leftrownorm ( I - \breve{A} )^{-1}
\rightrownorm_\infty$, $\| {\bs \psi} \|_\infty$, and therefore $\|
{\bs \varphi} \|_\infty$ are bounded on the compact subsets of $[
\rho, + \infty)$. Therefore, in order to establish \eqref{eq-dct-phi},
it is sufficient to prove that $\| {\bs \varphi} \|_\infty$ is bounded
by $K'\,\omega^{-2}$ near infinity. To this end, we develop $|{\bs
  \varphi}_j(\omega)|$ as follows:
\begin{eqnarray*}
|{\bs \varphi}_j(\omega)| &=& n \ti t_j \left| \E \left( [\ti Q]_{jj} 
\left( [\ti Q]_{jj}^{-1} - \ti t_j^{-1} \right) \right) \right| \\
&=& 
n \omega \ti t_j \left|
\E \left( [\ti Q]_{jj} \left( y_j^* Q_j y_j - \frac 1n \tr D_j T \right) 
\right) \right|  \\
&\leq& 
\omega \ti t_j \E [\ti Q]_{jj}   
\left| \tr D_j \E \left( Q  - T \right) \right| 
+ 
n \omega \ti t_j \left|
\E \left( [\ti Q]_{jj} 
\left( y_j^* Q_j y_j - \frac 1n \tr D_j \E Q \right) 
\right) \right|  \\
&\stackrel{(a)}{\leq}& 
\omega \ti t_j  \E [\ti Q]_{jj}   
\left| \tr D_j \E \left( Q  - T \right) \right| 
+ 
n \omega \ti t_j \ti c_j  \left| 
\E\left( y_j^* Q_j y_j - \frac 1n \tr D_j \E Q \right)\right|
\\
& & 
\ \ \ \ \ \ \ \ \ \ \ \ \ \ \ \ \ \ \ \ 
\ \ \ \ \ \ \ \ \ \ \ \ \ \ + 
n \omega^2 \ti t_j \ti c_j \left| 
\E \left( [\ti Q]_{jj} 
\left( y_j^* Q_j y_j - \frac 1n \tr D_j \E Q \right)^2
\right) \right| \\
&\stackrel{(b)}{\leq}& 
\omega \ti t_j  \E [\ti Q]_{jj}   
\left| \tr D_j \E \left( Q  - T \right) \right| 
+ 
\omega \ti t_j \ti c_j \left| \E (\tr D_j ( Q_j -  Q )) \right| 
\\
& & 
\ \ \ \ \ \ \ \ \ \ \ \ \ \ \ \ \ \ \ \ 
\ \ \ \ \ \ \ \ \ \ \ \ \ \ + 
2 n \omega^2 \ti t_j \ti c_j
\E \left( [\ti Q]_{jj}  
\left( y_j^* Q_j y_j - \frac 1n \tr D_j \E Q_j \right)^2
\right) \\
& & 
\ \ \ \ \ \ \ \ \ \ \ \ \ \ \ \ \ \ \ \ 
\ \ \ \ \ \ \ \ \ \ \ \ \ \ + 
2 \frac{\omega^2 \ti t_j \ti c_j}{n} \E [\ti Q]_{jj}   
\left( \tr D_j \E ( Q_j  - Q)  \right)^2\ ,
\end{eqnarray*}
where $(a)$ follows from \eqref{eq-identite-tildeQX} and $(b)$, from
the fact that 
$$ 
( y_j^* Q_j y_j - \frac 1n \tr D_j \E Q )^2 \leq 2 (
y_j^* Q_j y_j - \frac 1n \tr D_j \E Q_j )^2 +
2 (\frac 1n \tr D_j \E ( Q_j - Q) )^2.
$$ 
Let ${\bs \alpha}(\omega) = n \max_{1\leq i\leq N} | t_i - \E[Q]_{ii}
|$.  Using Lemma
\ref{lm-approximations-EQ-C-B}--(\ref{superlemma-item:rank1}), we
obtain from the last inequality
\[
\| {\bs \varphi}(\omega) \|_\infty \leq 
\frac{\sigma_{\max}^2}{\omega} \bs{\alpha}(\omega) +
\frac{\sigma_{\max}^2}{\omega^2} + 
\frac{2n}{\omega} 
\E \left( y_j^* Q_j y_j - \frac 1n \tr D_j \E Q_j \right)^2
+ \frac{2 \sigma_{\max}^4}{n \omega^3} \ .
\] 
As in \eqref{eq-expression-psi9}, we have 
\[ 
\E \left( y_j^* Q_j y_j - \frac 1n \tr D_j \E Q_j \right)^2
= 
\frac{1}{n^2} 
\left( \E ( \tr {D}_j {Q}_j {D}_j {Q}_j ) +
\kappa \sum_{i=1}^N \sigma^4_{ij} \E [ {Q}_j ]_{ii}^2 
\right) 
\leq \frac{N \sigma_{\max}^4 ( 1 + | \kappa | )}{n^2 \omega^2}\ .
\] 
Therefore, 
\begin{equation}
\label{eq-varphi-small} 
\| {\bs \varphi}(\omega) \|_\infty \leq \frac{\sigma_{\max}^2}{\omega} 
\bs{\alpha}(\omega) + \frac{K'}{\omega^2} 
\end{equation} 
for $\omega \in [\rho, +\infty)$. A similar derivation yields
$\bs{\alpha}(\omega) \leq \frac{\sigma_{\max}^2}{\omega} 
\| {\bs \varphi}(\omega) \|_\infty + \frac{K'}{\omega^2}$. Plugging this
inequality into \eqref{eq-varphi-small}, we obtain 
$$
(1 - \sigma_{\max}^4 / \omega^2 ) 
\| {\bs \varphi}(\omega) \|_\infty \leq \frac{K'}{\omega^2},
$$ 
hence $\| {\bs \varphi}(\omega) \|_\infty \leq K' \omega^{-2}$ for $\omega$ large
enough. 

We have proved that $\| {\bs \varphi}(\omega) \|_\infty$ is bounded on
compact subsets of $[ \rho, \infty)$, and furthermore, that
\eqref{eq-dct-phi} is true for $\omega$ large enough. Therefore,
\eqref{eq-dct-phi} holds for every $\omega \in [ \rho, +\infty)$. Step
\ref{sec-prf-bias-dct} is proved, and so is Theorem \ref{th-bias}.

\begin{appendix}
\section{Proof of Lemma \ref{lm-approximations-EQ-C-B}} 
\label{anx-proof-lm-approximations-EQ-C-B}
\subsection*{Proof of Lemma 
\ref{lm-approximations-EQ-C-B}--(\ref{superlemma-item1})} Straightforward.

\subsection*{Proof of Lemma 
\ref{lm-approximations-EQ-C-B}-(\ref{superlemma-item2})}
\begin{proof}[Proof of (\ref{eq-approx:EQ-T})]  
From \cite[Lemmas 6.1 and 6.6]{HLN07}, we get
$$
\frac 1n \tr U \left(Q(-\rho)-T(-\rho)\right) 
\xrightarrow[n\rightarrow\infty]{}0 \quad \mathrm{a.s.} 
$$
Now since
$$
\left| \frac 1n \tr U\left( Q(-\rho) -T(-\rho)\right) \right| \le 
\|U\| \left( \|Q(-\rho)\| +\| T(-\rho)\|\right) 
\le \frac{2\|U\|}{\rho},
$$
the Dominated Convergence Theorem yields the first part of 
(\ref{eq-approx:EQ-T}). The second part is proved similarly.
\end{proof}

\begin{proof}[Proof of (\ref{eq-approx:B-T})] 
Recall from Theorem \ref{first-order}-(1) and from the mere definitions of $T$ and $B$ that matrices $T(z)$ and $B(z)$ 
can be written as
\[ 
T = 
\left( - z I + 
\frac{1}{n} \sum_{j=1}^n \frac{1}{1+\frac 1n \tr D_j T} D_j 
\right)^{-1} 
\quad \mathrm{and} \quad 
B = 
\left( - z I + 
\frac{1}{n} \sum_{j=1}^n \frac{1}{1+\frac 1n \tr D_j \E Q} D_j 
\right)^{-1}\ . 
\] 
We therefore have 
\begin{eqnarray*}
\frac 1n \tr U(B(-\rho) - T(-\rho)) &=& 
\frac 1n \tr UBT(T^{-1} - B^{-1}) \\
&=& 
\frac{1}{n^2} \tr\left( 
UBT \sum_{j=1}^n 
\frac{\frac 1n \tr D_j (\E Q - T)}
{(1+\frac 1n \tr D_j T)(1+ \frac 1n \tr D_j \E Q)}  
D_j \right) \\
&=& \frac{1}{n^2} \sum_{i=1}^N \sum_{j=1}^n x^n_{ij} \ ,
\end{eqnarray*} 
with 
$\displaystyle{
x_{ij}^n = \frac{ [ U ]_{ii} b_i t_i \sigma^2_{ij}}
{(1+\frac 1n \tr D_j T)(1+ \frac 1n \tr D_j \E Q)}  
\frac 1n \tr D_j (\E Q - T)
}$. 
It can be easily checked that 
$| x_{ij}^n | \leq 2 \sup_n(\| U \|) \sigma_{\max}^4 / \rho^3$. Furthermore,
$x_{ij}^n \to_n 0$ for every $i,j$ by (\ref{eq-approx:EQ-T}). 
It remains to apply the Dominated Convergence Theorem
to the integral with respect to Lebesgue measure on 
$[0,1]^2$ of the staircase function $f_n(x,y)$ defined as 
$f_n(i/N, j/n) = x^n_{ij}$ to deduce that $\frac 1n \tr U(B-T) \to 0$.  
This ends the proof of (\ref{eq-approx:B-T}). 
\end{proof}
In the sequel, $K$ is a constant whose value might change from line to line but which remains independent
of $n$.
\begin{proof}[Proof of (\ref{eq-approx:(Q-EQ)2})]  
We have 
\begin{eqnarray}
\tr U \left( Q - \E Q \right) &\stackrel{(a)}{=}& 
\sum_{j=1}^n \left( \E_j - \E_{j+1} \right) \tr U Q \nonumber \\
&\stackrel{(b)}{=}& 
\sum_{j=1}^n \left( \E_j - \E_{j+1} \right) \tr U \left( Q - Q_j \right)\nonumber \\
&\stackrel{(c)}{=}& 
- \sum_{j=1}^n \left( \E_j - \E_{j+1} \right) 
\frac{y_j^* Q_j U Q_j y_j}{1 + y_j^* Q_j y_j} \quad \stackrel{\triangle}{=}\quad \sum_{j=1}^n x_j. \label{decomposition-mart}
\end{eqnarray}
where $(a)$ follows from the fact that $\E_1 \tr\, UQ=\tr\, UQ$ and 
$\E_{n+1} \tr\, UQ = \E \tr\, UQ$, $(b)$ follows from the fact that 
$\E_j \tr\, UQ_j =\E_{j+1} \tr\, U Q_j$ since $Q_j$ does not depend on 
$y_j$ and $(c)$ follows from (\ref{inversion-lemma}) and from the fact that 
$\tr\, Q_j y_j y_j^* Q_j U = y_j^* Q_j U Q_j y_j$. \\ 
Now, one can easily check that $\sum_{j=1}^n x_j  \left( = \tr U(
  Q - \E Q)\right)$ is the sum of a martingale difference
sequence with respect to the increasing filtration ${\mathcal F}_n,
\ldots, {\mathcal F}_1$ since $\E_k x_j = 0$ for $k > j$. Therefore, 
$$
\E \left( \tr U
  \left( Q - \E Q \right) \right)^2 = 
  \sum_{j=1}^n \E  x_j^2.
$$ 
Write $x_j = x_{j,1} + x_{j,2}$ where:
\begin{eqnarray*}
x_{j,1} &=& - \left( \E_j - \E_{j+1} \right)
\left(\frac{y_j^* Q_j U Q_j y_j}{1 + \frac 1n \tr D_j Q_j} \right), \\
x_{j,2} &=& - \left( \E_j - \E_{j+1} \right)\left(\frac{y_j^* Q_j U Q_j y_j}{1 + y_j^* Q_j y_j}
-  \frac{y_j^* Q_j U Q_j y_j}{1 + \frac 1n \tr D_j Q_j} \right).
\end{eqnarray*}
Using the fact that $y_j$ and ${\mathcal F}_{j+1}$ are independent,
and the fact that $Q_j$ does not depend on $y_j$, one easily obtains:
$$
\E_{j+1} \left( \frac{y_j^* Q_j U Q_j y_j}{1+\frac 1n \tr D_j Q_j} \right)
= \frac 1n \tr D_j \E_{j+1} \left( \frac{Q_j U Q_j}{1+\frac 1n \tr D_j Q_j} \right).
$$
Thus $x_{j,1}$ and $x_{j,2}$ write:
\begin{eqnarray*}
x_{j,1} 
&=& 
- 
y_j^* \E_{j+1} \left( \frac{Q_j U Q_j}{1 + \frac 1n \tr D_j Q_j}  \right) y_j 
+
\frac{1}{n} \tr D_j
\E_{j+1} \left( \frac{Q_j U Q_j}{1 + \frac 1n \tr D_j Q_j}  \right) \\ 
x_{j,2} &=& 
\left( \E_j - \E_{j+1} \right)
\frac{y_j^* Q_j U Q_j y_j}{(1 + \frac 1n \tr D_j Q_j) 
(1 + y_j^* Q_j y_j)}  
 \left( y_j^* Q_j y_j - \frac 1n \tr D_j Q_j \right)  \\
 &\stackrel{\triangle}=& \left( \E_j - \E_{j+1} \right) x_{j,3} \ .
\end{eqnarray*}
Since matrix $\| D_j \E_j  \left( \frac{Q_j U Q_j}{1 + \frac 1n \tr D_j Q_j} 
\right) \| \leq K$, Lemma \ref{lm-approx-quadra}-(\ref{silverstein}) 
in conjunction with Assumption {\bf A-\ref{hypo-moments-X}} yield
$\E x_{1,j}^2 \le K n^{-1}$. 
Furthermore, we have:
$$
| x_{j,3} | \leq  
 \left| y_j^* Q_j U Q_j y_j 
 \left( y_j^* Q_j y_j - \frac 1n \tr D_j Q_j \right) \right| 
$$
since $y_j^* Q_j y_j \geq 0$ and $\frac 1n \tr D_j Q_j \geq 0$. 
Cauchy-Schwarz inequality yields:
$$
\E x_{j,3}^2 \le \left( \E ( y_j^* Q_j U Q_j y_j )^4 \right)^{\frac 12}
\left( \E \left( y_j^* Q_j y_j - \frac 1n \tr D_j Q_j \right)^4 \right)^{\frac 12} 
$$
which in turn yields $\E x_{j,3}^2 < \frac{K}n$ since 
\begin{equation}\label{controle1}
\E ( y_j^* Q_j U Q_j y_j )^4 \le K \qquad \textrm{and} \qquad 
\E \left( y_j^* Q_j y_j - \frac 1n \tr D_j Q_j \right)^4 \le \frac{K}{n^2}\ ,
\end{equation}
where the first inequality in (\ref{controle1}) 
follows from $0 \le y_j^* Q_j U Q_j y_j \le \|Q_j U Q_j\| \|y_j\|^2$ 
and from Assumption {\bf A-\ref{hypo-moments-X}}, 
and the second 
from Assumption {\bf A-\ref{hypo-moments-X}} and Lemma \ref{lm-approx-quadra}-(\ref{silverstein}).

We are now in position to conclude. 
\begin{eqnarray*}
\E x_{j,2}^2 &=& \E \left( (\E_j - \E_{j+1}) x_{j,3}\right)^2
\quad \le\quad  2 \E \left( (\E_j x_{j,3})^2  + (\E_{j+1} x_{j,3})^2 \right) 
\\ 
&\stackrel{(a)}{\le}&  2\E \left( \E_j x_{j,3}^2 + \E_{j+1} x_{j,3}^2 \right) 
\quad = \quad 
4 \E x_{j,3}^2,
\end{eqnarray*}
where $(a)$ follows from Jensen's inequality. Now,
\[
\E x_j^2 =\E (x_{j,1} + x_{j,2})^2  \le 
\left( (\E x_{j,1}^2)^{\frac 12} + (\E x_{j,2}^2)^{\frac 12}\right)^2 \le 
\frac Kn
\]
and 
$
\E (\tr U(Q-\E Q))^2 =\sum_{j=1}^n \E x_j^2 \le K.
$
Inequality (\ref{eq-approx:(Q-EQ)2}) is proved. 
\end{proof} 

\begin{proof}[Proof of (\ref{eq-approx:(Q-EQ)4})] 
We rely again on the decomposition (\ref{decomposition-mart}) and follow the 
lines of the computations in (\cite{BaiSil04}, page 580):
$$
\tr U \left( Q - \E Q \right)
=- \sum_{j=1}^n \left( \E_j - \E_{j+1} \right) 
\frac{y_j^* Q_j U Q_j y_j}{1 + y_j^* Q_j y_j}.
$$
Thus,
\begin{eqnarray*}
\lefteqn{ \E \left( \frac 1N \tr U(Q-\E Q)\right)^4
=\frac1{N^4} \E \left( \sum_{j=1}^n (\E_j -\E_{j+1})\frac{y_j^* Q_j U Q_j y_j}{1+y_j^* Q_j y_j} \right)^4}\\
&\stackrel{(a)}{\le}& \frac{K}{N^4} \E \left( 
\sum_{j=1}^n \left( (\E_j -\E_{j+1})\frac{y_j^* Q_j U Q_j y_j}{1+y_j^* Q_j y_j} \right)^2 
\right)^2
\\
&\stackrel{(b)}{\le}&  \frac{K}{N^4} N \sum_{j=1}^n \E 
\left(  (\E_j -\E_{j+1})\frac{y_j^* Q_j U Q_j y_j}{1+y_j^* Q_j y_j} \right)^4  
\\
&\le& \frac{K}{N^2} \sup_{j} 
\E\left( (\E_j -\E_{j+1})\frac{y_j^* Q_j U Q_j y_j}{1+y_j^* Q_j y_j} \right)^4 \ ,
\end{eqnarray*}
where $(a)$ follows from Burkholder's inequality and $(b)$ from the
convexity inequality $(\sum_{i=1}^n a_i)^2 \le n\sum_{i=1}^n a_i^2$.
Recall now that 
$y_j ^* Q_j y_j \ge 0$ and $\|Q_j(-\rho)\|\le 1/\rho$.
Standard computations yield: 
$$
\E\left( (\E_j -\E_{j+1})\frac{y_j^* Q_j U Q_j y_j}{1+y_j^* Q_j y_j} \right)^4
\le K\E \left( y_j^* Q_j U Q_j y_j\right)^4 
\le \frac{K \|U\|^4}{\rho^8}\ \E \|y_j\|^4
$$
which is uniformly bounded by Assumptions 
{\bf A-\ref{hypo-moments-X}} and {\bf A-\ref{hypo-variance-field}}. 
Therefore, (\ref{eq-approx:(Q-EQ)4}) is proved.
\end{proof} 

\subsection*{Proof of Lemma 
\ref{lm-approximations-EQ-C-B}--(\ref{superlemma-item:rank1})}  

Developing the difference $Q-Q_j$ with the help of  
(\ref{inversion-lemma}), we obtain:
\begin{eqnarray*}
|\tr M(Q-Q_j) | &=& \left|  \tr M 
\left( \frac{Q_j y_j y_j^* Q_j}{1+y_j^* Q_j y_j} \right) \right|\\
&=& \frac{\left| y_j^* Q_j M Q_j y_j \right|}{1 + y_j^* Q_j y_j} 
\leq \| M \| 
\frac{\| Q_j y_j \|^2}{1 + y_j^* Q_j y_j} \ .
\end{eqnarray*}
Consider a spectral representation of $Y^j {Y^j}^*$, \emph{i.e.},
$Y^j {Y^j}^* = \sum_{i=1}^N \lambda_i e_i e_i^*$. We have 
$$
\| Q_j y_j \|^2 = 
\sum_{i=1}^N \frac{\left| e_i^* y_j \right|^2}
{\left(\lambda_i + \rho\right)^2} 
\ \ \mathrm{and} \ \ 
y_j^* Q_j y_j = \sum_{i=1}^N 
\frac{| e_i^* y_j |^2}{\lambda_i + \rho} 
\geq \rho \sum_{i=1}^N 
\frac{| e_i^* y_j |^2}{\left(\lambda_i + \rho\right)^2} \ ,
$$
hence the result. Inequality (\ref{superlemma-item:rank1}) is proved. 

\section{Proof of Formula \eqref{eq-tr(T-Q)=tr(T-Q)tilde}}  
\label{anx-prf-Q-Qtilde} 
Recalling that $Q(z) = ( Y Y^* -z I_N )^{-1}$ and 
$\ti Q(z) = ( Y^* Y -z I_n )^{-1}$, it is easy to show that 
$\tr(Q) - \tr(\ti Q) = (n-N) / z$. We shall show now that 
$\tr(T) - \tr(\ti T) = (n-N) / z$. Formula \eqref{eq-tr(T-Q)=tr(T-Q)tilde}
is obtained by combining these two equations. \\
Equations \eqref{eq-approx-determinist-complete} in the statement of Lemma
\ref{th-deter-approx-details} can be rewritten as
\[
t_i +  \frac {t_i}n \sum_{j=1}^n \sigma^2_{ij} \ti t_j = -\frac 1z
\ \mathrm{for} \ 1\le i \le N, \quad \quad 
\ti t_j +  \frac{\ti t_j}n \sum_{i=1}^N \sigma^2_{ij} t_i = -\frac 1z
\ \mathrm{for} \ 1 \le j \le n. 
\]
By summing the first $N$ equations over $i$ and the next $n$ equations over
$j$ and by eliminating the term 
$\frac 1n \sum_{i=1}^N \sum_{j=1}^n \sigma^2_{ij} t_i \ti t_j$, we 
obtain $\sum_i t_i - \sum_j \ti t_j = (n-N) / z$, which is the desired result.
Equation \eqref{eq-tr(T-Q)=tr(T-Q)tilde} is proved.

\end{appendix}

\bibliography{math}

\noindent {\sc Walid Hachem} and {\sc Jamal Najim},\\ 
CNRS, T\'el\'ecom Paris\\ 
46, rue Barrault, 75013 Paris, France.\\
e-mail: \{hachem, najim\}@enst.fr\\
\\
\noindent {\sc Philippe Loubaton},\\
IGM LabInfo, UMR 8049, Institut Gaspard Monge,\\
Universit\'e de Marne La Vall\'ee, France.\\
5, Bd Descartes, Champs sur Marne, \\
77454 Marne La Vall\'ee Cedex 2, France.\\
e-mail: loubaton@univ-mlv.fr\\
\\
\noindent

\end{document}